\documentclass[reqno,11pt,a4paper]{amsart}
\usepackage{kotex}
\usepackage{amssymb,amsmath,amsthm}
\usepackage{amssymb}
\usepackage{graphicx}
\usepackage{mathrsfs}
\usepackage{amssymb}
\usepackage{amsfonts,bm,bbm}
\usepackage{amsthm}
\usepackage[colorlinks,linkcolor=blue,anchorcolor=blue,citecolor=blue,urlcolor=purple]{hyperref}
\usepackage{fullpage}
\usepackage{color}
\usepackage{enumitem}
\usepackage{underscore}
\usepackage{cite}
\usepackage{doi}

\DeclareFontFamily{U}{mathx}{}
\DeclareFontShape{U}{mathx}{m}{n}{<-> mathx10}{}
\DeclareSymbolFont{mathx}{U}{mathx}{m}{n}
\DeclareMathAccent{\widecheck}{0}{mathx}{"71}

\usepackage[open,openlevel=2]{bookmark}
\numberwithin{equation}{section}
\newtheorem{Thm}{Theorem}[section]

\newtheorem{Lem}[Thm]{Lemma}
\newtheorem{Coro}[Thm]{Corollary}

\theoremstyle{remark}
\newtheorem{remark}{Remark}[section]

\newtheorem{theorem}{Theorem}[section]

\newtheorem{lemma}[theorem]{Lemma}

\newcommand{\R}{\mathbb{R}}
\newcommand{\T}{\mathbb{T}}

\newcommand{\ve}{\varepsilon}
\newcommand{\vt}{\vartheta}

\renewcommand{\S}{\mathbb{S}}

\let\p=\partial

\newcommand{\vertiii}[1]{{\left\vert\kern-0.25ex\left\vert\kern-0.25ex\left\vert #1 \right\vert\kern-0.25ex\right\vert\kern-0.25ex\right\vert}}

\usepackage{tikz,pgfplots}
\usepackage{tikz-cd}
\usepackage{tikz-3dplot}
\usepackage{amsmath,amssymb,amsthm,amsfonts,bm}
\usepackage{color}
\usepackage{tcolorbox}
\usepackage{pdfpages}
\usepackage{textcomp}
\usetikzlibrary{patterns}
\usetikzlibrary{shapes.geometric}
\usetikzlibrary{arrows.meta,arrows}
\usetikzlibrary{decorations.pathreplacing,decorations.pathmorphing}
\usetikzlibrary{hobby}
\usetikzlibrary{math}
\usepgfplotslibrary{fillbetween}
\usepackage{graphicx}
\usepackage{multicol}
\usepackage{float}
\usepackage{subfigure}

\allowdisplaybreaks

\makeatletter
\def\@tocline#1#2#3#4#5#6#7{\relax
	\ifnum #1>\c@tocdepth 
	\else
	\par \addpenalty\@secpenalty\addvspace{#2}%
	\begingroup \hyphenpenalty\@M
	\@ifempty{#4}{%
		\@tempdima\csname r@tocindent\number#1\endcsname\relax
	}{%
		\@tempdima#4\relax
	}%
	\parindent\z@ \leftskip#3\relax \advance\leftskip\@tempdima\relax
	\rightskip\@pnumwidth plus4em \parfillskip-\@pnumwidth
	#5\leavevmode\hskip-\@tempdima
	\ifcase #1
	\or\or \hskip 1em \or \hskip 2em \else \hskip 3em \fi%
	#6\nobreak\relax
	\hfill\hbox to\@pnumwidth{\@tocpagenum{#7}}\par
	\nobreak
	\endgroup
	\fi}
\makeatother

\begin{document}

\title[Asymptotic behavior of large-amplitude solutions to the Boltzmann equation with soft interactions in $L^p_v L^\infty_x$ spaces ]
{Asymptotic behavior of large-amplitude solutions to the Boltzmann equation with soft interactions in $L^p_v L^\infty_x$ spaces }

\author{Jong-in Kim}

\address{Department of Mathematics, Pohang University of Science and Technology, South Korea }
\email{kimjim@postech.ac.kr}

\author{Gyounghun Ko}
\address{Academy of Mathematics and Systems Sciences, Chinese Academy of Sciences, Beijing, 100190, China}
\email{gyeonghungo@amss.ac.cn}

\begin{abstract}	
In this paper, we study the global well-posedness of the Boltzmann equation within the $L_{v}^{p}L_{x}^{\infty}$ framework for soft potential models with angular cutoff in a periodic box $\T^3$. By using a time-involved weight function, inspired by the works of \cite{Liu2017ARMA, Duan2013MMMAS, Ko2022JDE}, we overcome the absence of a spectral gap. An analytical difficulty in the $L_v^p L_x^\infty$ setting is that the standard arguments used in \cite{Ko2022JDE,Li2022SIMA} for the nonlinear loss term are no longer applicable when dealing with time integration involving the collision frequency. To resolve this, we introduce a modified solution operator. Furthermore, we control the nonlinear gain term by deriving pointwise estimates bounded by $L_v^p$ and $L_v^\ell$ (for some $\ell <p$) norms. Thanks to the smallness of the initial relative entropy and Gr\"{o}nwall's inequality, we prove the global existence of unique solutions for large-amplitude initial data and obtain a sub-exponential convergence rate toward equilibrium.
\end{abstract}

\maketitle
\tableofcontents

\vspace{0.2cm}
\section{Introduction}
\subsection{Boltzmann equation}

The Boltzmann equation is one of the fundamental kinetic equations which describe the statistical evolution of dilute gases through the particle distribution function, taking into account binary collisions between particles. The governing equation is given by
\begin{align} \label{Boltzmanneq}
	\partial_t F+v\cdot\nabla_xF =Q(F,F) , \quad (t,x,v) \in [0,\infty) \times \T^3 \times \R^3
\end{align}
with initial data
\begin{align}\label{initialdata}
	F(0,x,v) = F_0(x,v), \quad (x,v) \in \T^3 \times \R^3.
\end{align}
Here, $F = F(t,x,v)$ represents the density distribution function of gas particles with position $x \in \T^3$ and velocity $v \in \R^3$ at time $t\ge 0$. The collision operator $Q$ has the nonlocal bilinear form
\begin{align} \label{collisionoperator}
	Q(F_1,F_2) = \int_{\R^3} \int_{\S^2} B(v-u,\omega) \left[F_1(v')F_2(u')-F_1(v)F_2(u) \right]d\omega du,
\end{align}
where the post-collision velocity pair $(v',u')$ and the pre-collision velocity pair $(v,u)$ satisfy the relation
\begin{align} \label{omegarep}
	v' = v- [(v-u)\cdot \omega]\omega, \ u' = u+[(v-u)\cdot \omega]\omega
\end{align}
with $\omega \in \S^2$, which follows from the conservation of momentum and  energy of the two particles before and after collision:
\begin{align} \label{momentumenergyconserv}
	v+u=v'+u', \quad       |v|^2+|u|^2=|v'|^2+|u'|^2.
\end{align}
The collision kernel $B=B(v-u,\omega)$ consists of the relative velocity $|v-u|$ and the angular part $\cos \theta := \frac{(v-u) \cdot \omega}{|v-u|}$. In what follows, we assume that it has the following form:
\begin{align*}
	B(v-u,\omega) = b(\cos \theta)|v-u|^\gamma, \quad 0\le b(\cos \theta ) \le C_b,
\end{align*}
where $-3 < \gamma <0$ and $0\le \theta \le \pi$, which represents soft potentials with angular cutoff. Under this assumption, we can write the collision operator $Q$ as
\begin{align*}
	Q(F_1,F_2) &= \int_{\R^3} \int_{\S^2} B(v-u,\omega) F_1(v')F_2(u')d\omega du -\int_{\R^3} \int_{\S^2} B(v-u,\omega)F_1(v)F_2(u) d\omega du\\
	& =: Q^+(F_1,F_2)-Q^-(F_1,F_2),
\end{align*}
where $Q^+(F_1,F_2)$ and $Q^-(F_1,F_2)$ mean the gain term and loss term, respectively. 
A fundamental structural property of the Boltzmann equation is the H-theorem, which asserts that the entropy functional
\begin{align*}
\int_{\T^3} \int_{\R^3} F \ln F \, dv dx
\end{align*}
is non-increasing along solutions of \eqref{Boltzmanneq}. Moreover, the entropy dissipation vanishes if and only if the distribution function is a local Maxwellian $\mathcal{M}$ of the form 
\begin{align*}
	\mathcal{M}(t,x,v)
	:= \frac{\rho(t,x)}{(2\pi T(t,x))^{3/2}}
	\exp\!\left(-\frac{|v-u(t,x)|^2}{2T(t,x)}\right),
\end{align*}
where $\rho(t,x)$, $u(t,x)$, and $T(t,x)$ denote the local density, bulk velocity, and temperature, respectively. Since local Maxwellians characterize equilibrium states, the study of the asymptotic behavior naturally reduces to understanding the relaxation of solutions toward a Maxwellian. Accordingly, we look for convergence toward the normalized global Maxwellian $\mu$ 
\begin{align*}
	\mu(v):= \frac{1}{(2\pi)^{3/2}}e^{-\frac{|v|^2}{2}}.
\end{align*}

In the periodic box $\T^3$, the mass, momentum, and energy are conserved:
\begin{align*}
	\int_{\T^3 \times \R^3} F(t,x,v)\begin{pmatrix}
			1 \\ v \\ |v|^2 
		\end{pmatrix}dxdv =\int_{\T^3 \times \R^3}F_0(x,v)\begin{pmatrix}
			1 \\ v \\ |v|^2 
		\end{pmatrix}dxdv ,\quad t>0
\end{align*}
for any solution to the equation \eqref{Boltzmanneq} with initial condition \eqref{initialdata}. 

\bigskip

\subsection{Notations}

\begin{enumerate}
\item We denote the Japanese bracket $\langle v \rangle = (1+|v|^2)^{1/2}$. 
\item As a convention, we denote the following function spaces for $p \in [1,\infty]$,
\begin{align*}
	L^p_{x,v}=L^{p}(\T^3 \times \mathbb{R}^3_v), \quad L^p_x=L^p(\T^3), \quad L^p_v=L^p(\mathbb{R}^3_v).
\end{align*}
\item  We define the mixed function space $L^p_vL^\infty_x$ with the following norm:
\begin{align*}
	\|f\|_{L^p_vL^\infty_x}:= \left(\int_{\R^3} \Big|\sup_{x\in \T^3}|f(x,v)|\Big|^p dv \right)^{1/p}.
\end{align*} 
\item For a positive Lebesgue measurable function $m$ on $\R^3$, we define the weighted function space $L^p_{x,v}(m)$ given by the norm
\begin{align*}
	\|f\|_{L^p_{x,v}(m)} := \|mf\|_{L^p_{x,v}}.
\end{align*}

\item For a positive Lebesgue measurable function $m$ on $\R^3$, we define the weighted function space $L^p_vL^\infty_x(m)$ given by the norm
\begin{align*} 
	\|f\|_{L^p_vL^\infty_x(m)} := \|mf\|_{L^p_vL^\infty_x}= \left(\int_{\R^3} \left|\sup_{x\in \T^3}|m(v)f(x,v)|\right|^p dv \right)^{1/p}.
\end{align*}
\item We set $(f,g)_{L^2_v} = \int_{\mathbb{R}^3} f(v)g(v)dv$ the inner product in $L^2(\mathbb{R}^3)$\\
\item $p'$ is the exponent conjugate of $p$.\\
\item If not specifically mentioned, $C_a$ or $C(a)$ is the generic positive constant depending on $a$, while $C_0, C_1, C_2, \cdots$ denote some specific positive constants.
\end{enumerate}

\subsection{Perturbation framework near Maxwellian}
We apply the perturbation near Maxwellian $F(t,x,v) = \mu(v) + \mu^{1/2}(v) f(t,x,v)$ to the Boltzmann equation \eqref{Boltzmanneq}, which becomes
\begin{align} \label{FPBER}
	\partial_t f+v \cdot \nabla_x f +Lf = \Gamma(f,f).
\end{align} 
Here, the operator $L$ is a linear operator of the form
\begin{align*}
	Lf = -\mu^{-1/2} \left\{Q(\mu, \mu^{1/2}f)+Q(\mu^{1/2}f,\mu)\right\} = \nu(v) f -Kf,
\end{align*}
with the collision frequency $\nu(v) = \int_{\R^3}\int_{\S^2} B(v-u,\omega)\mu(u) d\omega du \sim (1+|v|)^\gamma$, and the operator $K=K_2-K_1$ is defined by
\begin{align*} 
	(K_1f)(t,x,v) &= \int_{\R^3} \int_{\S^2} B(v-u,\omega) \mu^{1/2}(v)\mu^{1/2}(u)f(t,x,u)d\omega du,\\
	(K_2f)(t,x,v) &= \int_{\R^3} \int_{\S^2} B(v-u,\omega) \mu^{1/2}(u)\mu^{1/2}(u')f(t,x,v')d\omega du\nonumber\\ 
	& \quad + \int_{\R^3} \int_{\S^2} B(v-u,\omega) \mu^{1/2}(u)\mu^{1/2}(v')f(t,x,u')d\omega du.
\end{align*}
It is well-known that the operator $L$ has a kernel
\begin{align*}
	\text{Ker}(L) = \text{span}\left\{\mu^{1/2},v_1\mu^{1/2},v_2\mu^{1/2},v_3\mu^{1/2}, \frac{|v|^2-3}{\sqrt{6}}\mu^{1/2} \right\}.
\end{align*}
The nonlinear term $\Gamma(f,f) = \Gamma^+(f,f) - \Gamma^-(f,f)$ is given by
\begin{align*}
	\Gamma^+(f,f) = \mu^{-1/2}Q^+(\mu^{1/2}f,\mu^{1/2}f) , \quad \Gamma^-(f,f)=\mu^{-1/2}Q^-(\mu^{1/2}f,\mu^{1/2}f).
\end{align*}
Denoting
\begin{align*} 
	R(f)(t,x,v) = \int_{\R^3} \int_{\S^2} B(v-u,\omega) \left[\mu(u)+\mu^{1/2}(u)f(t,x,u)\right]d\omega du,
\end{align*}
the equation \eqref{FPBER} can be rewritten as
\begin{align*}
	\partial_t f +v\cdot \nabla_x f + R(f)f =Kf+ \Gamma^+(f,f).
\end{align*}
By imposing initial data $F_0$ such that
\begin{align*}
	\int_{\T^3 \times \R^3} F_0(x,v)\begin{pmatrix}
			1 \\ v \\ |v|^2 
		\end{pmatrix}dxdv =\int_{\T^3 \times \R^3}\mu(v)\begin{pmatrix}
			1 \\ v \\ |v|^2 
		\end{pmatrix}dxdv,
\end{align*}
we may assume that
\begin{align*}
	\int_{\T^3 \times \R^3} \sqrt{\mu(v)} f_0(x,v) \begin{pmatrix}
			1 \\ v \\ |v|^2 
		\end{pmatrix}dvdx = 0.
\end{align*}

\bigskip

\subsection{Main results}
We introduce the time-involved weight function $w=w_{q,\vt,\beta}=w_{q,\vt,\beta}(t,v)$, used in \cite{Liu2017ARMA, Duan2013MMMAS, Ko2022JDE}, given by
\begin{align} \label{weight}
	w_{q,\vt,\beta}(t,v) = (1+|v|^2)^{\beta/2}\exp\left\{\frac{q}{8}\left(1+\frac{1}{(1+t)^{\vt}}\right)|v|^2\right\},
\end{align}
where $\beta \ge 0$, $0<q<1$, and $0\le \vt <-\frac{2}{\gamma}$.
By considering a perturbation of the form $\mu+\mu^{1/2}f$ near the Maxwellian $\mu$ and setting $h(t,x,v)= w_{q,\vt,\beta}(t,v) f(t,x,v)$, the Boltzmann equation \eqref{Boltzmanneq} can be written as the perturbed Boltzmann equation for $h$ with $R(f)$:
\begin{align} \label{WPBE2}
	\partial_t h + v\cdot \nabla_x h + R(f)h = K_wh +w\Gamma^+\left(f,f\right),
\end{align}
with initial data $h_0$, where the weighted operator $K_w$ is defined by
\begin{align*} \label{woperK}
	K_w h := wK\left(\frac{h}{w}\right).
\end{align*}
Applying Duhamel principle to the equation \eqref{WPBE2}, the mild form for $h$ is
\begin{equation*} \label{mildsolh}
	\begin{aligned}
	h(t,x,v) &= G_v(t,0)h_0(x-tv,v)+\int_0^t G_v(t,s)K_wh(s,x-v(t-s),v)ds\\
	& \quad + \int_0^t G_v(t,s)w\Gamma^+\left(f, f\right)(s,x-v(t-s),v)ds,
\end{aligned}
	\end{equation*}
where $G_v(t,s)$ is the solution operator for the equation
\begin{align*}
	\partial_t h +v\cdot \nabla_x h  +R(f)h=0,
\end{align*}
and it can be written as
\begin{align*}
	G_v(t,s) = e^{-\int_s^t R(f) (\tau,x-v(t-\tau),v)d\tau}.
\end{align*}
Before introducing our main goal, we need to establish the small perturbation problem to the Boltzmann equation \eqref{Boltzmanneq} with the smallness of initial relative entropy $\mathcal{E}(F_0)$ defined by
\begin{align*}
	\mathcal{E}(F)= \int_{\T^3 \times \mathbb{R}^3} \left(\frac{F}{\mu}\ln\frac{F}{\mu}-\frac{F}{\mu}+1\right) \mu dvdx.
\end{align*}

\bigskip

\begin{Thm} [Small perturbation problem] \label{smallmain}
	Let $p$ and $\beta$ satisfy the condition
	\begin{align*}
		\begin{cases}
			p>13, \quad \beta >\frac{9}{2} \quad &\text{if } -1 \le \gamma < 0,\\
			p>\frac{3 \sqrt{8\gamma^2+9}+3-2\gamma}{-\gamma^2-3\gamma}, \quad \beta>5 \quad &\text{if } -3 < \gamma < -1.
		\end{cases}
	\end{align*}	
	Assume that $F_0(x,v) = \mu(v) + \mu^{1/2}(v) f_0(x,v) \ge 0$ satisfying \begin{align*}
		\iint_{\T^3 \times \R^3} \sqrt{\mu(v)} f_0(x,v) \begin{pmatrix}
			1 \\ v \\ |v|^2 
		\end{pmatrix}dvdx = 0 .
	\end{align*} 
 Then there is  $\eta_0  \ll  1$ so that there exists a constant $\ve_1 = \ve_1(\eta_0) >0$, depending only on $\eta_0$, such that if
	\begin{align*}
		\|w_{q,\vt,\beta}f_0\|_{L^p_vL^\infty_x} \le \eta_0, \quad \mathcal{E}(F_0)\le \ve_1,
	\end{align*}
	then there exists a unique solution $F(t,x,v) =\mu(v) + \mu^{1/2}(v) f(t,x,v) \ge 0 $ to the Boltzmann equation \eqref{Boltzmanneq} with initial data $F(0,x,v)=F_0(x,v)$. Moreover, there exist $C_{\infty}>0$ and $ \lambda_\infty>0$ such that
	\begin{align*}
		\|w_{q,\vt,\beta}f(t)\|_{L^p_v L^\infty_x} \le C_{\infty} e^{-\lambda_\infty (1+t)^\rho}\|w_{q,\vt,\beta}f_0\|_{L^p_v L^\infty_x}
	\end{align*}
	for all $t \ge 0$, where $\rho -1 =\frac{(1+\vt)\gamma}{2-\gamma}$.
\end{Thm}

\begin{remark}
	Note that $\eta_0$ depends only on $p$, $\gamma$, $\rho$, $\vt$, and $\beta$.
\end{remark}

\begin{remark}
	In Theorem \ref{smallmain}, the condition on $p$ is specified in \eqref{pcondition1} and \eqref{pcondition2} (See Corollary \ref{smalllp}.) 
\end{remark}

\begin{remark}
	In Theorem \ref{smallmain}, the conditions on $\beta$ is determined by \eqref{betacondition1} and \eqref{betacondition2}. The conditions \eqref{betacondition1} and \eqref{betacondition2} essentially stem from the application of the pointwise estimate in Lemma \ref{pointwiseGamma+estimate} for the gain term $\Gamma^+$ in the proof of Lemma \ref{exp.h.esti}. 
\end{remark}

\bigskip

The following main result of this paper represents the global well-posedness of solutions to the Boltzmann equation \eqref{Boltzmanneq} with initial data of large amplitude in the weighted $L^p_vL^\infty_x$ space, under a smallness assumption on initial relative entropy. 

\bigskip

\begin{Thm} [Large amplitude problem] \label{largemain} 
	Let $p$ and $\beta$ satisfy the condition
	\begin{align*}
		\begin{cases}
			p>13, \quad \beta >\frac{9}{2} \quad &\text{if } -1 \le \gamma < 0,\\
			p>\frac{3 \sqrt{8\gamma^2+9}+3-2\gamma}{-\gamma^2-3\gamma}, \quad \beta>5 \quad &\text{if } -3 < \gamma < -1.
		\end{cases}
	\end{align*}
	Assume that $F_0(x,v) = \mu(v) + \sqrt{\mu(v)} f_0(x,v)\geq 0$ satisfying 
	\begin{align*}
		\iint_{\T^3 \times \R^3} \sqrt{\mu(v)} f_0(x,v) \begin{pmatrix}
			1 \\ v \\ |v|^2 
		\end{pmatrix}dvdx = 0 .
	\end{align*} 
	Then, for any $M_0\ge 1$, there exists a constant $\ve_0>0$, depending only on $\eta_0$ and $M_0$, such that if $f_0$ satisfies
	\begin{align*}
		\Vert w_{q,\vt,\beta}f_0 \Vert_{L^p_v L^\infty_x} \leq M_0, \quad \mathcal{E}(F_0) \leq \ve_0,
	\end{align*}
	there exists a unique global solution $F(t,x,v) = \mu(v) + \sqrt{\mu(v)} f(t,x,v)\geq 0$ to the Boltzmann equation \eqref{Boltzmanneq} with initial data $F_0$ on $\R^+ \times \T^3 \times \R^3$. Moreover, $f$ satisfies 
	\begin{align*}
			\Vert w_{q,\vt,\beta}f(t) \Vert_{L^p_v L^\infty_x} \leq C_1 e^{-\lambda_0 (1+t)^\rho} \Vert w_{q,\vt,\beta}f_0 \Vert_{L^p_v L^\infty_x},  \quad \forall t\geq0,
	\end{align*}
	for some constants $C_1>0$ and $\lambda_0>0$, where $\rho-1 = \frac{(1+\vartheta)\gamma}{2-\gamma}$.
\end{Thm}

\bigskip

\begin{remark}
	In Theorem \ref{largemain}, the conditions on $p$ and $\beta$ are determined for similar conditions as in Theorem \ref{smallmain}.	
\end{remark}

\bigskip

In this paper, we study the global well-posedness of the Boltzmann equation in the $L^p_vL^\infty_x$ framework under the assumption of soft intermolecular interactions. We review some known results that are closely related to our work.\\
\indent In the seminal work \cite{Diperna1989}, the authors constructed the global-in-time existence of renormalized solutions to the Boltzmann equation with general nonnegative initial data having finite physical quantities on the whole space. Later, Hamdache \cite{Hamdache1992} extended this result to general boundary conditions. However, the uniqueness issue for these solutions remains open. On the other hand, the convergence to equilibrium is one of the topics of significant interest in the study for the Boltzmann equation. In this regard, Desvillettes-Villani \cite{Desvillettes2005} proved that under the assumption of the existence of a global solution satisfying some a priori high-order Sobolev bounds and a Gaussian lower bound, the solution converges almost exponentially to the global equilibrium. Subsequently, Gualdani-Mischler-Mouhot \cite{Gualdani2017} further improved the result for the hard sphere model by establishing the exponential convergence to equilibrium under suitable high-order assumptions. To describe this literature in more detail, the authors only covered the global well-posedness on hard sphere model ($\gamma =1$) in $L^1_vL^\infty_x(\langle v \rangle^\beta)$ because of the lack of spectral gap for soft potentials. By contrast, our work extends these results to soft potentials in unbounded function spaces, $L^p_vL^\infty_x(\langle v \rangle^\beta \mu^{-1/2})$, although we adapt exponential tail models. We also obtain the asymptotic behavior for large amplitude problem without the high-order a priori assumption. \\
\indent At the same time, the well-posedness and large-time behavior for weak solutions have been studied for the perturbation framework near Maxwellian. Based on the idea of Vidav \cite{Vidav1970}, Ukai \cite{Ukai1974} constructed the global-in-time solutions for the Boltzmann equation in a periodic box when initial data $f_0 = \frac{F_0-\mu}{\sqrt{\mu}}$ are sufficiently small in some high-order Sobolev spaces. Following this work, such research in high-order Sobolev spaces was developed by Guo \cite{Guo2002, Guo2003a, Guo2012}. Especially, we mention the work \cite{Strain2008} by Strain-Guo on soft potentials. \\
\indent Unfortunately, for general bounded domains, high-order regularity of solutions to the Boltzmann equation is not guaranteed in general. These phenomena have been studied in \cite{Guo2016, Kim2011}. To overcome these difficulties, studies on low-regularity solutions have been carried out. Guo \cite{Guo2010} introduced the $L^2$-$L^\infty$ bootstrap argument to treat the global existence, uniqueness, and asymptotic behavior of low-regularity solutions in weighted $L^\infty_{x,v}$ spaces to the Boltzmann equation under physical boundary conditions. Subsequently, Kim-Lee \cite{Kim2017} extended this result to general $C^3$ uniformly convex bounded domains, and Ko-Kim-Lee \cite{Ko2023} improved it to certain non-convex domains. The literature mentioned above is limited to hard potential models. Additionally, initial-boundary value problems with polynomial tails were studied in \cite{Briant2017, Briant2016}. \\
\indent On the other hand, for soft potential models, the lack of the spectral gap of the linearized Boltzmann operator $L$ occurs in contrast to  the hard potentials. In other words, the operator $\nu(v)$ has no positive lower bound over large velocities for soft potentials. To resolve such difficulties, Caflisch \cite{Caflisch1980, Caflisch1980a} obtained sub-exponential decay in time by losing part of the exponential velocity weight. Instead of sacrificing some velocity weights, Liu-Yang \cite{Liu2017ARMA} introduced time-involved weight functions for boundary problems, and Deng-Duan \cite{Deng2023CMP} used the weight function depending on spatial and velocity variables. Here, the latter even achieved the exponential time-decaying rate. However, the aforementioned studies are restricted to small perturbation regimes.\\
\indent To go beyond such restrictions, low-regularity theory for the Boltzmann equation with cutoff has been extended to accommodate sufficiently large amplitude initial data with a smallness condition in $L^p$ class. As a first step, Duan-Huang-Wang-Yang \cite{Duan2017a} demonstrated the global-in-time existence and uniqueness of solutions to the Boltzmann equation for large amplitude initial data with a smallness condition on initial relative entropy and the $L^1_xL^\infty_v$ norm, both in a periodic box and in the whole space. Subsequently, in Duan-Wang \cite{Duan2019b}, the authors only consider a smallness condition for initial relative entropy and extended the aforementioned result to the boundary problem, concerning diffuse reflection boundary conditions. In \cite{Jong2025arXiv}, the classical weighted $L^\infty_{x,v}$ space was generalized to unbounded function classes, weighted $L^p_vL^\infty_x$ spaces. In our paper, we extend the result for hard potentials in \cite{Jong2025arXiv} to soft potential models by introducing the time-involved weight function \eqref{weight}. As for soft potential cases, we mention the following literature. Duan-Huang-Wang-Zhang \cite{Duan2019} first overcame the difficulty caused by degenerated spectral gap thanks to advantage of diffuse reflection boundary conditions, and derived the sub-exponential time decay for solutions by using velocity weight function. Next, Ko-Lee-Park \cite{Ko2022JDE} is inspired by \cite{Liu2017ARMA} and extended the result for small perturbation regime to the large amplitude problem on a periodic box. Following the paper \cite{Duan2017a}, Li \cite{Li2022SIMA} established the global-in-time existence and uniqueness for solutions in $L^p_vL^\infty_{[0,T]}L^\infty_x$ space, but did not obtain the result for convergence to equilibrium. Later, we will deal with the paper in more detail. We briefly mention \cite{Ko2025JDE, Wang2019JDE, Ko2023JSP, Cao2022JFA, Jiang2025arXiv} for related literature on large amplitude problems. \\

\indent In this paper, the main difficulty lies in controlling the nonlinear term $\Gamma(f,f)$. In particular, we address the velocity growth arising from the loss term $\Gamma^-(f,f)$, which behaves like $\nu(v)$. In our approach to solving our problem, we need to handle the following time integration for the loss term:
\begin{align} \label{maindifficulty1}
	\int_0^t e^{-\nu(v)(t-s)}w\Gamma^-(f,f)(s,x-v(t-s),v)ds.
\end{align}
As in \cite{Ko2022JDE} and \cite{Li2022SIMA}, if we work in norm of $X=L^\infty_{x,v}\ \text{or}\ L^p_vL^\infty_{[0,T]}L^\infty_x$, the norm of the term \eqref{maindifficulty1} can be controlled by 
\begin{align*}
	&\left\|\int_0^t e^{-\nu(v)(t-s)}w\Gamma^-(f,f)(s,x-v(t-s),v)ds \right\|_{X}\\
	&\quad  \lesssim \left\| \nu(v)^{-1}w\Gamma^-(f,f)(v)\int_0^t e^{-\nu(v)(t-s)}\nu(v) ds \right\|_{X}\\
	& \quad \lesssim  \left\| \nu(v)^{-1}w\Gamma^-(f,f)(v) \right\|_{X}\\
	& \quad \lesssim  \left\|wf \right\|_{X}^2.
\end{align*}
In contrast, the above argument is not available in $L^p_vL^\infty_x$ spaces due to the factor $\nu(v)$ in the loss term $\Gamma^-$:
\begin{align*}
	\left\|\int_0^t e^{-\nu(v)(t-s)}w\Gamma^-(f,f)(s)ds \right\|_{L^p_vL^\infty_x}
	& \lesssim  \left( \int_{\R^3} \sup_{x\in \T^3} \left| \int_0^t e^{-\nu(v)(t-s)}\nu(v) |wf(s)|\|wf(s)\|_{L^p_vL^\infty_x} ds\right|^p dv\right)^{1/p}	\\
	&  \nleq  \int_0^t e^{-\nu(v)(t-s)} \left\|wf(s) \right\|_{L^p_vL^\infty_x}^2ds.
\end{align*}
To resolve this issue, we introduce the term 
\begin{align*}
	R(f)(t,x,v) =\int_{\R^3} \int_{\S^2} B(v-u,\omega) \left[\mu(u)+\mu^{1/2}(u)f(t,x,u)\right]d\omega du.
\end{align*}
By combining  the a priori assumption \eqref{apriorismall} with a smallness condition on initial relative entropy, we can deduce a lower bound for $R(f)$. (See Lemma \ref{Rfestimatesmall}.) This leads to a time-decay factor, thereby ensuring convergence to equilibrium.\\
\indent A central contribution in this work is the derivation of a pointwise estimate for the gain term $\Gamma^+(f,g)$ with the weight function $w_{q,\vt,\beta}$, which is bounded in terms of $L^p_v$ and $L^q_v$ norms with $q<p$. A significant challenge arises from the singularity in relative velocity part $|v-u|^\gamma$, $-3<\gamma<0$, within $L^p_v$ framework. Consequently, controlling this singularity requires delicate analysis to obtain the desired bound for $\Gamma^+$. (See subsection \ref{Estimates on nonlinear operators} for details.)\\
\indent We introduce a bridge mechanism that connects the small perturbation regime and the large-amplitude regime. By imposing a smallness condition on initial relative entropy, we can control large-amplitude solution. Specifically, after a sufficient time $T$, the large amplitude gradually decays until it falls below the amplitude $\eta_0$ defined in Theorem \ref{smallmain} . Furthermore, Lemma \ref{relativedecrease} establishes the monotonicity of the relative entropy. By choosing $\ve_0>0$, which is upper bound for the relative entropy of large-amplitude initial data, to be smaller than $\ve_1$ in Theorem \ref{smallmain}, it is possible to unify the small amplitude regime and the large amplitude regime.

\bigskip
 
\subsection{Arrangement of the paper}
In Section \ref{section2}, we provide the analytical tools for the weighted $L^p_v L^\infty_x$ framework. In particular, we derive pointwise estimates for the linear operator $K$ and for the weighted gain term $w\Gamma^+(f,f)$, both of which are bounded in terms of $L^p_v$ norms. We also review relative entropy and lemmas associated with it. In Section \ref{section4}, we treat the small perturbation problem within the weighted $L^p_vL^\infty_x$ framework. Assuming both small initial data in the weighted $L^p_vL^\infty_x$ norm and small initial relative entropy, we derive the estimate for the solution operator combining the collision frequency and the nonlinear loss term. Based on the estimate, we construct a global-in-time solution and prove convergence to equilibrium with a sub-exponential decay rate. In Section \ref{section5}, we address the large amplitude problem. Starting from initial data that may be large in $L^p_v L^\infty_x$ but has small relative entropy, we apply Gr\"onwall inequality to show that the solution eventually enters the small regime. This yields global existence and convergence to equilibrium. In Appendix \ref{appendix}, we provide the local-in-time existence theory. 
 
\bigskip

\section{Preliminaries} \label{section2}
\subsection{Estimates on linear operators}
In this subsection, we introduce several estimates for the integral operator $K$ which will be used to derive $L^p_v L^\infty_x$ estimates. In the soft potential case, the integral kernel of $K$ has a singularity near $v \sim u$, and a careful decomposition together with suitable cutoff arguments is required to handle this difficulty. Recall the definition of operator $K:=K_2-K_1$
\begin{align} \label{Kf}
	(Kf)(v)&= \int_{\R^3} k(v,u) f(u)\, du =\int_{\R^3} k_2(v,u) f (u) \, du - \int_{\R^3} k_1 (v,u)f(u) \, du,
\end{align} 
where $k_i(v,u)$ is a symmetric integral kernel of $K_i$ for $i=1,2$. 
\begin{Lem} \label{Kker} \cite{Duan2017a}
 The integral kernel $k_1$ and $k_2$ of \eqref{Kf} satisfy 
\begin{equation*}
	0\leq k_1(v,u) \leq C \vert v-u \vert ^\gamma e^{-\frac{\vert v \vert^2}{4}} e^{-\frac{\vert u \vert ^2}{4}}, 
\end{equation*}
and 
\begin{equation*}
	0\leq k_2(v,u) \leq \frac{C_\gamma}{\vert v -u \vert^{\frac{3-\gamma}{2}}} e^{-\frac{\vert v-u\vert^2}{8}} e^{-\frac{\vert \vert v \vert^2 - \vert u \vert^2 \vert^2}{8\vert v-u\vert^2}},
\end{equation*}
where $C>0$ is a generic constant and $C_{\gamma}$ is a constant depending on $\gamma$. 
\end{Lem}

\bigskip

To treat the singularity of $K$, we introduce modified kernel with smooth cutoff function $0 \leq \chi \leq 1$ such that  
\begin{equation*} \label{cutoff}
	\chi(\vert v-u \vert) =
	\begin{cases}
	1, \quad \textrm{if} \quad\vert v-u \vert \geq 2\ve, \\ 
	0, \quad \textrm{if} \quad \vert v-u \vert \leq \ve,
	\end{cases}
\end{equation*}
where $0<\ve<1$.
We split the operator $K$ using cutoff function $\chi$ :
\begin{equation*}
	Kf=K^{ns}f + K^{s}f, \quad K_2f=K_2^{ns}f+K_2^{s}f, \quad \textrm{and} \quad K_1f=K_1^{ns} f + K_1^{s}f.
\end{equation*}
Specifically,  for $i=1,2$, we define 
\begin{align} \label{K-chi}
	\begin{split}
		K^{ns} f &= \int_{\R^3} \chi(\vert v-u \vert) k(v,u) f(u) \, du = \int_{\R^3} k^{ns}(v,u)f(u) \, du,\\ 
		K^{s} f &= \int_{\R^3} (1-\chi(\vert v-u \vert) k(v,u) f(u) \, du = \int_{\R^3} k^{s} (v,u) f(u) du, \\ 
		K_i^{ns} f &= \int_{\R^3} \chi(\vert v-u \vert) k_i(v,u) f(u) = \int_{\R^3} k_i^{ns} (v,u) f(u) \, du, \\
		K_i^{s}f &= \int_{\R^3} (1-\chi(\vert v-u \vert) k_i(v,u) f(u) \, du =\int_{\R^3} k_i^{s}(v,u) f(u) \, du.
	\end{split}
\end{align}
\begin{Lem} \label{k1ker} 
	There exists a constant $C>0$ such that 
	\begin{align*}
		|k_1^{ns}(v,u)| \leq C \varepsilon^{\gamma} e^{-\frac{|v|^2}{4}} e^{-\frac{|u|^2}{4}}. 
	\end{align*}
\end{Lem}
\begin{Lem} \label{k2ker} \cite{Strain2008}
	There are constants $C>0$ and $C_{\ve}>0$ depending on $\ve$ such that 
\begin{equation*}
	\vert k_2^{ns}(v,u)\vert \leq C \ve^{\gamma-1} \frac{\exp\left(-\frac{1}{8} \vert u - v\vert^2 -\frac{1}{8} \frac{(\vert v\vert^2-\vert u \vert^2)^2}{\vert v-u \vert^2}\right)}{\vert v -u \vert}, 
\end{equation*}	
or 
\begin{equation} \label{k2esti.1}
	\vert k_2^{ns}(v,u)\vert \leq C_{\ve} \frac{\exp \left(-\frac{s_2}{8}\vert v- u\vert^2 -\frac{s_1}{8} \frac{(\vert v \vert^2 -\vert u \vert^2)^2}{\vert v -u \vert^2} \right) } {\vert v-u \vert (1+\vert v \vert + \vert u \vert)^{1-\gamma}},
\end{equation}
for any $0<s_1<s_2<1$. 
\end{Lem}

\begin{Lem} \label{Ksingular}
Let $-3<\gamma<0$ and $p>\frac{3}{3+\gamma}$. There exists a constant $C_{p,\gamma}>0$ such that
\begin{equation*} 
	 w_{q,\vartheta,\beta} (v) K^{s}f \leq C_{p,\gamma} \mu(v)^{\frac{1-q}{8}} \ve^{\gamma+\frac{3}{p'}} \Vert w_{q,\vartheta,\beta} f \Vert_{L^p_v}, 
\end{equation*} 
\end{Lem}
\begin{proof}
In \eqref{K-chi}, recall that
\begin{equation} \label{Ksplit.1}
	w_{q,\vartheta,\beta}K^{s}f = w_{q,\vartheta,\beta}K^{s}_1f+w_{q,\vartheta,\beta} K^{s}_2f.
\end{equation}  
First of all, we consider the term $w_{q,\vartheta,\beta}K^{s}_1f$ in \eqref{Ksplit.1}:
\begin{align} \label{K1esti}
	w_{q,\vt,\beta}(v) K_1^{s} f &\leq  C w_{q,\vartheta,\beta}(v)\mu(v)^{1/2} \int_{\vert v-u \vert \le 2\ve} 
	\vert v-u \vert^{\gamma} \mu(u)^{1/2} \vert f(u)\vert du\nonumber \\
	&\le C(1+|v|^2)^{\beta/2}\exp\left\{\frac{q}{8}(1+(1+t)^{-\vartheta}) \vert v \vert^{2}\right\}\mu(v)^{1/2} \Vert w_{q,\vartheta,\beta}f \Vert_{L^p_v}\nonumber\\
	& \quad \times \left(\int_{|v-u| \le 2 \ve} |v-u|^{p'\gamma} \mu(u)^{p'/2} \frac{1}{w_{q,\vt,\beta}(u)^{p'}}du \right)^{1/p'} \nonumber \\
	&\leq C_{p,\gamma} \mu(v)^{\frac{1-q}{8}}\ve^{\gamma+\frac{3}{p'}} \Vert w_{q,\vartheta,\beta} f \Vert_{L^p_v}, 
\end{align}
where $p>\frac{3}{3+\gamma}$.
Next, we address the remaining part $w_{q,\vartheta,\beta}K_2^{s}f$ in \eqref{Ksplit.1}:
\begin{equation} \label{K2ssss}
	w_{q,\vartheta,\beta}(v)K_2^{s} f \leq Cw_{q,\vartheta,\beta}(v) \int_{\vert v-u \vert \leq 2\ve}\int_{\S^2} \vert v-u \vert^{\gamma} \mu(u)^{1/2}\left[ \mu(u')^{1/2}|f(v')| +\mu(v')^{1/2} |f(u')| \right]d\omega du.
\end{equation}
By the rotation, we can make an interchange of $v'$ and $u'$, and change the second term in \eqref{K2ssss} to the same form as the first term in \eqref{K2ssss}. Thus it suffices to estimate the second term in \eqref{K2ssss}:
\begin{align} \label{w222}
	w_{q,\vartheta,\beta}(v) \int_{\vert v-u \vert \leq 2\ve}\int_{\S^2} \vert v-u \vert^{\gamma} \mu(u)^{1/2} \mu(v')^{1/2}|f(u')|d\omega du.
\end{align}
On $\vert v-u \vert \leq 2\ve$ and $|v|> 2 \ve$, it holds that
\begin{align*} \label{w.2}
	\vert v' \vert &= \vert v + [(u-v) \cdot \omega] \omega \vert \geq \vert v \vert - \vert v-u \vert \geq \vert v \vert - 2\ve, \\ 
	\vert u \vert &= \vert v+u-v \vert \geq \vert v \vert - \vert v-u \vert \geq \vert v \vert -2\ve,
\end{align*}
which implies 
\begin{align}
	\mu(u)^{1/2}\mu(v')^{1/2} &\leq e^{-\frac{ (\vert v \vert -2\ve)^2}{4}} e^{- \frac{(\vert v \vert -2\ve)^2}{4}}\leq e^{-\frac{\vert v \vert^2}{2}} e^{2\ve\vert v \vert}  \leq C \mu(v)^{1/2},
\end{align}
where we have used $\ve<1$ and the following fact 
\begin{equation*}
		C\ve \vert v \vert =\left(\frac{\vert v \vert}{2} \right)\left(2C\ve\right) \leq \frac{\vert v \vert^2}{4} + C.
\end{equation*}
On the other hand, under $|v|\le 2\ve$, we have
\begin{align} \label{w3333}	
	1\le e^{\ve^2}e^{-\frac{|v|^2}{4}} < e \cdot e^{-\frac{|v|^2}{4}} \le C\mu(v)^{1/2},
\end{align} 
where we have used $\ve<1$.
From \eqref{w.2} and \eqref{w3333}, the term \eqref{w222} becomes
\begin{align} \label{K2 esti}
	& w_{q,\vartheta,\beta}(v) \int_{\vert v-u \vert \leq 2\ve}\int_{\S^2} \vert v-u \vert^{\gamma} \mu(u)^{1/2} \mu(v')^{1/2}|f(u')|d\omega du\nonumber\\
	& \le C w_{q,\vartheta,\beta}(v)\mu(v)^{1/2} \int_{\vert v-u \vert \leq 2\ve}\int_{\S^2} \vert v-u \vert^{\gamma} |f(u')|d\omega du\nonumber\\
	& \le C w_{q,\vartheta,\beta}(v)\mu(v)^{1/2}\left( \int_{\vert v-u \vert \leq 2\ve}\int_{\S^2} \vert v-u \vert^{p'\gamma} d\omega du\right)^{1/p'} \left(\int_{\vert v-u \vert \leq 2\ve}\int_{\S^2}  |w_{q,\vt,\beta}f(u')|^p d\omega du\right)^{1/p} \nonumber\\
	& \le C_{p,\gamma} \mu(v)^{\frac{ 1-q }{8}}\ve^{\gamma+\frac{3}{p'}} \Vert w_{q,\vartheta,\beta}f\Vert_{L^p_v},
\end{align}
where $p>\frac{3}{3+\gamma}$.
Inserting \eqref{K1esti} and \eqref{K2 esti} into \eqref{Ksplit.1}, we can deduce that 
\begin{align*}
	w_{q,\vartheta,\beta}(v) K^{s}f \leq  C_{p,\gamma} \mu(v)^{\frac{1-q}{8}} \ve^{\gamma+\frac{3}{p'}} \Vert w_{q,\vartheta,\beta}f \Vert_{L^p_v}.
\end{align*}

\end{proof}
\begin{Lem} \label{Knonsing}
Let $-3<\gamma<0$ and $p>\max\left\{\frac{3}{3+\gamma},2\right\}$. There exists a constant $C_{p,q,\ve}>0$ such that
\begin{equation*}
	w_{q,\vartheta,\beta}(v) \int_{\R^3} k^{ns} (v,u) e^{\ve \vert v-u \vert^2} \vert f(u) \vert \, du \leq C_{p,q,\ve} \langle v \rangle^{\gamma-1-\frac{1}{p'}} \Vert w_{q,\vartheta,\beta} f \Vert_{L^p_v}. 
\end{equation*}
 \end{Lem}
 \begin{proof}
 In \eqref{K-chi}, recall that
\begin{equation*} 
	w_{q,\vartheta,\beta}K^{ns}f = w_{q,\vartheta,\beta}K^{ns}_1f+w_{q,\vartheta,\beta} K^{ns}_2f.
\end{equation*}  
 	We firstly consider the part $w_{q,\vartheta,\beta}K^{ns}_1f$: 
\begin{align*}
	&w_{q,\vartheta,\beta}(v) \int_{\R^3} {k}_1^{ns}(v,u) e^{\ve \vert v-u\vert^2} \vert f(u) \vert  \,du \\ 
	&\quad \leq C w_{q,\vartheta,\beta}(v) \mu(v)^{1/2} \int_{\R^3} \chi(\vert v-u\vert) \vert v-u \vert^{\gamma} \mu(u)^{1/2} \frac{e^{\ve \vert v-u\vert^2} \vert w_{q,\vartheta,\beta}f(u) \vert}{w_{q,\vartheta,\beta}(u)}\, du\\
	&\quad \leq C_{p,q} \mu(v) ^{\frac{ 1-q}{8}}\Vert w_{q,\vartheta,\beta}f\Vert_{L^p_v}\left( \int_{\R^3} \vert v-u\vert^{p'\gamma}  \mu(u)^{p'/2} e^{p'\ve \vert v-u\vert^2}du \right)^{1/p'} \\
	&\quad \leq C_{p,q} \langle v \rangle^{\gamma} \mu(v)^{\frac{ 1-q }{16}}\Vert w_{q,\vartheta,\beta}f\Vert_{L^p_v},
\end{align*}	
where $p>\frac{3}{3+\gamma}$. \newline
It remains to deal with the term $w_{q,\vartheta,\beta}K^{ns}_2f$. Recall \eqref{k2esti.1} and we set $s_0=\min\{s_1,s_2\}$. We have
\begin{align} \label{K2esti.1}
	&w_{q,\vartheta,\beta}(v) \int_{\R^3} {k}_2^{ns} (v,u) \left(\frac{e^{\ve \vert v-u \vert ^2 }\vert w_{q,\vartheta,\beta}f(u)\vert}{w_{q,\vartheta,\beta}(u)}\right) \,du \nonumber \\ 
	&\quad \leq C_{\ve} \Vert w_{q,\vartheta,\beta}f \Vert_{L^p_v} \langle v \rangle^{\gamma-1} w_{q,\vartheta,\beta}(v) \left(\int_{\R^3} \frac{\exp\left(-\frac{p's_0}{8} \vert v-u \vert^2 -\frac{p's_0}{8} \frac{(\vert v \vert^2 - \vert u \vert^2)^2}{\vert v-u \vert^2} \right)}{\vert v-u \vert^{p'}} \frac{e^{p'\ve\vert v-u\vert ^2}}{w_{q,\vartheta,\beta}(u)^{p'}}du\right)^{1/p'}.
\end{align}
We can easily compute
\begin{equation*}
	\frac{w_{q,\vartheta,\beta}(v)}{w_{q,\vartheta,\beta}(u)} \leq C_{\beta}(1+|v-u|^2)^{\beta/2}e^{-\frac{\tilde{q}}{4}(\vert u \vert^2 - \vert v \vert ^2)},
\end{equation*}
where $\tilde{q}=\frac{q}{2}\left(1+(1+t)^{-\vartheta}\right)$. We mention  $\frac{q}{2} < \tilde{q} \leq q$. Denote $z:=v-u$. We now calculate the total exponent in the integral of \eqref{K2esti.1}: 
\begin{align*}
	&-\frac{p's_0}{8} \vert z \vert^2 -\frac{p's_0}{8} \frac{\vert |z \vert^2 - 2v \cdot z \vert^2}{\vert z\vert ^2} - \frac{p'\tilde{q}}{4} ( \vert v -z \vert^2 -\vert v \vert^2)\\
	&\quad =-\frac{p's_0}{4} \vert z \vert^2 +\frac{p's_0}{2} v\cdot z - \frac{p's_0}{2} \frac{ \vert v \cdot z \vert^2}{\vert z \vert^2} -\frac{ p'\tilde{q}}{4}(\vert z \vert ^2 -2v\cdot z)\\
	&\quad = -\frac{p'}{4}(\tilde{q}+s_0) \vert z \vert ^2 +\frac{p'}{2} (s_0+\tilde{q}) v\cdot z - \frac{p's_0}{2} \frac{(v\cdot z)^2}{ \vert z\vert ^2}. 
\end{align*}
Setting
\begin{equation*}  
0<\tilde{q} \leq q < s_0<1,
\end{equation*}
the discriminant of the above quadratic form of $\vert \eta \vert $ and $\frac{v\cdot \eta}{\vert \eta \vert}$ is 
\begin{equation*}
	\Delta = \frac{p'}{4}(s_0+\tilde{q})^2 - (\tilde{q}+ s_0)\frac{p's_0}{2} = \frac{p'}{4}(\tilde{q}^2 -s_0^2) <0. 
\end{equation*}
Thus, we have, for $\ve>0$ sufficiently small and $q<s_0$, that there is $C_q>0$ such that 
\begin{align} \label{e.1}
	&-\frac{s_0-8\ve}{8} \vert z \vert^2 -\frac{s_0}{8} \frac{\vert \vert z \vert ^2 -2v\cdot z \vert^2}{\vert z \vert^2} - \frac{\tilde{q}}{4} (\vert z \vert^2 -2v\cdot z) \nonumber \\ 
	&\quad \leq -C_q \left ( \vert z\vert^2 +\frac{\vert v \cdot z \vert^2}{\vert z \vert^2} \right ) \nonumber\\ 
	&\quad = -C_q\left ( \frac{\vert z \vert^2}{2} +\left(\frac{\vert z \vert^2}{2}+\frac{\vert v \cdot z \vert^2}{\vert z \vert^2}\right) \right ) \nonumber\\
	&\quad \leq -C_q \left ( \frac{\vert z \vert^2}{2} + \vert v \cdot z \vert \right ).
\end{align}

Inserting \eqref{e.1}  into \eqref{K2esti.1}, we have 
\begin{align}\label{e.3}
	&\Vert w_{q,\vartheta,\beta}f \Vert_{L^p_v} \langle v \rangle^{\gamma-1} w_{q,\vartheta,\beta}(v) \left(\int_{\R^3} \frac{\exp\left(-\frac{p's_0}{8} \vert v-u \vert^2 -\frac{p's_0}{8} \frac{(\vert v \vert^2 - \vert u \vert^2)^2}{\vert v-u \vert^2} \right)}{\vert v-u \vert^{p'}} \frac{e^{p'\ve\vert v-u\vert ^2}}{w_{q,\vartheta,\beta}(u)^{p'}}\,du\right)^{1/p'} \nonumber \\
	&\quad \leq C_{q,\ve} \langle v \rangle^{\gamma-1} \Vert w_{q,\vartheta,\beta}f \Vert_{L^p_v} \left(\int_{\R^3} \frac{(1+\vert z \vert^2)^{p'\beta/2}}{\vert z \vert^{p'}} \exp\left \{ -C_{p,q} \left ( \frac{\vert z \vert^2}{2}+ \vert v \cdot z \vert \right ) \right\} \,dz\right)^{1/p'} \nonumber \\
	&\quad \leq C_{p,q,\ve} \langle v \rangle^{\gamma-1} \Vert w_{q,\vartheta,\beta}f \Vert_{L^p_v} \left(\int_{\R^3}\frac{1}{\vert z \vert^{p'}} \exp\left \{ -C_{p,q} \left ( \frac{\vert z \vert^2}{4}+ \vert v \cdot z\vert \right ) \right\} dz\right)^{1/p'}.
\end{align} 
For $\vert v \vert \geq 1 $, we make a change of variables $z_{\shortparallel}= \left( z \cdot \frac{v}{\vert v \vert} \right ) \frac{v}{\vert v \vert}$, and $z_\perp = z -z_{\shortparallel}$ so that $ \vert v \cdot z\vert = \vert v \vert \cdot \vert z_{\shortparallel}\vert$, which implies that 
\begin{align}\label{e.4}
 &\int_{\R^3} \frac{1}{\vert z \vert^{p'}} \exp\left \{ -C_{p,q} \left ( \frac{\vert z \vert^2}{4}+ \vert v \cdot z \vert \right ) \right\} \,dz \nonumber\\ 
 &\quad \leq C \int_{\R^2} \frac{1}{\vert z_{\perp} \vert^{p'}} e^{-C_{p,q} \frac{\vert z \vert^2}{4}} \left \{\int_{-\infty}^\infty e^{-C_{p,q} \vert v \vert \cdot \vert z_{\shortparallel} \vert} \, d \vert z_{\shortparallel} \vert \right\}\, dz_{\perp} \nonumber \\
 &\quad \leq \frac{C_{p,q}}{\vert v \vert} \int_{\R^2} \frac{1}{\vert z_\perp\vert^{p'}} e^{-\frac{C_{p,q}}{4} \vert z_\perp \vert^2} \left \{ \int_{-\infty}^\infty e^{-C_q \vert y \vert} dy\right \} \, d z_{\perp} \nonumber\\
 &\quad \leq \frac{C_{p,q}}{1+\vert v \vert}, 
\end{align}
where $p' < 2$.
On the other hand, for $\vert v \vert \leq 1$, 
\begin{align} \label{e.5}
&\int_{\R^3} \frac{1}{\vert z \vert^{p'}} \exp\left \{ -C_{p,q} \left ( \frac{\vert z \vert^2}{4}+ \vert v \cdot z \vert \right ) \right\} \,dz\nonumber \\ 
 &\quad \leq C \int_{\R^2} \frac{1}{\vert z_{\perp} \vert^{p'}}  e^{-\frac{C_{p,q}}{4} \vert z_\perp \vert^2}  \left \{\int_{-\infty}^\infty e^{-\frac{C_{p,q}}{4}\vert z_{\shortparallel}\vert^2}e^{-C_{p,q} \vert v \vert \cdot \vert z_{\shortparallel} \vert} \, d \vert z_{\shortparallel} \vert \right\}\, dz_{\perp}\nonumber\\
 &\quad \leq C \int_{\R^2} \frac{1}{\vert z_\perp\vert^{p'}} e^{-\frac{C_{p,q}}{4} \vert z_\perp \vert^2} \left \{ \int_{-\infty}^\infty e^{-\frac{C_{p,q}}{4}\vert z_{\shortparallel}\vert^2} d\vert z_{\shortparallel}\vert\right \} \, d z_{\perp} \quad \nonumber\\
 &\quad \leq \frac{C_{p,q}}{1+\vert v \vert}, 
\end{align}
where $p'<2$.
Inserting \eqref{e.3}, \eqref{e.4}, and \eqref{e.5} into \eqref{K2esti.1}, we obtain 
\begin{equation*}
	w_{q,\vartheta,\beta}(v) \int_{\R^3} {k}_2^{ns} (v,u) \left( \frac{e^{\ve \vert v-u \vert^2}\vert w_{q,\vartheta,\beta}f(u) \vert}{w_{q,\vartheta,\beta}(u)}\right) \,du \leq C_{p,q,\ve} \langle v \rangle^{\gamma-1-\frac{1}{p'}} \Vert w_{q,\vartheta,\beta}f \Vert_{L^p_v}.
\end{equation*}
Thus, we complete the proof of Lemma \ref{Knonsing}.
 \end{proof}
 
\begin{Coro} \label{pKestimate} 
	Let $-3<\gamma<0$ and $p>\max\left\{\frac{3}{3+\gamma},2 \right\}$. We have
	\begin{align*}
		\int_{\R^3} |k_w(v,u)|^{p'}du \le \frac{C_{p,q,\beta}}{(1+|v|)^{p'(1-\gamma)+1}}
	\end{align*}	
	for $v \in \R^3$.
\end{Coro}

\bigskip

\subsection{Estimates on nonlinear operators}
\label{Estimates on nonlinear operators}
 The estimates obtained in this subsection will be used to control the nonlinear terms in the $L^p_v L^\infty_x$ framework. Due to the different behaviors of the loss and gain parts of the collision operator, we treat them separately. We emphasize that the estimates for the gain term $\Gamma^+$ play a crucial role in the large-amplitude $L^p_v L^\infty_x$ analysis. In particular, they provide the key mechanism for deriving relative entropy. 

\begin{Lem} \label{pointwiseGamma-estimate} Let $-3 < \gamma < 0$ and $\ell>\frac{3}{3+\gamma}$. Then there exists a constant $C_{\ell}>0$, depending on $\ell$, such that
\begin{align} \label{LpGamma111111}
	\left|w_{q,\vt,\beta}(v) \Gamma^-(f,g)(v) \right| \le C_\ell \nu(v) |w_{q,\vt,\beta}f(v)| \left(\int_{\R^3}\mu(u)^{1/2}|g(u)|^{\ell}du\right)^{1/\ell}.
\end{align}
 Furthermore, if $p>\frac{3}{3+\gamma}$, it holds that
 \begin{align} \label{LpGamma1111}
 	\left\|w_{q,\vt,\beta}\Gamma^-(f,g) \right\|_{L^p_v} \le C_{p} \|w_{q,\vt,\beta}f\|_{L^p_v} \|w_{q,\vt,\beta}g\|_{L^p_v}.
 \end{align}
\end{Lem}
\begin{proof}
We can easily compute
	\begin{align*}
		\left|w_{q,\vt,\beta}(v) \Gamma^-(f,g)(v) \right| &\le w_{q,\vt,\beta}(v) \int_{\R^3} \int_{\S^2}b(\cos \theta) |v-u|^\gamma  \mu(u)^{1/2} |f(v)||g(u)|d\omega du\\
		& \le C |w_{q,\vt,\beta}(v)f(v)| \int_{\R^3} |v-u|^\gamma \mu(u)^{1/2} |g(u)|du\\
		& \le C|w_{q,\vt,\beta}(v)f(v)| \left(\int_{\R^3}  |v-u|^{\ell'\gamma} \mu(u)^{1/2} du\right)^{1/\ell'} \left(\int_{\R^3}\mu(u)^{1/2}|g(u)|^\ell du\right)^{1/\ell}\\
		& \le C_{\ell}\nu(v)|w_{q,\vt,\beta}(v)f(v)| \left(\int_{\R^3}\mu(u)^{1/2}|g(u)|^\ell du\right)^{1/\ell},
	\end{align*}
	where $\ell'\gamma>-3$.\\
	\indent If we set $\ell =p$ and then we take the $L^p_v$ norm to \eqref{LpGamma111111}, we can derive the estimate \eqref{LpGamma1111}. 	
\end{proof}

\bigskip

Before handling the pointwise estimate for $w_{q,\vt,\beta} \Gamma^+(f,f)$, we need to calculate the Jacobian determinant $\left|\frac{du'}{du} \right|$ because it is necessary to derive the following estimate:
\begin{align*}
	\left( \int_{\R^3}\int_{\S^2}b(\cos \theta) \left|w_{q,\vt,\beta}(u')f(u')\right|^pd\omega du\right)^{1/p} \lesssim \|w_{q,\vt,\beta}f\|_{L^p_v}.
\end{align*}
To calculate the determinant $\left|\frac{du'}{du} \right|$, we temporarily use the $\sigma$-representation instead of the $\omega$-representation \eqref{omegarep} :
\begin{align*}
	v' = \frac{v+u}{2} + \frac{|v-u|}{2}\sigma , \quad u'=\frac{v+u}{2}-\frac{|v-u|}{2}\sigma
\end{align*}
with $\sigma \in \S^2$. Then the collision operator \eqref{collisionoperator} becomes
\begin{align*}
	Q(F_1,F_2) = C_\sigma\int_{\R^3}\int_{\S^2} \tilde{B}(v-u,\sigma) \left[F_1(v')F_2(u')-F_1(v)F_2(u) \right]d\sigma du,
\end{align*}
where $C_\sigma$ is some constant. 
Here, the collision kernel $\tilde{B}$ is
\begin{align*}
	\tilde{B}(v-u,\sigma) = \tilde{b}(\cos \tilde{\theta}) |v-u|^\gamma, \quad 0\le \tilde{b}(\cos \tilde{\theta}) \le C_{\tilde{b}},
\end{align*}
where $-3 < \gamma < 0$, $0\le \tilde{\theta} \le \frac{\pi}{2}$, and $\cos\tilde{\theta} = \frac{(v-u)}{|v-u|}\cdot \sigma$. For each $\sigma$ and fixed $v$, we make the change of variables $u \mapsto u'$. This change of variables is well-defined on the set $\{\tilde{\theta} : \cos\tilde{\theta}>0  \}$, and we have
\begin{align*}
	\left|\frac{du'}{du} \right| = \left|\frac{1}{2}I + \frac{1}{2} \frac{(v-u)}{|v-u|}\otimes\sigma \right| = \frac{1+\cos \tilde{\theta}}{8}\ge \frac{1}{8}.
\end{align*}
Using the change of variables, we can derive
\begin{align} \label{Lpchangeofvariable}
	\left( \int_{\R^3}\int_{\S^2}b(\cos \theta) \left|w_{q,\vt,\beta}(u')f(u')\right|^pd\omega du\right)^{1/p} &\lesssim \left( \int_{\R^3}\int_{\S^2}\tilde{b}(\cos \tilde{\theta}) \left|w_{q,\vt,\beta}(u)f(u)\right|^pd\sigma du\right)^{1/p} \nonumber\\
	&\lesssim \|w_{q,\vt,\beta}f\|_{L^p_v}.
\end{align}

\bigskip

\begin{Lem} \label{pointwiseGamma+estimate} Let $-3 < \gamma < 0$. 
Then the following hold:
\begin{enumerate}[label=(\arabic*)]
  \item If $-1\le \gamma<0$ and $p> 7-4\gamma$, then there exists a constant $C_{p,\gamma,\beta}>0$ such that
	\begin{align} \label{pointwiseGamma+estimate1}
		\left|w_{q,\vt, \beta}(v) \Gamma^+(f,f)(v) \right| \le \frac{C_{p,\gamma,\beta} \|w_{q,\vt,\beta}f\|_{L^p_v}}{(1+|v|)^{-\frac{\gamma}{p}+\frac{5p-5}{4p}}}\left(\int_{\R^3}(1+|\eta|)^{\frac{4p-8}{p}-2\beta}|w_{q,\vt,\beta}f(\eta)|^2 d\eta \right)^{1/2}
	\end{align}
	for all $t\in [0,\infty)$, $v\in \R^3$.\\
	\item If $-3<\gamma<-1$ and $p>\frac{6}{3+\gamma}$, then there exists a constant $C_{p,\gamma,\beta}>0$ such that
	\begin{align}\label{pointwiseGamma+estimate2}
		\left|w_{q,\vt, \beta}(v) \Gamma^+(f,f)(v) \right| \le \frac{C_{p,\gamma,\beta} \|w_{q,\vt,\beta}f\|_{L^p_v}}{(1+|v|)^{-\gamma+\frac{p-1}{p}\varpi}}\left(\int_{\R^3}(1+|\eta|)^{\frac{4}{m-1}-p'm'\beta}|w_{q,\vt,\beta}f(\eta)|^{p'm'} d\eta \right)^{\frac{1}{p'm'}}
	\end{align}
	for all $t\in [0,\infty)$, $v\in \R^3$, where $m=\frac{1}{2}\left(\frac{p-1}{p-2}-\frac{3}{p'\gamma}\right)$ and $\varpi=\frac{1}{2}\left(p'\gamma+\frac{3}{m}\right)$. Note that $p'm'<p$.
\end{enumerate}
\end{Lem}
\begin{proof}
	From the energy conservation law \eqref{momentumenergyconserv}, we get $|v|^2 \le |u'|^2+|v'|^2$, which implies that
	\begin{align*}
		\text{either} \quad |v|^2 \le 2|u'|^2 \quad \text{or} \quad |v|^2 \le 2|v'|^2.
	\end{align*}
	Then we have
	\begin{equation} \label{gamma1}
	\begin{aligned}
		&\left|w(v) \Gamma^+(f,f)(v) \right|\\ &\le C_\beta\int_{\R^3}\int_{\S^2}B(v-u, \omega) \mu(u)^{1/2}\left|w_{q,\vt,\beta}(u')f(u')\exp\left\{\frac{q}{8} \left(1+\frac{1}{(1+t)^\vt}\right)|v'|^2\right\}f(v') \right|d\omega du\\
		&\quad +C_\beta\int_{\R^3}\int_{\S^2}B(v-u, \omega) \mu(u)^{1/2}\left|\exp\left\{\frac{q}{8} \left(1+\frac{1}{(1+t)^\vt}\right)|u'|^2\right\}f(u')w_{q,\vt,\beta}(v')f(v') \right|d\omega du.
\end{aligned}
\end{equation}
By the rotation, one can make an interchange of $v'$ and $u'$, and change the second term in \eqref{gamma1} to the same form as the first term in \eqref{gamma1}. Thus it suffices to estimate the first term in \eqref{gamma1}. We denote the first term in \eqref{gamma1} by $I_1$.\\
\indent We use the H\"older inequality to obtain
\begin{align*}
	I_1 &\le C_\beta \left( \int_{\R^3}\int_{\S^2}b(\cos \theta) \left|w_{q,\vt,\beta}(u')f(u')\right|^pd\omega du\right)^{1/p}\\
	& \quad \times \left(\int_{\R^3}\int_{\S^2}b(\cos \theta)|v-u|^{p'\gamma} \mu^{p'/2}(u)\exp\left\{\frac{qp'}{8}\left(1+\frac{1}{(1+t)^\vt} \right)|v'|^2\right\}|f(v')|^{p'} d\omega du \right)^{1/p'},
\end{align*}
where $p'$ is the exponent conjugate of $p$. Using \eqref{Lpchangeofvariable}, we obtain 
\begin{align*}
	\left( \int_{\R^3}\int_{\S^2}b(\cos \theta) \left|w_{q,\vt,\beta}(u')f(u')\right|^pd\omega du\right)^{1/p} \lesssim \|w_{q,\vt,\beta}f\|_{L^p_v}.
\end{align*}
As in \cite{Glassey1996}, we rewrite \eqref{omegarep} as $u'=v+z_\perp$, $v' = v+z_\shortparallel$, with $z=u-v$, $z_\shortparallel = (z\cdot \omega)\omega$, $z_\perp = z-(z\cdot \omega)\omega$. Then we have
\begin{align*}
	b(\cos \theta)|v-u|^{p' \gamma}\le C\left|\frac{(v-u)}{|v-u|}\cdot \omega \right||v-u|^{p'\gamma} \le C|z_\shortparallel||z|^{p'\gamma-1} \le C|z_\shortparallel|\left(|z_\shortparallel|^2+|z_\perp|^2\right)^{\frac{p'\gamma-1}{2}}.
\end{align*}
We make a change of variables $u\mapsto z$ to obtain
\begin{align*}
	I_1 &\le C_\beta \|w_{q,\theta,\beta}f\|_{L^p_v} \bigg(\int_{\R^3}\int_{\S^2}|z_\shortparallel|\left(|z_\shortparallel|^2+|z_\perp|^2\right)^{\frac{p'\gamma-1}{2}} e^{-\frac{p'|v+z|^2}{4}}\exp\left\{\frac{qp'}{8}\left(1+\frac{1}{(1+t)^\vt} \right) |v+z_{\shortparallel}|^2 \right\}\\
	&\quad \times \left| f(v+z_\shortparallel) \right|^{p'} d\omega dz \bigg)^{1/p'}.
\end{align*}
Note that $dz = dz_\perp d|z_{\shortparallel}|$ and we make a change of variables $(\omega, |z_{\shortparallel}|) \mapsto \eta:=v+z_{\shortparallel}$ with $d|z_{\shortparallel}| d\omega = \frac{1}{|z_\shortparallel|^2}d\eta$:
\begin{align*} 
	I_1 &\le C_\beta \|w_{q,\vt,\beta}f\|_{L^p_v} \Bigg( \int_{\R^3}  \frac{\exp\left\{ \frac{qp'}{8}\left(1+\frac{1}{(1+t)^\vt} \right) |\eta|^2 \right\}|f(\eta)|^{p'}}{|\eta-v|} \nonumber\\
	& \quad \times  \Bigg(\int_{(\eta-v)\cdot z_\perp=0}\left(|\eta-v|^2+|z_\perp|^2\right)^{\frac{p'\gamma-1}{2}} e^{-\frac{p'|\eta+z_\perp|^2}{4}} dz_\perp \Bigg)  d\eta \Bigg)^{1/p'}\nonumber\\
	& \le C_\beta \|w_{q,\vt,\beta}f\|_{L^p_v} \Bigg( \int_{\R^3}  \frac{\exp\left\{ \frac{qp'}{8}\left(1+\frac{1}{(1+t)^\vt} \right) |\eta|^2 \right\}|f(\eta)|^{p'}}{|\eta-v|} \nonumber\\
	& \quad \times  \Bigg(\int_{(\eta-v)\cdot z_\perp=0}|z_\perp|^{-2+\varpi}\left(|\eta-v|^2+|z_\perp|^2\right)^{\frac{p'\gamma+1-\varpi}{2}} e^{-\frac{p'|\eta+z_\perp|^2}{4}} dz_\perp \Bigg)  d\eta \Bigg)^{1/p'},
\end{align*}
where $0<\varpi<2$ is chosen later.\\
\newline 
\textbf{Case 1:} $-1\le  \gamma <0$.\\
If $p'\gamma + 1 - \varpi < 0$, which is equivalent to requiring
\[
p > \frac{1-\varpi}{1-\varpi+\gamma}
\quad \text{when } \varpi > 1+\gamma,
\]
we have
\begin{align*}
	I_1 &\le C_\beta \|w_{q,\vt,\beta}f\|_{L^p_v} \Bigg( \int_{\R^3}  |\eta-v|^{p'\gamma-\varpi}\exp\left\{ \frac{qp'}{8}\left(1+\frac{1}{(1+t)^\vt} \right) |\eta|^2 \right\}|f(\eta)|^{p'} \\
	& \quad \times  \Bigg(\int_{(\eta-v)\cdot z_\perp=0}|z_\perp|^{-2+\varpi} e^{-\frac{p'|\eta+z_\perp|^2}{4}} dz_\perp \Bigg)  d\eta \Bigg)^{1/p'}.
\end{align*}
Here, 
\begin{align*}
	\int_{(\eta-v)\cdot z_\perp=0}|z_\perp|^{-2+\varpi} e^{-\frac{p'|\eta+z_\perp|^2}{4}} dz_\perp\le C_{p,\gamma}
\end{align*}
for some constant $C_{p,\gamma}>0$. It follows that
\begin{align*}
	I_1 \le C_{p,\gamma,\beta}\|w_{q,\vt,\beta}f\|_{L^p_v} \left(\int_{\R^3}  |\eta-v|^{p'\gamma-\varpi}\left|\exp\left\{ \frac{q}{8}\left(1+\frac{1}{(1+t)^\vt} \right) |\eta|^2 \right\}f(\eta)\right|^{p'} d\eta \right)^{1/p'}.
\end{align*}
Since $p'<2$, we use the Cauchy-Schwarz inequality to derive that
\begin{align*}
	&\left(\int_{\R^3}  |\eta-v|^{p'\gamma-\varpi}\left|\exp\left\{ \frac{q}{8}\left(1+\frac{1}{(1+t)^\vt} \right) |\eta|^2 \right\}f(\eta)\right|^{p'} d\eta\right)^{1/p'} \\
	& \le C_{p,\beta}\left(\int_{\R^3} |\eta-v|^{(p'\gamma-\varpi)\frac{2p-2}{p-2}} \frac{1}{(1+|\eta|)^4}d \eta  \right)^{\frac{p-2}{2p}} \left(\int_{\R^3}(1+|\eta|)^{\frac{4p-8}{p}-2\beta}|w_{q,\vt,\beta}f(\eta)|^2 d\eta\right)^{1/2}.
\end{align*}
If $(p'\gamma - \varpi)\frac{2p-2}{p-2} > -3$, that is, $p > \dfrac{6-2\varpi}{2\gamma - 2\varpi + 3}$ when $\varpi < \dfrac{2\gamma+3}{2}$. It holds that
\begin{align*}
	\int_{\R^3} |\eta-v|^{(p'\gamma-\varpi)\frac{2p-2}{p-2}} \frac{1}{(1+|\eta|)^4}d \eta \le \frac{C_{p,\gamma}}{(1+|v|)^{(-p'\gamma+\varpi)\frac{2p-2}{p-2}}},
\end{align*}
which implies that
\begin{align*}
	I_1 \le \frac{C_{p,\gamma,\beta}\|w_{q,\vt,\beta}f\|_{L^p_v}}{(1+|v|)^{(-p'\gamma+\varpi)\frac{p-1}{p}}}\left(\int_{\R^3}(1+|\eta|)^{\frac{4p-8}{p}-2\beta}|w_{q,\vt,\beta}f(\eta)|^2 d\eta\right)^{1/2}.
\end{align*}
We now consider the condition on $\ve$ : $1+\gamma< \varpi <\gamma+3/2 $. Then we choose $\varpi := \gamma+\frac{5}{4}$.\\
The condition on $p$ : $p>\frac{1-\varpi}{1-\varpi+\gamma}=4\gamma+1$, $p>\frac{6-2\varpi}{2\gamma-2\varpi+3}=7-4\gamma$.\\
\newline
\textbf{Case 2:} $-3 < \gamma <-1$.\\
Under the condition that $p'\gamma + 1 - \varpi < 0$, that is, $p'\gamma + 1 < \varpi$, we have
\begin{align*}
	I_1 &\le C_\beta \|w_{q,\vt,\beta}f\|_{L^p_v} \Bigg( \int_{\R^3}  |\eta-v|^{p'\gamma-\varpi}\exp\left\{ \frac{qp'}{8}\left(1+\frac{1}{(1+t)^\vt} \right) |\eta|^2 \right\}|f(\eta)|^{p'} \\
	& \quad \times  \Bigg(\int_{(\eta-v)\cdot z_\perp=0}|z_\perp|^{-2+\varpi} e^{-\frac{p'|\eta+z_\perp|^2}{4}} dz_\perp \Bigg)  d\eta \Bigg)^{1/p'}.
\end{align*}
Here, 
\begin{align*}
	\int_{(\eta-v)\cdot z_\perp=0}|z_\perp|^{-2+\varpi} e^{-\frac{p'|\eta+z_\perp|^2}{4}} dz_\perp\le C_{p,\gamma}
\end{align*}
for some constant $C_{p,\gamma}>0$. It follows that
\begin{align*}
	I_1 \le C_{p,\gamma,\beta}\|w_{q,\vt,\beta}f\|_{L^p_v} \left(\int_{\R^3}  |\eta-v|^{p'\gamma-\varpi}\left|\exp\left\{ \frac{q}{8}\left(1+\frac{1}{(1+t)^\vt} \right) |\eta|^2 \right\}f(\eta)\right|^{p'} d\eta \right)^{1/p'}.
\end{align*}
We use the H\"older inequality to derive
\begin{align*}
	&\left(\int_{\R^3}  |\eta-v|^{p'\gamma-\varpi}\left|\exp\left\{ \frac{q}{8}\left(1+\frac{1}{(1+t)^\vt} \right) |\eta|^2 \right\}f(\eta)\right|^{p'} d\eta \right)^{1/p'}\\
	& \le \left(\int_{\R^3}  |\eta-v|^{(p'\gamma-\varpi)m}\frac{1}{(1+|\eta|)^4}d\eta\right)^{\frac{1}{p'm}}\\
	& \quad \times  \left(\int_{\R^3}(1+|\eta|)^{\frac{4}{m-1}}\left|\exp\left\{ \frac{q}{8}\left(1+\frac{1}{(1+t)^\vt} \right) |\eta|^2 \right\}f(\eta)\right|^{p'm'} d\eta \right)^{\frac{1}{p'm'}},
\end{align*}
where $m$ is chosen later, depending on $p$ and $\gamma$, and $m'$ is the exponent conjugate of $m$. Note that we will derive $L^p$ estimate for $w\Gamma^+$ from the above estimate, so we need the condition on $p,m$ as $p'm'<p$, namely,
\begin{align*}
m > \frac{p-1}{p-2}
\quad \text{for } p > 2.
\end{align*}
Under the condition $(p'\gamma - \varpi)m > -3$, that is,
\begin{align*}
\varpi < \frac{3}{m} + p'\gamma,
\end{align*}
it holds that
\begin{align*}
	\int_{\R^3} |\eta-v|^{(p'\gamma-\varpi)m} \frac{1}{(1+|\eta|)^4}d \eta \le \frac{C_{p,\gamma}}{(1+|v|)^{(-p'\gamma+\varpi)m}},
\end{align*}
which implies that 
\begin{align*}
	I_1 \le \frac{C_{p,\gamma,\beta}\|w_{q,\vt,\beta}f\|_{L^p_v}}{(1+|v|)^{-\gamma+\frac{p-1}{p}\varpi}}\left(\int_{\R^3}(1+|\eta|)^{\frac{4}{m-1}-p'm'\beta}|w_{q,\vt,\beta}f(\eta)|^{p'm'} d\eta\right)^{\frac{1}{p'm'}}.
\end{align*}
The condition on $\varpi$ is given by $\max\{0,p'\gamma+1\} < \varpi < \min\{2, p'\gamma+\frac{3}{m}\}$.
Here, $m$ must satisfy $1<\frac{3}{m}$ and $p'\gamma+\frac{3}{m}>0$, which implies $m< -\frac{3}{p'\gamma}$ since $-3<\gamma<-1$. The condition on $m$ is given by $\frac{p-1}{p-2}<m < -\frac{3}{p'\gamma}$. Here, $p$ must satisfy $\frac{p-1}{p-2}<-\frac{3}{p'\gamma}$, which is equivalent to $p>\frac{6}{3+\gamma}$. Under these conditions, we can choose $m:=\frac{1}{2}\left(\frac{p-1}{p-2}-\frac{3}{p'\gamma}\right)$, and we can also take $\varpi:=\frac{1}{2}\left(p'\gamma+\frac{3}{m}\right)$.
\bigskip
\begin{remark}
	The time integration $\int_0^t e^{-\nu(v)(t-s)} ds\sim \frac{1}{\nu(v)}$ absorbs the term $\nu(v)$. Thus from the pointwise estimate for $w\Gamma^+$, we lose the velocity decaying factor of $\nu(v)$. In order to guarantee a $L^p$ framework, $\frac{p-1}{p}\bar{\omega} \cdot p>3$ is necessary. This requires  
	\begin{align*}
		p> \frac{3(\sqrt{8\gamma^2 +9} +3-2\gamma)}{(-\gamma)(3+\gamma)}.
	\end{align*} 
	Notice that for all $\gamma \in (-3,-1)$,
	\begin{align*}
		\frac{3(\sqrt{8\gamma^2 +9} +3-2\gamma)}{(-\gamma)(3+\gamma)} > \frac{6}{3+\gamma}. 
	\end{align*} 
	Hence, when $-3<\gamma<-1$, we expect the range of $p$ to be $p > \frac{3(\sqrt{8\gamma^2 +9} +3-2\gamma)}{(-\gamma)(3+\gamma)}$. 
\end{remark}

\end{proof}
\begin{Coro} \label{LpGamma+est}		Let $-3 < \gamma < 0$. Assume that $p$ satisfies the condition
\begin{align*}
	 p > \begin{cases}
7-4\gamma, & \text{if } -1 \le \gamma <0, \\
\frac{6}{3+\gamma}, & \text{if } -3< \gamma <-1,
\end{cases}
\end{align*}
and $\beta$ satisfies the condition 
\begin{align*}
	\beta >\begin{cases}
\frac{7p-14}{2p}, & \text{if }-1 \le \gamma <0, \\
\frac{3pm-6m+p-1}{pm}, & \text{if }-3< \gamma <-1.
\end{cases}
\end{align*}
Then there exists a constant $C_{p,\gamma,\beta}>0$ such that
	\begin{align*}
		\left\|w_{q,\vt,\beta}\Gamma^+(f,f)\right\|_{L^p_v} &\le C_{p,\gamma,\beta} \|w_{q,\vt,\beta}f\|_{L^p_v}^2.
	\end{align*}
\end{Coro}

\bigskip

\subsection{Relative entropy}
We recall the definition of the relative entropy associated with the Maxwellian $\mu$:
\begin{align*}
	\mathcal{E}(F)(t) = \int_{\T^3} \int_{\R^3} \left(\frac{F}{\mu} \ln \frac{F}{\mu}- \frac{F}{\mu} +1\right)\mu dvdx. 
\end{align*}
As in the $H$-theorem for the Boltzmann equation, the following lemma shows that the relative entropy is non-increasing in time.
\begin{Lem} \label{relativedecrease} \cite{Ko2022JDE}
	Assume $F$ satisfies the Boltzmann equation \eqref{Boltzmanneq} with initial data $F_0$. Then
	\begin{align*}
		\mathcal{E}(F(t)) \le \mathcal{E}(F_0)
	\end{align*}
	for all $t \ge 0$.
\end{Lem}

\bigskip

The following lemma provides a control mechanism for the solutions in the $L^p_v L^\infty_x$  space. In particular, combined with the gain term estimates established in the previous subsection, it allows us to use smallness of initial relative entropy within the $L^p_v L^\infty_x$ framework.
\bigskip
\begin{Lem} \label{L1L2cont} \cite{Guo2010Bounded}
	Assume $F$ satisfies the Boltzmann equation \eqref{Boltzmanneq}. We have
	\begin{align*}
		\int_{\T^3} \int_{\mathbb{R}^3} \frac{1}{4\mu}|F-\mu|^2 \mathbf{1}_{\{|F-\mu|\le \mu\}} dvdx + \int_{\T^3} \int_{\mathbb{R}^3}\frac{1}{4}|F-\mu| \mathbf{1}_{\{|F-\mu| > \mu\}} dvdx \le \mathcal{E}(F_0)
	\end{align*}
	for all $t \ge 0$.
	Moreover, if we write $F = \mu + \mu^{1/2} f$, then
	\begin{align*}
	\int_{\T^3} \int_{\mathbb{R}^3} \frac{1}{4}|f|^2 \mathbf{1}_{\{|f|\le \sqrt{\mu}\}} dvdx + \int_{\T^3} \int_{\mathbb{R}^3}\frac{\sqrt{\mu}}{4}|f| \mathbf{1}_{\{|f| > \sqrt{\mu}\}} dvdx \le \mathcal{E}(F_0)
	\end{align*}
	for all $t \ge 0$.
\end{Lem}

\section{Small perturbation problem in $L^p_vL^\infty_x$}
\label{section4}

From Section \ref{section4} to Section \ref{section5}, we fix $p>0$ and $\beta>0$ which satisfy the following conditions :
\begin{align*} 
		\begin{cases}
			p>13, \quad \beta >\frac{9}{2} \quad &\text{if } -1 \le \gamma < 0,\\
			p>\frac{3 \sqrt{8\gamma^2+9}+3-2\gamma}{-\gamma^2-3\gamma}, \quad \beta>5 \quad &\text{if } -3 < \gamma < -1,
		\end{cases}
	\end{align*}
and we also fix $0<q<1$ and $0\le \vt <-\frac{2}{\gamma}$. Let $f$ satisfy the equation \eqref{FPBER} with initial data $f_0(x,v)$ and $F(t,x,v) = \mu(v)+\mu(v)^{1/2}f(t,x,v) \ge 0$ over the time interval $[0,T]$ for $0<T\le \infty$.
	
Throughout this section, we make the a priori assumption:
\begin{align} \label{apriorismall}
	\sup_{0\le t < T}\|w_{q,\vt,\beta}f(t)\|_{L^p_vL^\infty_x} \le \bar{\eta} \ll 1,
\end{align}
where $\bar{\eta}$ depends on the initial amplitude $\eta_0>0$ with $\|w_{q,\vt,\beta}f_0\|_{L^p_vL^\infty_x} \le \eta_0$, but does not depend on the solution $f$. It will be determined in subsection \ref{smallglobalproof}.

\subsection{$L^2$ estimate}
In this subsection, our aim is to prove the $L^2$ estimate to the Boltzmann equation \eqref{FPBER}. To derive the exponential time-decay in $L^p_vL^\infty_x$ for the equation \eqref{def.be.h}, we need to consider the exponential $L^2$ decay. We define the $L^2_{v}$ projection $P$ of $f$ corresponding to operator $L$ as
\begin{align*} 
	Pf(t,x,v) = a(t,x) \mu^{1/2}(v) + b(t,x)\cdot v\mu^{1/2}(v) + c(t,x)\frac{|v|^2-3}{\sqrt{6}}\mu^{1/2}(v),
\end{align*}
where
\begin{align*}
	& a(t,x) = \int_{\mathbb{R}^3}f(t,x,v) \mu^{1/2}(v)dv,\\
	& b(t,x) = \int_{\mathbb{R}^3}vf(t,x,v) \mu^{1/2}(v)dv,\\
	& c(t,x) = \int_{\mathbb{R}^3}\frac{|v|^2-3}{\sqrt{6}}f(t,x,v) \mu^{1/2}(v)dv.
\end{align*}
It is well-known the operator $L$ satisfies the $L^2$ coercivity $(Lf,f)_{L_v^2} \ge C_L\|(I-P)f\|_{L_v^2(\nu^{1/2})}^2$ for all $f$ in $L_v^2$, where $C_L$ is a generic constant. 
\newline
\indent The following lemma states the $L^2_{x,v}$ bound for $Pf$ by $(I-P)f$. The lemma gives the key estimate to derive the exponential decay in $L^2_{x,v}$.
We consider the following equation : 
\begin{align} \label{LBEwg}
	\p_t f + v\cdot \nabla_x f +Lf =g,
\end{align}
where $g$ is a given function which satisfies the condition
\begin{align} \label{gcondition}
	\int_{\T^3} \int_{\R^3} g\mu^{1/2}dvdx=0, \quad \int_{\T^3} \int_{\R^3} vg\mu^{1/2}dvdx=0, \quad \int_{\T^3} \int_{\R^3} \frac{|v|^2-3}{\sqrt{6}}g\mu^{1/2}dvdx=0.
\end{align}
\begin{Lem}  \cite{Liu2017ARMA}
	Assume that $f\in L^2_{x,v}$ is a solution to the equation \eqref{LBEwg} satisfying
	\begin{align} \label{conserforf}
		\iint_{\T^3 \times \R^3}f \mu^{1/2}  \begin{pmatrix}
			1 \\ v \\ |v|^2 
		\end{pmatrix}dvdx = 0 
	\end{align}  
	and $g$ satisfies \eqref{gcondition}. There exists a function $G(t)$ such that for all $0\le s \le t$, $G(s) \lesssim \|f(s)\|_{L^2_{x,v}}^2$ and
	\begin{align*}
		\int_s^t \|Pf(\tau)\|_{L^2_{x,v}(\nu^{1/2})}^2d\tau &\lesssim G(t)-G(s) + \int_s^t\|(I-P)f(\tau)\|_{L^2_{x,v}(\nu^{1/2})}^2 d\tau+\int_s^t \|\nu^{-1/2}g(\tau)\|_{L^2_{x,v}}^2d\tau.
	\end{align*}
\end{Lem} 
\bigskip

\begin{Coro} \cite{Liu2017ARMA}
	Assume that $f\in L^2_{x,v}$ is a solution to the equation \eqref{LBEwg} satisfying \eqref{conserforf} and $g$ satisfies \eqref{gcondition}. Then we have
	\begin{align} \label{L2est2}
		\|f(t)\|_{L^2_{x,v}}^2+\int_0^t \|f(\tau)\|_{L^2_{x,v}(\nu^{1/2})}^2d\tau &\lesssim \|f_0\|_{L^2_{x,v}}^2+ \int_0^t \|\nu^{-1/2}g(\tau)\|_{L^2_{x,v}}^2d\tau.
	\end{align}

\begin{remark}
	The estimate \eqref{L2est2} is derived from the inequality
	\begin{align} \label{L2est3}
		\frac{d}{d\tau} \|f(\tau)\|_{L^2_{x,v}}^2 + \|f(\tau)\|_{L^2_{x,v}(\nu^{1/2})}^2 \lesssim \|\nu^{-1/2}g(\tau)\|_{L^2_{x,v}}^2.
	\end{align} 
\end{remark}

\end{Coro}

\bigskip

For simplicity, we abbreviate the weight functions
\begin{align*}
	w_{q}(v)=w_{q,\infty,0}(t,v):=e^{\frac{q}{8}|v|^2} \quad \text{and} \quad w_{q,\beta}(v)=w_{q,\infty,\beta}(t,v):=(1+|v|^2)^{\beta/2}e^{\frac{q}{8}|v|^2}.
\end{align*}

\bigskip

\begin{Lem} \label{L2estw} \cite{Liu2017ARMA}
	Assume that $f\in L^2_{x,v}(w_{q/2,\beta})$ is a solution to the equation \eqref{LBEwg} satisfying \eqref{conserforf} and $g \in L^2_{x,v}(w_{q/2,\beta})$ satifies \eqref{gcondition}. Then we have
	\begin{align*} 
		\sup_{0\le s \le t}\|w_{q/2,\beta}f(s)\|_{L^2_{x,v}}^2+\int_0^t \|w_{q/2,\beta}f(\tau)\|_{L^2_{x,v}(\nu^{1/2})}^2d\tau &\le C_w\|w_{q/2,\beta}f_0\|_{L^2_{x,v}}^2 \nonumber \\
		&\quad + C_w\int_0^t \|\nu^{-1/2}w_{q/2,\beta}g(\tau)\|_{L^2_{x,v}}^2d\tau,
	\end{align*}
	where $C_w \ge 1$ is a generic contant.
\end{Lem}
By taking $g:=\Gamma(f,f)$, we consider the $L^2_{x,v}$ estimate for the full perturbed Boltzmann equation. As a step toward goal in this section, under the a priori assumpition \eqref{apriorismall}, we aim to derive
\begin{align} \label{L2est10}
	\sup_{0\le s \le t} \|w_{q/2,\beta}f(s)\|_{L^2_{x,v}}^2+\int_0^t \|w_{q/2,\beta}f(\tau)\|_{L^2_{x,v}(\nu^{1/2})}^2d\tau &\le C_{p,q,\gamma} \|w_{q/2,\beta}f_0\|_{L^2_{x,v}}^2.
\end{align} 
To derive \eqref{L2est10}, we first claim that
\begin{align} \label{L2Gamma2}
	\int_0^t \|\nu^{-1/2}w_{q/2,\beta}\Gamma(f,f)(\tau)\|_{L^2_{x,v}}^2 d\tau \le  C_{\Gamma1}\sup_{0\le \tau \le t}\|w_{q,\vt,\beta}f(\tau)\|_{L^p_v L^\infty_x}^2\int_0^t \|w_{q/2,\beta}f(\tau)\|_{L^2_{x,v}(\nu^{1/2})}^2 d\tau,
\end{align}
where $C_{\Gamma1}$ depends on $p,q,\gamma,\beta$. In order to derive \eqref{L2Gamma2}, let us estimate
\begin{align*}
	\int_0^t \|\nu^{-1/2}w_{q/2,\beta}\Gamma^-(f,f)(\tau)\|_{L^2_{x,v}}^2 d\tau
\end{align*}
By the definition of $\Gamma^-$, the above term is bounded by
\begin{align} \label{wGamma-L2}
	&\int_0^t \|\nu^{-1/2}w_{q/2,\beta}\Gamma^-(f,f)(\tau)\|_{L^2_{x,v}}^2 d\tau \nonumber \\ \nonumber
	& \le C \int_0^t  \int_{\T^3 \times \R^3}\nu(v)^{-1} w_{q/2,\beta}(v)^2\left( \int_{\R^3}|v-u|^\gamma \mu^{1/2}(u) |f(u)||f(v)|du \right)^2 dxdvd\tau \\\nonumber
	&\le C \int_0^t \int_{\T^3 \times \R^3}\nu(v)^{-1} w_{q/2,\beta}(v)^2|f(v)|^2 \left( \int_{\R^3}|v-u|^{p'\gamma} \mu^{1/2}(u) du \right)^{2/p'} \\\nonumber
	& \quad \times \left( \int_{\R^3}\mu^{1/2}(u) |f(u)|^pdu \right)^{2/p}dxdv d\tau\\\nonumber
	& \le C_p \int_0^t\int_{\T^3 \times \R^3}\nu(v)^{-1} w_{q/2,\beta}(v)^2|\nu(v) f(v)|^2 \left(\int_{\R^3}\mu^{1/2}(u)|f(u)|^p du\right)^{2/p}dxdvd\tau\\\nonumber
	& \le C_p \int_0^t \|w_{q,\vt,\beta}f(\tau)\|_{L^p_v L^\infty_x}^2 \| w_{q/2,\beta}f(\tau)\|_{L^2_{x,v}(\nu^{1/2})}^2d\tau\\ 
	& \le C_p \sup_{0\le \tau \le t}\|w_{q,\vt,\beta}f(\tau)\|_{L^p_v L^\infty_x}^2\int_0^t \| w_{q/2,\beta}f(\tau)\|_{L^2_{x,v}(\nu^{1/2})}^2 d\tau, 
\end{align}
where $\beta>0$.\\
Now, let us estimate the following time integration for the term $\Gamma^+$ :
\begin{align*}
	\int_0^t \|\nu^{-1/2}w_{q/2,\beta}\Gamma^+(f,f)(\tau)\|_{L^2_{x,v}}^2 d\tau.
\end{align*}
Using the definition for $\Gamma^+$ and the Cauchy-Schwarz inequality, we have
\begin{align*}
	&\|\nu^{-1/2}w_{q/2,\beta}\Gamma^+(f,f)(\tau)\|_{L^2_{x,v}}^2\\ & \le \int_{\T^3} \int_{\R^3} \nu^{-1}(v)w_{q/2,\beta}(v)^2 \left(\int_{\R^3}\int_{\S^2} b(\cos \theta) |u-v|^\gamma  \mu^{1/2}(u)f(v') f(u')d\omega du\right)^2dvdx\\
	& \le C_b\int_{\T^3} \int_{\R^3} \nu^{-1}(v) w_{q/2,\beta}(v)^2\left(\int_{\R^3}\int_{\S^2} b(\cos \theta) |u-v|^\gamma  \mu^{1/2}(u)|f(v')|^2 |f(u')|^2d\omega du\right)\\
	& \quad \times \left(\int_{\R^3}|u-v|^\gamma  \mu^{1/2}(u)du\right)dvdx\\
	& \le C_b \int_{\T^3} \int_{\R^3}  w_{q/2,\beta}(v)^2\left(\int_{\R^3}\int_{\S^2} b(\cos \theta) |u-v|^\gamma  \mu^{1/2}(u)|f(v')|^2 |f(u')|^2d\omega du\right)dvdx.
\end{align*}
By the energy conservation, we have
\begin{align*}
	&\|\nu^{-1/2}w_{q/2,\beta}\Gamma^+(f,f)(\tau)\|_{L^2_{x,v}}^2\\
	& \le C_b \int_{\T^3} \int_{\R^3} \left(\int_{\R^3}\int_{\S^2} b(\cos \theta) |u-v|^\gamma  \mu^{1/2}(u')|w_{q/2,\beta}(v')f(v')|^2 |w_{q/2,\beta}(u')f(u')|^2d\omega du\right)dvdx.
\end{align*}
We make a change of variables $(u,v) \mapsto (u',v')$ to derive
\begin{align*}
	&\|\nu^{-1/2}w_{q/2,\beta}\Gamma^+(f,f)(\tau)\|_{L^2_{x,v}}^2\\ &\le C_b \int_{\T^3} \int_{\R^3} |w_{q/2,\beta}(v)f(v)|^2 \int_{\R^3}\int_{\S^2} b(\cos \theta) |u-v|^\gamma  \mu^{1/2}(u') |w_{q/2,\beta}(u)f(u)|^2d\omega dudvdx\\
	& \le C_b \int_{\T^3} \int_{\R^3} |w_{q/2,\beta}(v)f(v)|^2 \left( \int_{\R^3}\int_{\S^2} b(\cos \theta)  |w_{q,\vt,\beta}f(u)|^pd\omega du\right)^{2/p}\\
	& \quad \times \left(\int_{\R^3}\int_{\S^2} b(\cos \theta) |u-v|^{\frac{p}{p-2}\gamma}  \mu(u')^{\frac{p}{2(p-2)}} \left(\frac{1}{w_{q/2}(u)}\right)^{\frac{p}{p-2}}d\omega du\right)^{\frac{p-2}{p}}dvdx\\
	& \le C_{p,q,\gamma}\|w_{q,\vt,\beta}f(\tau)\|_{L^p_vL^\infty_x}^2\int_{\T^3} \int_{\R^3} \nu(v)|w_{q/2,\beta}(v)f(v)|^2 dvdx\\
	& \le C_{p,q,\gamma}\|w_{q,\vt,\beta}f(\tau)\|_{L^p_vL^\infty_x}^2\|w_{q/2,\beta}f(\tau)\|_{L^2_{x,v}(\nu^{1/2})}^2,
\end{align*}
where $p>\frac{6}{3+\gamma}$ and we have used the \eqref{Lpchangeofvariable}. Hence we can derive
\begin{align} \label{wGamma+L2}
	&\int^t_0 \|\nu^{-1/2}w_{q/2,\beta}\Gamma^+(f,f)(\tau)\|_{L^2_{x,v}}^2d\tau \nonumber \\
	&\le C_{p, q,\gamma,\beta}\int_0^t   \|w_{q,\vt,\beta}f(\tau)\|_{L^p_vL^\infty_x}^2\|w_{q/2,\beta}f(\tau)\|_{L^2_{x,v}(\nu^{1/2})}^2 d\tau\nonumber\\
	& \le  C_{p,q, \gamma ,\beta}\sup_{0\le \tau \le t}\|w_{q,\vt,\beta}f(\tau)\|_{L^p_v L^\infty_x}^2\int_0^t \|w_{q/2,\beta}f(\tau)\|_{L^2_{x,v}(\nu^{1/2})}^2 d\tau.
\end{align}
Combining \eqref{wGamma-L2} and \eqref{wGamma+L2}, we can deduce \eqref{L2Gamma2}.\\
Next, from  Lemma \ref{L2estw} and \eqref{L2Gamma2}, we have
\begin{align*}
	&\sup_{0\le s \le t}\|w_{q/2,\beta}f(s)\|_{L^2_{x,v}}^2+\int_0^t \|w_{q/2,\beta}f(\tau)\|_{L^2_{x,v}(\nu^{1/2})}^2d\tau \\
	&\le C_w\|w_{q/2,\beta}f_0\|_{L^2_{x,v}}^2
	+ C_wC_{\Gamma1}\sup_{0\le \tau \le t}\|w_{q,\vt,\beta}f(\tau)\|_{L^p_v L^\infty_x}^2\int_0^t \|w_{q/2,\beta}f(\tau)\|_{L^2_{x,v}(\nu^{1/2})}^2 d\tau.
\end{align*}
Under the a priori assumption \eqref{apriorismall}, it holds that 
\begin{align}\label{cond1apriori}
	C_wC_{\Gamma1}\bar{\eta}^2\le 1/2,
\end{align}
and then
\begin{align} \label{L2a100}
	\sup_{0\le s \le t}\|w_{q/2,\beta}f(s)\|_{L^2_{x,v}}^2+\int_0^t \|w_{q/2,\beta}f(\tau)\|_{L^2_{x,v}(\nu^{1/2})}^2d\tau \le 2C_w\|w_{q/2,\beta}f_0\|_{L^2_{x,v}}^2.
\end{align}
\bigskip

Multiplying $e^{\lambda (1+\tau)^\rho}$ to \eqref{L2est3} and integrating the inequality with respect to the time variable over $[0,t]$, we have
\begin{align} \label{L2est5} \nonumber
	&e^{\lambda (1+t)^\rho}\|f(t)\|_{L^2_{x,v}}^2 + \int_0^t e^{\lambda (1+\tau)^\rho}\|f(\tau)\|_{L^2_{x,v}(\nu^{1/2})}^2 d\tau \nonumber \\
	&\le C \|f_0\|_{L^2_{x,v}}^2 + C \int_0^t \lambda \rho (1+\tau)^{\rho-1}e^{\lambda (1+\tau)^\rho}\|f(\tau)\|_{L^2_{x,v}}^2 d\tau \nonumber \\
	&\quad + C\int_0^t e^{\lambda (1+\tau)^\rho}\| \nu^{-1/2} \Gamma(f,f)(\tau)\|_{L^2_{x,v}}^2d\tau.
\end{align}
Now, we need to estimate the term in \eqref{L2est5}:
\begin{align} \label{L2a1}
	\int_0^t \lambda \rho (1+\tau)^{\rho-1}e^{\lambda (1+\tau)^\rho}\|f(\tau)\|_{L^2_{x,v}}^2 d\tau.
\end{align}
To estimate \eqref{L2a1}, we set
\begin{align*}
	E:= \left\{ v \in \R^3 : (1+\tau)^{\rho-1} \le \kappa_0 (1+|v|^2)^{\gamma/2}\right\},
\end{align*}
where $\kappa_0$ is chosen later. On $E$, it holds that
\begin{align*}
	\lambda \rho (1+\tau)^{\rho-1} e^{\lambda (1+\tau)^\rho}  \le C \kappa_0 \nu(v)\lambda \rho  e^{\lambda (1+\tau)^\rho} ,
\end{align*}
which implies that
\begin{align} \label{L2a2}
	\int_0^t \lambda \rho (1+\tau)^{\rho-1}e^{\lambda (1+\tau)^\rho}\|f(\tau)\mathbf{1}_{E} \|_{L^2_{x,v}}^2 d\tau \le C \lambda\kappa_0 \rho \int_0^t \|e^{\lambda(1+\tau)^\rho}f(\tau)\|_{L^2_{x,v}(\nu^{1/2})}^2d\tau.
\end{align}
On the other hands, on $E^c$, we have
\begin{align*}
	\lambda(1+\tau)^\rho \le  \lambda \kappa_0^{\frac{\rho}{\rho-1}} (1+|v|^2)^{\frac{\rho\gamma}{2(\rho-1)}},
\end{align*}
and it follows that
\begin{align} \label{L2a3} \nonumber
	&\int_0^t \lambda \rho (1+\tau)^{\rho-1}e^{\lambda (1+\tau)^\rho}\|f(\tau)\mathbf{1}_{E^c} \|_{L^2_{x,v}}^2 d\tau \\ \nonumber 
	&\le C  \int_0^t \lambda \rho (1+\tau)^{\rho-1}e^{-\lambda (1+\tau)^\rho} \left\|e^{\lambda(1+\tau)^\rho}f(\tau)\mathbf{1}_{E^c}\right\|_{L^2_{x,v}}^2d\tau\\
	& \le C \int_0^t \lambda \rho (1+\tau)^{\rho-1}e^{-\lambda (1+\tau)^\rho} \left\|\exp\left\{\lambda\kappa_0^{\frac{\rho}{\rho-1}}(1+|v|^2)^{\frac{\rho\gamma}{2(\rho-1)}} \right\}f(\tau)\mathbf{1}_{E^c}\right\|_{L^2_{x,v}}^2d\tau.
\end{align}
Here, $\frac{\rho\gamma}{2(\rho-1)} \le 1$ holds since $\rho = 1+\frac{(1+\vt)\gamma}{2-\gamma}$, which implies that
\begin{align*}
	\lambda\kappa_0^{\frac{\rho}{\rho-1}}(1+|v|^2)^{\frac{\rho\gamma}{2(\rho-1)}} \le \lambda\kappa_0^{\frac{\rho}{\rho-1}} (1+|v|^2) \le \frac{q}{16}(1+|v|^2).
\end{align*}
for a suitable $\lambda>0$. Note that $\lambda>0$ depends on $q$.
Thus \eqref{L2a3} becomes
\begin{align} \label{L2a4} \nonumber
	&\int_0^t \lambda \rho (1+\tau)^{\rho-1}e^{\lambda (1+\tau)^\rho}\|f(\tau)\mathbf{1}_{E^c} \|_{L^2_{x,v}}^2 d\tau \nonumber \\
	&\le C_q \int_0^t \lambda \rho (1+\tau)^{\rho-1}e^{-\lambda (1+\tau)^\rho} \left\|e^{\frac{q}{16}|v|^2}f(\tau)\mathbf{1}_{E^c}\right\|_{L^2_{x,v}}^2d\tau \nonumber \\
	& \le C_q \sup_{0\le \tau \le t} \|w_{q/2,\beta}f(\tau)\|_{L^2_{x.v}}^2,
\end{align}
where $\int_0^t \lambda \rho (1+\tau)^{\rho-1}e^{-\lambda (1+\tau)^\rho} d\tau \lesssim 1$. From \eqref{L2a100}, \eqref{L2a4} implies that
\begin{align} \label{L2a5}
	\int_0^t \lambda \rho (1+\tau)^{\rho-1}e^{\lambda (1+\tau)^\rho}\|f(\tau)\mathbf{1}_{E^c} \|_{L^2_{x,v}}^2 d\tau  \le 2C_qC_w\|w_{q/2,\beta}f_0\|_{L^2_{x,v}}^2.
\end{align}
Gathering \eqref{L2a2} and \eqref{L2a5}, we obtain
\begin{align} \label{L2a6}
	&\int_0^t \lambda \rho (1+\tau)^{\rho-1}e^{\lambda (1+\tau)^\rho}\|f(\tau) \|_{L^2_{x,v}}^2 d\tau \nonumber \\
	&\le C\lambda\kappa_0\rho\int_0^t \|e^{\lambda(1+\tau)^\rho}f(\tau)\|_{L^2_{x,v}(\nu^{1/2})}^2d\tau+ 2C_qC_w\|w_{q/2,\beta}f_0\|_{L^2_{x,v}}^2.
\end{align}
\bigskip

We need to estimate the time integration to the nonlinear term $\Gamma(f,f). $ Let $p>\frac{6}{3+\gamma}$ and $\beta > 0$. We claim that
\begin{align} \label{L2Gamma1}
	&\int_0^t e^{\lambda_2 (1+\tau)^{\rho}}\|\nu^{-1/2}\Gamma(f,f)(\tau)\|_{L^2_{x,v}}^2d\tau\nonumber \\ &\le C_{\Gamma2}  \sup_{0 \le \tau \le t}\|w_{q,\vt,\beta}f(\tau)\|_{L^p_vL^\infty_x}^2 \int_{0}^t e^{\lambda_2 (1+\tau)^{\rho}}\|f(\tau)\|_{L^2_{x,v}(\nu^{1/2})}^2 d\tau, 
\end{align}
where $C_{\Gamma2}$ depends on $p$ and $\gamma$.
First of all, let us estimate
\begin{align*}
	\int_0^t e^{\lambda_2 (1+\tau)^{\rho}} \|\nu^{-1/2}\Gamma^-(f,f)(\tau)\|_{L^2_{x,v}}^2d\tau.
\end{align*}
By the definition of $\Gamma^-$, the above term is bounded by
\begin{align} \label{Gamma-L2}
	&\int^t_0 e^{\lambda_2 (1+\tau)^{\rho}}\|\nu^{-1/2}\Gamma^-(f,f)(\tau)\|_{L^2_{x,v}}^2d\tau \nonumber \\ \nonumber
	& \le C \int_0^t e^{\lambda_2 (1+\tau)^{\rho}} \int_{\T^3 \times \R^3}\nu(v)^{-1} \left( \int_{\R^3}|v-u|^\gamma \mu^{1/2}(u) |f(u)||f(v)|du \right)^2 dxdvd\tau \\\nonumber
	&\le C \int_0^t e^{\lambda_2 (1+\tau)^{\rho}} \int_{\T^3 \times \R^3}\nu(v)^{-1}|f(v)|^2 \left( \int_{\R^3}|v-u|^{p'\gamma} \mu^{1/2}(u) du \right)^{2/p'} \\\nonumber
	& \quad \times \left( \int_{\R^3}\mu^{1/2}(u) |f(u)|^pdu \right)^{2/p}dxdv d\tau\\\nonumber
	& \le C_p \int_0^t e^{\lambda_2 (1+\tau)^{\rho}} \int_{\T^3 \times \R^3}\nu(v)^{-1}|\nu(v) f(v)|^2 \left(\int_{\R^3}\mu^{1/2}(u)|f(u)|^p du\right)^{2/p}dxdvd\tau\\\nonumber
	& \le C_p \int_0^t e^{\lambda_2 (1+\tau)^{\rho}} \|w_{q,\vt,\beta}f(\tau)\|_{L^p_v L^\infty_x}^2 \|f(\tau)\|_{L^2_{x,v}(\nu^{1/2})}^2d\tau\\ 
	& \le C_p \sup_{0\le \tau \le t}\|w_{q,\vt,\beta}f(\tau)\|_{L^p_v L^\infty_x}^2\int_0^t e^{\lambda_2 (1+\tau)^{\rho}}\|f(\tau)\|_{L^2_{x,v}(\nu^{1/2})}^2 d\tau, 
\end{align}
where $\beta>0$.\\
Next, let us estimate the following time integration for the term $\Gamma^+$ :
\begin{align*}
	\int_0^t e^{\lambda_2 (1+\tau)^{\rho}}\|\nu^{-1/2}\Gamma^+(f,f)(\tau)\|_{L^2_{x,v}}^2 d\tau.
\end{align*}
By the definition of $\Gamma^+$ and the Cauchy-Schwarz inequality, we have
\begin{align*}
	&\|\nu^{-1/2}\Gamma^+(f,f)(\tau)\|_{L^2_{x,v}}^2\\
	 & \le \int_{\T^3} \int_{\R^3} \nu^{-1}(v) \left(\int_{\R^3}\int_{\S^2} b(\cos \theta) |u-v|^\gamma  \mu^{1/2}(u)f(v') f(u')d\omega du\right)^2dvdx\\
	& \le C_b\int_{\T^3} \int_{\R^3} \nu^{-1}(v) \left(\int_{\R^3}\int_{\S^2} b(\cos \theta) |u-v|^\gamma  \mu^{1/2}(u)|f(v')|^2 |f(u')|^2d\omega du\right)\\
	& \quad \times \left(\int_{\R^3}|u-v|^\gamma  \mu^{1/2}(u)du\right)dvdx\\
	& \le C_b \int_{\T^3} \int_{\R^3}  \left(\int_{\R^3}\int_{\S^2} b(\cos \theta) |u-v|^\gamma  \mu^{1/2}(u)|f(v')|^2 |f(u')|^2d\omega du\right)dvdx.
\end{align*}
We make a change of variables $(u,v) \mapsto (u',v')$ to derive
\begin{align*}
	&\|\nu^{-1/2}\Gamma^+(f,f)(\tau)\|_{L^2_{x,v}}^2 \\&\le C_b \int_{\T^3} \int_{\R^3} |f(v)|^2 \int_{\R^3}\int_{\S^2} b(\cos \theta) |u-v|^\gamma  \mu^{1/2}(u') |f(u)|^2d\omega dudvdx\\
	& \le C_b \int_{\T^3} \int_{\R^3} |f(v)|^2 \left( \int_{\R^3}\int_{\S^2} b(\cos \theta)  |w_{q,\vt,\beta}f(u)|^pd\omega du\right)^{2/p}\\
	& \quad \times \left(\int_{\R^3}\int_{\S^2} b(\cos \theta) |u-v|^{\frac{p}{p-2}\gamma}  \mu(u')^{\frac{p}{2(p-2)}} \left(\frac{1}{w_{q,\vt,\beta}(u)}\right)^{\frac{p}{p-2}}d\omega du\right)^{\frac{p-2}{p}}dvdx\\
	& \le C_{p,q,\gamma}\|w_{q,\vt,\beta}f(\tau)\|_{L^p_vL^\infty_x}^2\int_{\T^3} \int_{\R^3} \nu(v)|f(v)|^2 dvdx\\
	& \le C_{p,q,\gamma}\|w_{q,\vt,\beta}f(\tau)\|_{L^p_vL^\infty_x}^2\|f(\tau)\|_{L^2_{x,v}(\nu^{1/2})}^2,
\end{align*}
where $p>\frac{6}{3+\gamma}$ and we have used \eqref{Lpchangeofvariable}. Hence we conclude that
\begin{align} \label{Gamma+L2}
	&\int_0^t e^{\lambda_2 (1+\tau)^{\rho}}\|\nu^{-1/2}\Gamma^+(f,f)(\tau)\|_{L^2_{x,v}}^2 d\tau \nonumber \\ 
	&\le C_{p,q,\gamma} \sup_{0\le \tau \le t}\|w_{q,\vt, \beta}f(\tau)\|_{L^p_v L^\infty_x}^2\int_0^t e^{\lambda_2 (1+\tau)^{\rho}}\|f(\tau)\|_{L^2_{x,v}(\nu^{1/2})}^2 d\tau.
\end{align}
Combining \eqref{Gamma-L2} and \eqref{Gamma+L2}, we can deduce \eqref{L2Gamma1}. 
\bigskip

As a final step, under the a priori assumpition \eqref{apriorismall}, we show that
\begin{align*}
	e^{\lambda (1+t)^\rho}\|f(t)\|_{L^2_{x,v}} \le C_{p,q,\gamma,\beta} \|w_{q,\beta}f_0\|_{L^p_vL^\infty_x}.
\end{align*}
Inserting \eqref{L2a6} and \eqref{L2Gamma1} into \eqref{L2a5}, it holds that 
\begin{align*} 
&e^{\lambda (1+t)^\rho}\|f(t)\|_{L^2_{x,v}}^2 + \int_0^t e^{\lambda (1+\tau)^\rho}\|f(\tau)\|_{L^2_{x,v}(\nu^{1/2})}^2 d\tau\nonumber\\
 &\le C \|f_0\|_{L^2_{x,v}}^2 + C\lambda\kappa_0\rho\int_0^t e^{\lambda(1+\tau)^\rho}\|f(\tau)\|_{L^2_{x,v}(\nu^{1/2})}^2d\tau+2C_qC_w\|w_{q/2,\beta}f_0\|_{L^2_{x,v}}^2\nonumber\\ 
	&\quad + C_{\Gamma2}  \sup_{0 \le \tau \le t}\|w_{q,\vt,\beta}f(\tau)\|_{L^p_vL^\infty_x}^2 \int_{0}^t e^{\lambda_2 (1+\tau)^{\rho}}\|f(\tau)\|_{L^2_{x,v}(\nu^{1/2})}^2 d\tau.	
\end{align*}
Under the apriori assumption \eqref{apriorismall}, it holds that
\begin{align} \label{cond2apriori}
C_{\Gamma2} \bar \eta ^2\le 1/4,
\end{align}
and we choose $\kappa_0>0$ sufficiently small so that
\begin{align*}
	C\lambda\kappa_0\rho \le 1/4,
\end{align*}
and thus we can deduce that
\begin{align*}
	e^{\lambda (1+t)^\rho}\|f(t)\|_{L^2_{x,v}}^2 + \int_0^t e^{\lambda (1+\tau)^\rho}\|f(\tau)\|_{L^2_{x,v}(\nu^{1/2})}^2 d\tau &\le 2C \|f_0\|_{L^2_{x,v}}^2 +4C_qC_w\|w_{q/2,\beta}f_0\|_{L^2_{x,v}}^2\\
	&\le C_{p,q,\gamma,\beta}\|w_{q,\beta}f_0\|_{L^p_vL^\infty_x}^2.
\end{align*}
Hence we can conclude the following lemma:
\begin{Lem} \label{L2decayunderapriori}
	Under the a priori assumption \eqref{apriorismall}, it holds that
	\begin{align*}
		e^{\frac{\lambda}{2} (1+t)^\rho}\|f(t)\|_{L^2_{x,v}} \le C_{p,q,\gamma,\beta} \|w_{q,\beta}f_0\|_{L^p_vL^\infty_x}
	\end{align*}
	for all $0\le t < T$.
\end{Lem}

\bigskip

\subsection{$R(f)$ estimate} 
We can rewrite the Boltzmann equation \eqref{FPBER} for $h=w_{q,\vartheta,\beta}f$ as

\begin{align}\label{def.be.h}
	\p_th(t,x,v)+v\cdot\nabla_xh(t,x,v) + \tilde{\nu}(t,v)h(t,x,v) = K_wh(t) + w\Gamma(f,f)(t),
\end{align}
where
\begin{align*} 
	\tilde{\nu}(t,v) := \nu(v) + \frac{\vartheta q|v|^2}{8(1+t)^{\vartheta+1}} .
\end{align*}
\indent For the $L^p_v L^\infty_x$ solution under the soft potential, the absence of a positive lower bound of the collision frequency $\nu(v)$ and the appearance of the local term from Lemma \ref{pointwiseGamma-estimate} make the order of integration 
\begin{align*}
	\int_0^t \int_{\R^3} e^{-\nu(v)(t-s)} \nu(v) |h(s,X(s),v)| dvds
\end{align*}
highly delicate. To overcome this difficulty, it is necessary to introduce the operator
\begin{align} \label{Rf}
	R(f) := \int_{\R^3\times \S^2} B(v-u,\omega) [\mu(u) + \sqrt{\mu(u)} f(u)] d\omega du + \frac{\vartheta q|v|^2}{8(1+t)^{\vartheta+1}}.
\end{align}
Therefore, in this section, let us consider the Boltzmann equation with the operator $R(f)$, namely
\begin{align*}
	\partial_t h + v \cdot \nabla_x h + R(f) h = K_w h + w \Gamma^+(f,f).
\end{align*}
While the time-dependent collision frequency $\tilde{\nu}(t,v)$ has a positive lower bound, we need to verify whether the newly defined operator $R(f)$ also admits a positive lower bound.  From the definition \eqref{Rf} of the operator $R(f)$, it holds that 
\begin{align} \label{Rf.esti1}
		R(f) &=\int_{\R^3\times \S^2} B(v-u,\omega) [\mu(u) + \sqrt{\mu(u)} f(u)] d\omega du + \frac{\vartheta q|v|^2}{8(1+t)^{\vartheta+1}}  \nonumber\\ 
		& \geq \nu(v) -C_2 \nu(v) \int_{\R^3} e^{-\frac{|u|^2}{8}} |f(t,x,u)| du + \frac{\vartheta q|v|^2}{8(1+t)^{\vartheta+1}},
\end{align} 
for some constant $C_2>0$. Therefore, to obtain a positive lower bound of the operator $R(f)$, we need to get the following estimate: 
\begin{align*}
	\int_{\R^3} e^{-\frac{|u|^2}{8}} |h(t,x,u)| du \leq \frac{1}{2C_2}. 
\end{align*}
In this lemma, we establish an estimate for the quantity above. 
\begin{Lem} \label{exp.h.esti}
	Let $-3<\gamma<0$. The following results can be derived for each case:
	\begin{enumerate}[label=(\arabic*)]
		\item If $-1\leq \gamma <0$, then it holds that 
		\begin{align*}
			&\int_{\R^3} e^{-\frac{|u|^2}{8}} |h(t,x,u)| du\\ &\leq e^{-\lambda (1+t)^\rho} \Vert h_0 \Vert_{L^p_v L^\infty _x} + C_{p,q,\gamma,\beta,\varepsilon} \delta \left[\sup_{0\leq s \leq t}\Vert h(s) \Vert_{L^p_vL^\infty_x} + \sup_{0 \leq s \leq t} \Vert h(s) \Vert_{L^p_vL^\infty_x}^2\right] \\
		&\quad + C_{p,\gamma} \varepsilon^{\gamma+\frac{3}{p'}}\sup_{0\leq s \leq t } \Vert h(s) \Vert_{L^p_vL^\infty_x} +  C_{\gamma}\left(\frac{1}{N^{-\gamma}}+\frac{1}{N^{\frac{3-\gamma}{2}}}\right) \sup_{0\leq s \leq t} \Vert h(s)\Vert_{L^p_v L^\infty_x} \\
		& \quad +C_{N,\delta} \mathcal{E}(F_0)^{1/2} + C_{p,N,\delta}\sup_{0\le s \le t}\|h(s)\|_{L^p_vL^\infty_x}^{\frac{p}{2p-2}}\mathcal{E}(F_0)^{\frac{p-2}{2p-2}}\\
		&\quad + \frac{C_{p,\gamma,\beta,\ell,\delta}}{N} \sup_{0\leq s \leq t} \Vert h(s) \Vert_{L^p_v L^\infty_x}^2 + C_{p,\gamma,\beta,N,\delta}\sup_{0\leq s \leq t} \Vert h(s)\Vert_{L^p_v L^\infty_x}\mathcal{E}(F_0)^{1/2} \\
		& \quad + C_{p,\gamma, \beta, N,\delta}\sup_{0\le s \le t}\|h(s)\|_{L^p_vL^\infty_x}^{\frac{3p-2}{2p-2}}\mathcal{E}(F_0)^{\frac{p-2}{2p-2}}+C_{\ell,N,\delta} \sup_{0\leq s \leq t}\|h(s)\|_{L^p_vL^\infty_x}^{2-\frac{1/\ell-1/p}{1/2-1/p}}\mathcal{E}(F_0)^{\frac{1}{2}\cdot\frac{1/\ell-1/p}{1/2-1/p}}\\
		& \quad + C_{\ell,N,\delta}\sup_{0\leq s \leq t}\|h(s)\|_{L^p_vL^\infty_x}^{2-\frac{1/\ell-1/p}{1-1/p}}\mathcal{E}(F_0)^{\frac{1/\ell-1/p}{1-1/p}},
		\end{align*}
		for all $(t,x) \in [0,T] \times \T^3$. \\
		\item If $-3< \gamma <-1$, then it holds that 
		\begin{align*}
			&\int_{\R^3} e^{-\frac{|u|^2}{8}} |h(t,x,u)| du\\ &\leq e^{-\lambda (1+t)^\rho} \Vert h_0 \Vert_{L^p_v L^\infty _x} + C_{p,q,\gamma,\beta,b,\varepsilon} \delta \left[\sup_{0\leq s \leq t}\Vert h(s) \Vert_{L^p_vL^\infty_x} + \sup_{0 \leq s \leq t} \Vert h(s) \Vert_{L^p_vL^\infty_x}^2\right] \\
		&\quad +  C_{p,\gamma} \varepsilon^{\gamma+\frac{3}{p'}}\sup_{0\leq s \leq t } \Vert h(s) \Vert_{L^p_vL^\infty_x} +  C_{\gamma}\left(\frac{1}{N^{-\gamma}}+\frac{1}{N^{\frac{3-\gamma}{2}}}\right) \sup_{0\leq s \leq t} \Vert h(s)\Vert_{L^p_v L^\infty_x} \\
		& \quad +C_{N,\delta} \mathcal{E}(F_0)^{1/2} + C_{p,N,\delta}\sup_{0\le s \le t}\|h(s)\|_{L^p_vL^\infty_x}^{\frac{p}{2p-2}}\mathcal{E}(F_0)^{\frac{p-2}{2p-2}}\\
		&\quad + \frac{C_{p,\gamma,\beta,\ell,\delta}}{N}  \sup_{0 \leq s \leq t} \Vert h(s) \Vert_{L^p_v L^\infty_x}^2 +C_{\ell,N,\delta} \sup_{0\leq s \leq t}\|h(s)\|_{L^p_vL^\infty_x}^{2-\frac{1/\ell-1/p}{1/2-1/p}}\mathcal{E}(F_0)^{\frac{1}{2}\cdot\frac{1/\ell-1/p}{1/2-1/p}}\nonumber\\
	& \quad  + C_{\ell,N,\delta}\sup_{0\leq s \leq t}\|h(s)\|_{L^p_vL^\infty_x}^{2-\frac{1/\ell-1/p}{1-1/p}}\mathcal{E}(F_0)^{\frac{1/\ell-1/p}{1-1/p}}\nonumber\\
		& \quad  +C_{p,\gamma,N,\delta} \sup_{0 \leq s \leq t}\|h(s)\|_{L^p_vL^\infty_x}^{1-\frac{1/p'm'-1/p}{1/2-1/p}}\mathcal{E}(F_0)^{\frac{1}{2}\cdot\frac{1/p'm'-1/p}{1/2-1/p}} \nonumber \\
		&\quad + C_{p,\gamma,N,\delta}\sup_{0 \leq s \leq t}\|h(s)\|_{L^p_vL^\infty_x}^{1-\frac{1/p'm'-1/p}{1-1/p}}\mathcal{E}(F_0)^{\frac{1/p'm'-1/p}{1-1/p}},
		\end{align*}
		for all $(t,x) \in [0,T] \times \T^3$. 
	\end{enumerate}
\end{Lem}
\begin{proof}
	Applying Duhamel's iteration to \eqref{def.be.h}, one obtains 
	\begin{align} \label{Rf.duhamel}
			&\int_{\R^3} e^{-\frac{|u|^2}{8}} |h(t,x,u)| du\nonumber \\
			&= \int_{|u|\geq N} e^{-\frac{|u|^2}{8}} |h(t,x,u)| du + \int_{|u| \leq N} e^{-\frac{|u|^2}{8}} |h(t,x,u)| du\nonumber\\
			&\leq \frac{C}{N}\sup_{0\leq s \leq t} \Vert h(s) \Vert_{L^p_v L^\infty_x}\nonumber \\
			&\quad +  \int_{|u| \leq N } e^{-\frac{|u|^2}{8}}e^{-\int_0^t 	\tilde{\nu}(\tau,u)d\tau} |h_0(x-ut,u)| du \nonumber\\
			&\quad + \int_{t-\delta}^t \int_{|u| \leq N} e^{-\frac{|u|^2}{8}}   e^{-\int_s^t 	\tilde{\nu}(\tau,u)d\tau} [|K_w^{ns} h| +|K_w^{s}h|+ |w\Gamma(f,f)|] (s,x-u(t-s),u)duds\nonumber \\
			&\quad +  \int_{0}^{t-\delta} \int_{|u| \leq N } e^{-\frac{|u|^2}{8}}   e^{-\int_s^t 	\tilde{\nu}(\tau,u)d\tau} |K_w^{s} h|(s,x-u(t-s),u)duds\nonumber\\
			&\quad +  \int_{0}^{t-\delta} \int_{|u| \leq N } e^{-\frac{|u|^2}{8}}  e^{-\int_s^t 	\tilde{\nu}(\tau,u)d\tau} |K_w^{ns} h|(s,x-u(t-s),u)duds\nonumber\\
			&\quad + \int_{0}^{t-\delta} \int_{|u| \leq N} e^{-\frac{|u|^2}{8}}  e^{-\int_s^t 	\tilde{\nu}(\tau,u)d\tau} |w\Gamma(f,f)|(s,x-u(t-s),u) duds \nonumber\\
			&:= I_1 +I_2 + I_3 +I_4 + I_5.
	\end{align}
	For $I_1$, we use the following fact 
	\begin{align*}
		\tilde{\nu}(u,t) \geq C(1+t)^{\frac{(1+\vartheta)\gamma}{2-\gamma}}, 
	\end{align*}
	for some generic constant $C$. We omit the proof for brevity and refer the reader to (3.5) in \cite{Ko2022JDE}. Thus, 
	\begin{align*} 
	e^{-\int_0^t 	\tilde{\nu}(\tau,u)d\tau} \leq e^{-C\int_0^t (1+\tau)^{\frac{(1+\vartheta)\gamma}{2-\gamma}}d\tau} =e^{-\lambda(1+t)^{\rho}},
	\end{align*}
	where $\rho = 1+ \frac{(1+\vartheta)\gamma}{2-\gamma}>0$ and $\lambda = \frac{C}{\rho}$. Therefore, it follows from H\"{o}lder's inequality that 
	\begin{align} \label{I1}
		I_1 \leq e^{-\lambda (1+t)^\rho} \Vert h_0 \Vert_{L^p_v L^\infty _x}. 
	\end{align}
	Let us consider $I_2$. Using Lemma \ref{Ksingular}, Lemma \ref{Knonsing}, Lemma \ref{pointwiseGamma-estimate} and Corollary \ref{LpGamma+est}, we easily check that 
	\begin{align} \label{I2} 
		I_2 \leq C_{p,q,\gamma,\beta,\varepsilon} \delta \left[\sup_{0\leq s \leq t}\Vert h(s) \Vert_{L^p_vL^\infty_x} + \sup_{0 \leq s \leq t} \Vert h(s) \Vert_{L^p_vL^\infty_x}^2\right].
	\end{align}
	As for $I_3$, we use Lemma \ref{Ksingular} to obtain the estimate
	\begin{align} \label{I3} 
			I_3 &\leq C_{p,\gamma} \varepsilon^{\gamma+\frac{3}{p'}}\sup_{0\leq s \leq t } \Vert h(s) \Vert_{L^p_vL^\infty_x} \int_0^{t-\delta} \int_{|u| \leq N} e^{-\frac{|u|^2}{8}}  e^{-\int_s^t 	\tilde{\nu}(\tau,u)d\tau}  \mu(u)^{\frac{1-q}{8}}  duds \nonumber \\
			& \leq C_{p,\gamma} \varepsilon^{\gamma+\frac{3}{p'}}\sup_{0\leq s \leq t } \Vert h(s) \Vert_{L^p_vL^\infty_x} \int_{|u| \leq N} e^{-\frac{|u|^2}{16}}\int_0^{t-\delta} e^{-\nu(u)(t-s)} \nu(u)ds  du \nonumber  \\
			& \leq C_{p,\gamma} \varepsilon^{\gamma+\frac{3}{p'}}\sup_{0\leq s \leq t } \Vert h(s) \Vert_{L^p_vL^\infty_x}.
	\end{align}
	For convenience, we denote $X(s):=x-u(t-s)$.
	In the case of $I_4$, using Lemma \ref{k1ker} and Lemma \ref{k2ker} gives the following estimate:
	\begin{align} \label{I4} 
			I_4 &= \int_0^{t-\delta} \int_{|u| \leq N} e^{-\frac{|u|^2}{8}}  e^{-\int_s^t 	\tilde{\nu}(\tau,u)d\tau}  \int_{|\eta| \geq 2N} k_w^{ns}(u,\eta) |h(s,X(s),\eta)| d\eta  duds \nonumber \\
			&\quad + \int_0^{t-\delta} \int_{|u| \leq N} e^{-\frac{|u|^2}{8}}  e^{-\int_s^t 	\tilde{\nu}(\tau,u)d\tau} \int_{|\eta| \leq 2N} k_w^{ns}(u,\eta) |h(s,X(s),\eta)| d\eta  duds \nonumber \\
			&\leq  C_{\gamma }\int_0^{t-\delta} \int_{|u| \leq N} e^{-\frac{|u|^2}{8}} e^{-\int_s^t 	\tilde{\nu}(\tau,u)d\tau} \nonumber \\
			&\quad \quad  \times \int_{|\eta| \geq 2N} \left[(1+|u|^2)^{\beta}|u-\eta|^{\gamma} e^{-\frac{1-q}{4}|u|^2}e^{-\frac{|\eta|^2}{4}}+ \frac{1}{|u-\eta|^{\frac{3-\gamma}{2}}} e^{-\frac{|u-\eta|^2}{8}}e^{-\frac{||u|^2-|\eta|^2|^2}{8|u-\eta|^2}}\frac{w(u)}{w(\eta)}\right]\nonumber\\
			& \qquad \quad \times  |h(s,X(s),\eta)|d\eta du ds  \nonumber \\
			&\quad +  \int_0^{t-\delta} \int_{|u| \leq N} e^{-\frac{|u|^2}{8}}  e^{-\int_s^t 	\tilde{\nu}(\tau,u)d\tau}  \int_{|\eta| \leq 2N} k_w^{ns}(u,\eta)|h(s,X(s),\eta)| d\eta  duds \nonumber \\
			&\leq C_{\gamma}\left(\frac{1}{N^{-\gamma}}+\frac{1}{N^{\frac{3-\gamma}{2}}}\right) \sup_{0\leq s \leq t} \Vert h(s)\Vert_{L^p_v L^\infty_x} \int_0^{t-\delta} \int_{|u| \leq N}  e^{-\int_s^t 	\tilde{\nu}(\tau,u)d\tau} \nu(u)e^{-\frac{|u|^2}{16}}du ds \nonumber\\
			&\quad +  C_\ve \int_0^{t-\delta} \int_{|u| \leq N} e^{-\frac{|u|^2}{16}}  e^{-\int_s^t 	\tilde{\nu}(\tau,u)d\tau}  \nu(u)    \left(\int_{|\eta|\leq 2N} |h(s,X(s),\eta)|^2 d\eta \right)^{1/2} duds   \nonumber \\
			&\leq C_{\gamma}\left(\frac{1}{N^{-\gamma}}+\frac{1}{N^{\frac{3-\gamma}{2}}}\right) \sup_{0\leq s \leq t} \Vert h(s)\Vert_{L^p_v L^\infty_x}   \nonumber \\
			&\quad +C_\ve   \sup_{0\leq s \leq t-\delta} \left(\int_{|u|\leq N} \int_{|\eta|\leq 2N} |h(s,X(s),\eta)|^2 dud\eta \right)^{1/2}.
	\end{align}
	Making a change of variables $u \mapsto y:=X(s) = x-u(t-s)$ with $\left|\frac{dy}{du} \right| =(t-s)^3$, the last term in \eqref{I4} becomes
	\begin{align*}
		 C_{N,\delta} \sup_{0\leq s \leq t-\delta} \left(\int_{\T^3} \int_{|\eta|\leq 2N} |h(s,y,\eta)|^2 d\eta dy \right)^{1/2}.
	\end{align*}
	Here, we use Lemma \ref{L1L2cont} and the interpolation inequality to obtain
\begin{equation} \label{RTIRE}
\begin{aligned}
	&\left(\int_{\T^3}\int_{|\eta|\le 2N}|h(s,y,\eta)|^2d\eta dy\right)^{1/2}\\
	& \le C_N \left(\int_{\T^3}\int_{|\eta|\le 2N}|f(s,y,\eta)|^2 \mathbf{1}_{\{|f|\le \sqrt{\mu}\}} d\eta dy\right)^{1/2}\\
	&\quad +C_N\left(\int_{\T^3}\int_{|\eta|\le 2N}|h(s,y,\eta)|^2 \mathbf{1}_{\{|f|> \sqrt{\mu}\}}d\eta dy\right)^{1/2}\\
	&\le C_N \mathcal{E}(F_0)^{1/2} + C_N\left(\int_{\T^3}\int_{|\eta|\le 2N}|h(s,y,\eta)| \mathbf{1}_{\{|f|> \sqrt{\mu}\}}d\eta dy\right)^{\frac{p-2}{2p-2}} \\ &\qquad \times \left(\int_{\T^3}\int_{|\eta|\le 2N}|h(s,y,\eta )|^p \mathbf{1}_{\{|f|> \sqrt{\mu}\}}d\eta dy\right)^{\frac{1}{2p-2}}\\
	&\le C_N \mathcal{E}(F_0)^{1/2} + C_{p,N}\|h(s)\|_{L^p_vL^\infty_x}^{\frac{p}{2p-2}}\left(\int_{\T^3}\int_{|\eta|\le 2N}\mu^{1/2}(\eta)|f(s,y,\eta)| \mathbf{1}_{\{|f|> \sqrt{\mu}\}}d\eta dy\right)^{\frac{p-2}{2p-2}}\\
	&\le C_N \mathcal{E}(F_0)^{1/2} + C_{p,N}\|h(s)\|_{L^p_vL^\infty_x}^{\frac{p}{2p-2}}\mathcal{E}(F_0)^{\frac{p-2}{2p-2}}.
\end{aligned}
\end{equation}
Hence we have
\begin{align} \label{Rf17}
	I_4 &\le  C_{\gamma}\left(\frac{1}{N^{-\gamma}}+\frac{1}{N^{\frac{3-\gamma}{2}}}\right) \sup_{0\leq s \leq t} \Vert h(s)\Vert_{L^p_v L^\infty_x}+C_{N,\delta} \mathcal{E}(F_0)^{1/2} \nonumber \\
	&\quad + C_{p,N,\delta}\sup_{0\le s \le t}\|h(s)\|_{L^p_vL^\infty_x}^{\frac{p}{2p-2}}\mathcal{E}(F_0)^{\frac{p-2}{2p-2}}.
\end{align}

	Let us handle the remaining term $I_5$. We firstly treat the case $-1 \leq \gamma <0$.  \\
	\newline 
	\indent \textbf{(Case 1. $-1 \leq \gamma <0$)}
	By using Lemma \ref{pointwiseGamma-estimate} and Lemma \ref{pointwiseGamma+estimate}, we obtain 
	\begin{align*}
		I_5 &\leq C_{p,\gamma,\beta}\sup_{0\leq s \leq t} \Vert h(s) \Vert_{L^p_v L^\infty_x } \int_0^{t-\delta} \int_{|u|\leq N} e^{-\frac{|u|^2}{8}} e^{-\int_s^t 	\tilde{\nu}(\tau,u)d\tau} \\
		&\qquad \times \left(\int_{\R^3} (1+|\eta|)^{\frac{4p-8}{p}-2\beta} |h(s,X(s),\eta)|^2 d\eta \right)^{1/2}duds\\
		&\quad + C_\ell \int_0^{t-\delta} \int_{|u|\leq N} e^{-\frac{|u|^2}{8}}  e^{-\int_s^t 	\tilde{\nu}(\tau,u)d\tau}  \nu(u) |h(s,X(s),u)| \\
		&\qquad \times \left(\int_{\R^3} \sqrt{\mu(\eta)} |h(s,X(s),\eta)|^\ell d \eta \right)^{1/\ell}duds\\
		&:=I_{5,1} + I_{5,2}.
	\end{align*}
	For $I_{5,1}$, We split the domain of integration in the $\eta$-velocity variable as follows:
	\begin{align} \label{betacondition1}
		I_{5,1} &\le C_{p,\gamma,\beta}\sup_{0\leq s \leq t} \Vert h(s) \Vert_{L^p_v L^\infty_x } \int_0^{t-\delta} \int_{|u|\leq N} e^{-\frac{|u|^2}{8}} e^{-\int_s^t 	\tilde{\nu}(\tau,u)d\tau}\nonumber\\
		&\qquad \times \left(\int_{|\eta| \geq 2N} (1+|\eta|)^{\frac{4p-8}{p}-2\beta} |h(s,X(s),\eta)|^2 d\eta \right)^{1/2}duds \nonumber\\
		&\quad +C_{p,\gamma,\beta}\sup_{0\leq s \leq t} \Vert h(s) \Vert_{L^p_v L^\infty_x } \int_0^{t-\delta} \int_{|u|\leq N} e^{-\frac{|u|^2}{8}} e^{-\int_s^t 	\tilde{\nu}(\tau,u)d\tau} \nonumber\\
		&\qquad \times \left(\int_{|\eta| \leq 2N} (1+|\eta|)^{\frac{4p-8}{p}-2\beta} |h(s,X(s),\eta)|^2 d\eta \right)^{1/2}duds\nonumber\\
		&\leq C_{p,\gamma,\beta} \sup_{0\leq s \leq t} \Vert h(s) \Vert_{L^p_v L^\infty_x}^2 \int_0^{t-\delta} \int_{|u|\leq N} e^{-\frac{|u|^2}{8}}  e^{-\int_s^t 	\tilde{\nu}(\tau,u)d\tau} \nonumber \\
		&\qquad \times \left(\int_{|\eta| \geq 2N} (1+|\eta|)^{4-\frac{2p\beta}{p-2}} d\eta \right)^{\frac{p-2}{2p}}duds\nonumber \\
		&\quad +C_{p,\gamma,\beta}\sup_{0\leq s \leq t} \Vert h(s) \Vert_{L^p_v L^\infty_x } \int_0^{t-\delta} \int_{|u|\leq N} e^{-\frac{|u|^2}{16}} e^{-\int_s^t 	\tilde{\nu}(\tau,u)d\tau} \nu(u)\nonumber\\
		&\qquad \times \left(\int_{|\eta| \leq 2N} |h(s,X(s),\eta)|^2 d\eta \right)^{1/2}duds\nonumber\\
		&\leq \frac{C_{p,\gamma,\beta}}{N} \sup_{0\leq s \leq t} \Vert h(s) \Vert_{L^p_v L^\infty_x}^2 \int_0^{t-\delta} \int_{|u|\leq N} e^{-\frac{|u|^2}{8}}  e^{-\int_s^t 	\tilde{\nu}(\tau,u)d\tau}\nonumber \\
		&\qquad \times  \left(\int_{|\eta| \geq 2N} (1+|\eta|)^{4-\frac{2p(\beta-1)}{p-2}} d\eta \right)^{\frac{p-2}{2p}}duds\\
		&\quad + C_{p,\gamma,\beta}\sup_{0\leq s \leq t} \Vert h(s) \Vert_{L^p_v L^\infty_x } \sup_{0\leq s \leq t-\delta} \left(\int_{|u|\leq N} \int_{|\eta| \leq 2N} |h(s,X(s),\eta)|^2 dud\eta \right)^{1/2}\nonumber\\
		&\leq \frac{C_{p,\gamma,\beta}}{N} \sup_{0\leq s \leq t} \Vert h(s) \Vert_{L^p_v L^\infty_x}^2 + C_{p,\gamma,\beta,N,\delta}\sup_{0\leq s \leq t} \Vert h(s)\Vert_{L^p_v L^\infty_x}\mathcal{E}(F_0)^{1/2}\nonumber \\
		&\quad + C_{p,\gamma, \beta, N,\delta}\sup_{0\le s \le t}\|h(s)\|_{L^p_vL^\infty_x}^{\frac{3p-2}{2p-2}}\mathcal{E}(F_0)^{\frac{p-2}{2p-2}},\nonumber
	\end{align} 
	where we have used a similar argument to control \eqref{I4} in the last inequality.\\
	We next deal with the term $I_{5,2}$. Note that, on the region $\{|u| \leq N\}$, the following holds 
	\begin{align*}
		e^{-\int_s^t \tilde{\nu}(\tau,u) d\tau} \leq e^{-\int_s^t \nu(u)d\tau} \leq e^{-CN^{\gamma}(t-s)}.
	\end{align*}
	We fix $\ell = \frac{1}{2}\left(\max\left\{\frac{3}{3+\gamma},2\right\}+p \right)$. Using the fact above, one obtains that 
	\begin{align} \label{I52} 
		I_{5,2} &\leq C_\ell \int_0^{t-\delta} e^{-CN^{\gamma}(t-s)}\int_{|u|\leq N} e^{-\frac{|u|^2}{8}} |h(s,X(s),u)| \left(\int_{\R^3} \sqrt{\mu(\eta)} |h(s,X(s),\eta)|^\ell d\eta \right)^{1/\ell} duds \nonumber\\
		&\le  C_\ell \int_0^{t-\delta} e^{-CN^{\gamma}(t-s)}\int_{|u|\leq N} e^{-\frac{|u|^2}{8}} |h(s,X(s),u)| \nonumber \\
		&\qquad \times \left(\int_{|\eta| \geq 2N } \sqrt{\mu(\eta)} |h(s,X(s),\eta)|^\ell d\eta \right)^{1/\ell} duds\nonumber\\
		&\quad +C_\ell \int_0^{t-\delta} e^{-CN^{\gamma}(t-s)}\int_{|u|\leq N} e^{-\frac{|u|^2}{8}} |h(s,X(s),u)| \nonumber \\
		&\qquad \times \left(\int_{|\eta| \leq 2N } \sqrt{\mu(\eta)} |h(s,X(s),\eta)|^\ell d\eta \right)^{1/\ell} duds\nonumber\\
		&\leq \frac{C_\ell}{N^4} \sup_{0\leq s \leq t} \Vert h(s) \Vert_{L^p_v L^\infty_x}^2 \int_0^{t-\delta} e^{-CN^\gamma (t-s)} ds \nonumber \\
		&\quad + C_\ell  \int_0^{t-\delta} e^{-CN^{\gamma}(t-s)} \left(\int_{|u|\leq N} |e^{-\frac{|u|^2}{8}}h(s,X(s),u)|^{\frac{\ell}{\ell-1}}du \right)^{1-\frac{1}{\ell}}\nonumber\\
		& \qquad \times \left(\int_{|u|\leq N} \int_{|\eta|\leq 2N} |h(s,X(s),\eta)|^\ell dud\eta \right)^{1/\ell}ds\nonumber \\
		&\leq \frac{C_\ell}{N} \sup_{0\leq s \leq t} \Vert h(s) \Vert_{L^p_v L^\infty_x}^2  \nonumber \\
		&\quad + C_\ell\sup_{0\leq s \leq t} \Vert h(s) \Vert_{L^p_v L^\infty_x}  \int_0^{t-\delta} e^{-CN^{\gamma}(t-s)} \left(\int_{|u|\leq N} \int_{|\eta|\leq 2N} |h(s,X(s),\eta)|^\ell dud\eta \right)^{1/\ell}ds,
	\end{align}
	where $\ell>2$.
	Making a change of variables $u \mapsto y:=X(s) = x-u(t-s)$ with $\left|\frac{dy}{du} \right| =(t-s)^3$, the last term in \eqref{I52} becomes
	\begin{align*}
		 C_{\ell,N,\delta}\sup_{0\leq s \leq t} \Vert h(s) \Vert_{L^p_v L^\infty_x}  \int_0^{t-\delta} e^{-CN^{\gamma}(t-s)} \left(\int_{\T^3} \int_{|\eta|\leq 2N} |h(s,y,\eta)|^\ell d\eta dy \right)^{1/\ell}ds.
	\end{align*}
	Here, we use Lemma \ref{L1L2cont} and the interpolation inequality to obtain
\begin{align} \label{RTIRE1212}
	&\left(\int_{\T^3}\int_{|\eta|\le 2N}|h(s,y,\eta)|^\ell d\eta dy\right)^{1/\ell}\nonumber\\
	& \le C_N \left(\int_{\T^3}\int_{|\eta|\le 2N}|h(s,y,\eta)|^\ell \mathbf{1}_{\{|f|\le \sqrt{\mu}\}} d\eta dy\right)^{1/\ell}\nonumber\\
	&\quad +C_N\left(\int_{\T^3}\int_{|\eta|\le 2N}|h(s,y,\eta)|^\ell \mathbf{1}_{\{|f|> \sqrt{\mu}\}}d\eta dy\right)^{1/\ell}\nonumber\\
	& \le C_{\ell,N}\left(\int_{\T^3}\int_{|\eta|\le 2N}|f(s,y,\eta)|^2 \mathbf{1}_{\{|f|\le \sqrt{\mu}\}} d\eta dy\right)^{\frac{1}{2} \cdot \frac{1/\ell-1/p}{1/2-1/p}}\nonumber\\
	&\qquad \times \left(\int_{\T^3}\int_{|\eta|\le 2N}|h(s,y,\eta)|^p \mathbf{1}_{\{|f|\le \sqrt{\mu}\}} d\eta dy\right)^{\frac{1}{p}\cdot \left(1-\frac{1/\ell-1/p}{1/2-1/p}\right)}\nonumber\\
	& \quad + C_{\ell,N}\left(\int_{\T^3}\int_{|\eta|\le 2N}|f(s,y,\eta)| \mathbf{1}_{\{|f|> \sqrt{\mu}\}} d\eta dy\right)^{\frac{1/\ell-1/p}{1-1/p}}\nonumber\\
	&\qquad \times \left(\int_{\T^3}\int_{|\eta|\le 2N}|h(s,y,\eta)|^p \mathbf{1}_{\{|f|>\sqrt{\mu}\}} d\eta dy\right)^{\frac{1}{p}\cdot \left(1-\frac{1/\ell-1/p}{1-1/p}\right)}\nonumber\\
	&\le C_{\ell,N} \|h(s)\|_{L^p_vL^\infty_x}^{1-\frac{1/\ell-1/p}{1/2-1/p}}\mathcal{E}(F_0)^{\frac{1}{2}\cdot\frac{1/\ell-1/p}{1/2-1/p}} + C_{\ell,N}\|h(s)\|_{L^p_vL^\infty_x}^{1-\frac{1/\ell-1/p}{1-1/p}}\mathcal{E}(F_0)^{\frac{1/\ell-1/p}{1-1/p}}.
\end{align}
Thus, we derive
\begin{align*}
	I_{5,2} &\le \frac{C_{\ell,\delta}}{N} \sup_{0\leq s \leq t} \Vert h(s) \Vert_{L^p_v L^\infty_x}^2 +C_{\ell,N,\delta} \sup_{0\leq s \leq t}\|h(s)\|_{L^p_vL^\infty_x}^{2-\frac{1/\ell-1/p}{1/2-1/p}}\mathcal{E}(F_0)^{\frac{1}{2}\cdot\frac{1/\ell-1/p}{1/2-1/p}}\\
	& \quad  + C_{\ell,N,\delta}\sup_{0\leq s \leq t}\|h(s)\|_{L^p_vL^\infty_x}^{2-\frac{1/\ell-1/p}{1-1/p}}\mathcal{E}(F_0)^{\frac{1/\ell-1/p}{1-1/p}}.
\end{align*}
	Hence, 
	\begin{align} \label{case1.I5}
		I_5 &\leq \frac{C_{p,\gamma,\beta,\ell,\delta}}{N} \sup_{0\leq s \leq t} \Vert h(s) \Vert_{L^p_v L^\infty_x}^2 + C_{p,\gamma,\beta,N,\delta}\sup_{0\leq s \leq t} \Vert h(s)\Vert_{L^p_v L^\infty_x}\mathcal{E}(F_0)^{1/2}\nonumber \\
		& \quad + C_{p,\gamma, \beta, N,\delta}\sup_{0\le s \le t}\|h(s)\|_{L^p_vL^\infty_x}^{\frac{3p-2}{2p-2}}\mathcal{E}(F_0)^{\frac{p-2}{2p-2}}+C_{\ell,N,\delta} \sup_{0\leq s \leq t}\|h(s)\|_{L^p_vL^\infty_x}^{2-\frac{1/\ell-1/p}{1/2-1/p}}\mathcal{E}(F_0)^{\frac{1}{2}\cdot\frac{1/\ell-1/p}{1/2-1/p}}\nonumber\\
		& \quad + C_{\ell,N,\delta}\sup_{0\leq s \leq t}\|h(s)\|_{L^p_vL^\infty_x}^{2-\frac{1/\ell-1/p}{1-1/p}}\mathcal{E}(F_0)^{\frac{1/\ell-1/p}{1-1/p}}.
	\end{align}
	Next, let us consider the case $-3< \gamma <-1$. \\
	\newline 
	\indent \textbf{(Case 2. $-3<\gamma <-1$)} Similar to Case 1, it follows from Lemma \ref{pointwiseGamma-estimate} and Lemma \ref{pointwiseGamma+estimate} that 
	\begin{align*}
		I_5 &\leq C_{p,\gamma,\beta} \sup_{0\leq s \leq t} \Vert h(s) \Vert_{L^p_v L^\infty_x} \int_0^{t-\delta} \int_{|u| \leq N} e^{-\frac{|u|^2}{8}}  e^{-\int_s^t 	\tilde{\nu}(\tau,u)d\tau} \\
		&\qquad \times \left(\int_{\R^3} (1+|\eta|)^{\frac{4}{m-1}-p'm'\beta}|h(s,X(s),\eta)|^{p'm'} d\eta \right)^{\frac{1}{p'm'}}duds \\
		&\quad + C_\ell \int_0^{t-\delta} \int_{|u|\leq N} e^{-\frac{|u|^2}{8}}  e^{-\int_s^t 	\tilde{\nu}(\tau,u)d\tau}  \nu(u)|h(s,X(s),u)| \\
		&\qquad \times \left(\int_{\R^3} \sqrt{\mu(\eta)} |h(s,X(s),\eta)|^\ell d\eta \right)^{1/\ell} duds \\
		&:= I_{5,1} + I_{5,2}.  
	\end{align*}
	Since the $I_{5,2}$ term is identical to Case 1, we only provide the estimate for $I_{5,1}$. Similar to Case 1, we first divide the integration domain with respect to $\eta$, and then use the interpolation inequality to obtain an estimate
	\begin{align} \label{betacondition2}
		I_{5,1} &\le C_{p,\gamma,\beta} \sup_{0\leq s \leq t} \Vert h(s) \Vert_{L^p_v L^\infty_x} \int_0^{t-\delta} \int_{|u| \leq N} e^{-\frac{|u|^2}{8}}  e^{-\int_s^t 	\tilde{\nu}(\tau,u)d\tau}\nonumber \\
		&\qquad \times  \left(\int_{|\eta| \geq 2N} (1+|\eta|)^{\frac{4}{m-1}-p'm'\beta}|h(s,X(s),\eta)|^{p'm'} d\eta \right)^{\frac{1}{p'm'}}duds\nonumber\\
		&\quad + C_{p,\gamma,\beta} \sup_{0\leq s \leq t} \Vert h(s) \Vert_{L^p_v L^\infty_x} \int_0^{t-\delta} \int_{|u| \leq N} e^{-\frac{|u|^2}{8}}  e^{-\int_s^t 	\tilde{\nu}(\tau,u)d\tau}\nonumber  \\
		&\qquad \times \left(\int_{|\eta|\leq 2N} (1+|\eta|)^{\frac{4}{m-1}-p'm'\beta}|h(s,X(s),\eta)|^{p'm'} d\eta \right)^{\frac{1}{p'm'}}duds\nonumber\\
		&\leq C_{p,\gamma,\beta} \sup_{0 \leq s \leq t} \Vert h(s) \Vert_{L^p_v L^\infty_x}^2 \int_0^{t-\delta} \int_{|u| \leq N} e^{-\frac{|u|^2}{8}}  e^{-\int_s^t 	\tilde{\nu}(\tau,u)d\tau} \nonumber\\
		&\qquad \times \left(\int_{|\eta|\geq 2N} (1+|\eta|)^{\left(\frac{4}{m-1}-p'm\beta\right)\cdot \frac{p}{p-p'm'}}d\eta \right)^{\frac{p-p'm'}{pp'm'}}duds \\
		&\quad + C_{p,\gamma,\beta} \sup_{0\leq s \leq t} \Vert h(s) \Vert_{L^p_v L^\infty_x} \int_0^{t-\delta} \int_{|u|\leq N} e^{-\frac{|u|^2}{8}}  e^{-\int_s^t 	\tilde{\nu}(\tau,u)d\tau}\nonumber \nonumber\\
		&\qquad \times  \left(\int_{|\eta| \leq 2N} |h(s,X(s),\eta)|^{p'm'} d\eta \right)^{1/p'm'}duds\nonumber \\
		&\leq \frac{C_{p,\gamma,\beta}}{N}  \sup_{0 \leq s \leq t} \Vert h(s) \Vert_{L^p_v L^\infty_x}^2\nonumber\\
		&\quad + C_{p,\gamma,\beta}\sup_{0\leq s \leq t} \Vert h(s) \Vert_{L^p_v L^\infty_x } \sup_{0\leq s \leq t-\delta} \left(\int_{|u|\leq N} \int_{|\eta| \leq 2N} |h(s,X(s),\eta)|^{p'm'} dud\eta \right)^{1/p'm'}\nonumber\\
		&\le \frac{C_{p,\gamma,\beta}}{N}  \sup_{0 \leq s \leq t} \Vert h(s) \Vert_{L^p_v L^\infty_x}^2+C_{p,\gamma,N,\delta} \sup_{0 \leq s \leq t}\|h(s)\|_{L^p_vL^\infty_x}^{2-\frac{1/p'm'-1/p}{1/2-1/p}}\mathcal{E}(F_0)^{\frac{1}{2}\cdot\frac{1/p'm'-1/p}{1/2-1/p}}\nonumber\\
		& \quad  + C_{p,\gamma,N,\delta}\sup_{0 \leq s \leq t}\|h(s)\|_{L^p_vL^\infty_x}^{2-\frac{1/p'm'-1/p}{1-1/p}}\mathcal{E}(F_0)^{\frac{1/p'm'-1/p}{1-1/p}}\nonumber
	\end{align}
	where we have used a similar argument to \eqref{RTIRE1212} in the last inequality. Thus, in Case 2, we obtain the estimate for $I_5$ as
	\begin{align} \label{case2.I5}
		I_5 &\leq \frac{C_{p,\gamma,\beta,\ell,\delta}}{N}  \sup_{0 \leq s \leq t} \Vert h(s) \Vert_{L^p_v L^\infty_x}^2 +C_{\ell,N,\delta} \sup_{0\leq s \leq t}\|h(s)\|_{L^p_vL^\infty_x}^{2-\frac{1/\ell-1/p}{1/2-1/p}}\mathcal{E}(F_0)^{\frac{1}{2}\cdot\frac{1/\ell-1/p}{1/2-1/p}}\nonumber\\
	& \quad  + C_{\ell,N,\delta}\sup_{0\leq s \leq t}\|h(s)\|_{L^p_vL^\infty_x}^{2-\frac{1/\ell-1/p}{1-1/p}}\mathcal{E}(F_0)^{\frac{1/\ell-1/p}{1-1/p}}\nonumber \\
	&\quad +C_{p,\gamma,N,\delta} \sup_{0 \leq s \leq t}\|h(s)\|_{L^p_vL^\infty_x}^{1-\frac{1/p'm'-1/p}{1/2-1/p}}\mathcal{E}(F_0)^{\frac{1}{2}\cdot\frac{1/p'm'-1/p}{1/2-1/p}}\nonumber\\
		& \quad  + C_{p,\gamma,N,\delta}\sup_{0 \leq s \leq t}\|h(s)\|_{L^p_vL^\infty_x}^{1-\frac{1/p'm'-1/p}{1-1/p}}\mathcal{E}(F_0)^{\frac{1/p'm'-1/p}{1-1/p}}.
	\end{align}
	Putting together \eqref{Rf.duhamel}, \eqref{I1}, \eqref{I2}, \eqref{I3}, \eqref{Rf17}, and \eqref{case1.I5} for Case 1, we deduce that 
	\begin{align*}
		&\int_{\R^3} e^{-\frac{|u|^2}{8}} |h(t,x,u)| du\\ &\leq e^{-\lambda (1+t)^\rho} \Vert h_0 \Vert_{L^p_v L^\infty _x} + C_{p,q,\gamma,\beta,\varepsilon} \delta \left[\sup_{0\leq s \leq t}\Vert h(s) \Vert_{L^p_vL^\infty_x} + \sup_{0 \leq s \leq t} \Vert h(s) \Vert_{L^p_vL^\infty_x}^2\right] \\
		&\quad + C_{p,\gamma} \varepsilon^{\gamma+\frac{3}{p'}}\sup_{0\leq s \leq t } \Vert h(s) \Vert_{L^p_vL^\infty_x} +  C_{\gamma}\left(\frac{1}{N^{-\gamma}}+\frac{1}{N^{\frac{3-\gamma}{2}}}\right) \sup_{0\leq s \leq t} \Vert h(s)\Vert_{L^p_v L^\infty_x} \\
		& \quad +C_{N,\delta} \mathcal{E}(F_0)^{1/2} + C_{p,N,\delta}\sup_{0\le s \le t}\|h(s)\|_{L^p_vL^\infty_x}^{\frac{p}{2p-2}}\mathcal{E}(F_0)^{\frac{p-2}{2p-2}}\\
		&\quad + \frac{C_{p,\gamma,\beta,\ell,\delta}}{N} \sup_{0\leq s \leq t} \Vert h(s) \Vert_{L^p_v L^\infty_x}^2 + C_{p,\gamma,\beta,N,\delta}\sup_{0\leq s \leq t} \Vert h(s)\Vert_{L^p_v L^\infty_x}\mathcal{E}(F_0)^{1/2}\nonumber \\
		& \quad + C_{p,\gamma, \beta, N,\delta}\sup_{0\le s \le t}\|h(s)\|_{L^p_vL^\infty_x}^{\frac{3p-2}{2p-2}}\mathcal{E}(F_0)^{\frac{p-2}{2p-2}}+C_{\ell,N,\delta} \sup_{0\leq s \leq t}\|h(s)\|_{L^p_vL^\infty_x}^{2-\frac{1/\ell-1/p}{1/2-1/p}}\mathcal{E}(F_0)^{\frac{1}{2}\cdot\frac{1/\ell-1/p}{1/2-1/p}}\nonumber\\
		& \quad + C_{\ell,N,\delta}\sup_{0\leq s \leq t}\|h(s)\|_{L^p_vL^\infty_x}^{2-\frac{1/\ell-1/p}{1-1/p}}\mathcal{E}(F_0)^{\frac{1/\ell-1/p}{1-1/p}}.
	\end{align*}
	For Case 2, collecting \eqref{Rf.duhamel}, \eqref{I1}, \eqref{I2}, \eqref{I3}, \eqref{Rf17}, and \eqref{case2.I5}, we obtain the following estimate 
	\begin{align*}
		&\int_{\R^3} e^{-\frac{|u|^2}{8}} |h(t,x,u)| du\\ &\leq e^{-\lambda (1+t)^\rho} \Vert h_0 \Vert_{L^p_v L^\infty _x} + C_{p,q,\gamma,\beta,b,\varepsilon} \delta \left[\sup_{0\leq s \leq t}\Vert h(s) \Vert_{L^p_vL^\infty_x} + \sup_{0 \leq s \leq t} \Vert h(s) \Vert_{L^p_vL^\infty_x}^2\right] \\
		&\quad +  C_{p,\gamma} \varepsilon^{\gamma+\frac{3}{p'}}\sup_{0\leq s \leq t } \Vert h(s) \Vert_{L^p_vL^\infty_x} +  C_{\gamma}\left(\frac{1}{N^{-\gamma}}+\frac{1}{N^{\frac{3-\gamma}{2}}}\right) \sup_{0\leq s \leq t} \Vert h(s)\Vert_{L^p_v L^\infty_x} \\
		& \quad +C_{N,\delta} \mathcal{E}(F_0)^{1/2} + C_{p,N,\delta}\sup_{0\le s \le t}\|h(s)\|_{L^p_vL^\infty_x}^{\frac{p}{2p-2}}\mathcal{E}(F_0)^{\frac{p-2}{2p-2}}\\
		&\quad + \frac{C_{p,\gamma,\beta,\ell,\delta}}{N}  \sup_{0 \leq s \leq t} \Vert h(s) \Vert_{L^p_v L^\infty_x}^2 +C_{\ell,N,\delta} \sup_{0\leq s \leq t}\|h(s)\|_{L^p_vL^\infty_x}^{2-\frac{1/\ell-1/p}{1/2-1/p}}\mathcal{E}(F_0)^{\frac{1}{2}\cdot\frac{1/\ell-1/p}{1/2-1/p}}\nonumber\\
	& \quad  + C_{\ell,N,\delta}\sup_{0\leq s \leq t}\|h(s)\|_{L^p_vL^\infty_x}^{2-\frac{1/\ell-1/p}{1-1/p}}\mathcal{E}(F_0)^{\frac{1/\ell-1/p}{1-1/p}}+C_{p,\gamma,N,\delta} \sup_{0 \leq s \leq t}\|h(s)\|_{L^p_vL^\infty_x}^{1-\frac{1/p'm'-1/p}{1/2-1/p}}\mathcal{E}(F_0)^{\frac{1}{2}\cdot\frac{1/p'm'-1/p}{1/2-1/p}}\nonumber\\
		& \quad  + C_{p,\gamma,N,\delta}\sup_{0 \leq s \leq t}\|h(s)\|_{L^p_vL^\infty_x}^{1-\frac{1/p'm'-1/p}{1-1/p}}\mathcal{E}(F_0)^{\frac{1/p'm'-1/p}{1-1/p}}.
	\end{align*}
	
\end{proof}
\begin{Lem} \label{Rfestimatesmall}
	Under the a priori assumption \eqref{apriorismall}, there exists a sufficiently small positive constant $\eta_0>0$ satisfying $\Vert h_0 \Vert_{L^p_vL^\infty_x} \leq \eta_0$, such that for any $T>0$, there exists a small positive constant $\varepsilon_2 = \varepsilon_2(\eta_0,T)>0$, whenever if $\mathcal{E}(F_0)  \leq \varepsilon_2$, then 
	\begin{align*}
		R(f) (t,x,v) \geq \frac{1}{2} \nu(v) + \frac{\vartheta q|v|^2}{8(1+t)^{\vartheta+1}}, \qquad \forall (t,x,v) \in [0,T) \times \T^3 \times \R^3.
	\end{align*} 
\end{Lem}
\begin{proof}
	For both cases in Lemma \ref{exp.h.esti}, we first choose $\delta>0$ and $\varepsilon>0$ sufficiently small, and then $N>0$ large enough, and finally the constant $\varepsilon_2$ satisfying $\mathcal{E}(F_0) \leq \varepsilon_2$ is small enough, so that 
	\begin{align*}
		\int_{\R^3} e^{-\frac{|u|^2}{8}} |h(t,x,u)| du \leq e^{-\lambda (1+t)^\rho} \Vert h_0 \Vert_{L^p_vL^\infty_x} + \frac{1}{4C_2}.
	\end{align*}
	Thus, if we choose the constant $\eta_0$ satisfying that 
	\begin{align} \label{cond3apriori}
		\eta_0 < \frac{1}{4C_2},
	\end{align} 
	we obtain the following result: 
	\begin{align*}
		\int_{\R^3} e^{-\frac{|u|^2}{8}} |h(t,x,u)| du \leq \frac{1}{2C_2}.
	\end{align*}
	From \eqref{Rf.esti1}, we get 
	\begin{align*}
		R(f) (t,x,v) \geq \frac{1}{2} \nu(v) + \frac{\vartheta q|v|^2}{8(1+t)^{\vartheta+1}}, \qquad \forall (t,x,v) \in [0,T) \times \T^3 \times \R^3.
	\end{align*}
\end{proof}	

\bigskip

\subsection{$L^p_vL^\infty_x$ estimate} To get the sub-exponential time-decay property of the solution operator $G_v(t,s)$ using the positive lower bound of $R(f)$, we consider the following equation as follow : 
\begin{align}\label{Rf nonlinear equation}
	\p_th+v\cdot \nabla_xh+ R(f)h = K_wh + w\Gamma^+(f,f),
\end{align}
where $R(f)$ is defined in \eqref{Rf}. Then the solution of \eqref{Rf nonlinear equation} can be written by 
\begin{align} \label{Rf Duhamel}
	h(t,x,v)
	&= G_v(t,0)h_0(x-tv,v) + \int_0^t G_v(t,s) [K_wh(s) + w\Gamma^+(f,f)(s)]ds.
\end{align}
We can define the solution operator $G_v(t,s)$ as follow : 
\begin{align*}
	G_v(t,s) : = e^{-\int_s^t R(f)(\tau,X(\tau),v) d\tau},
\end{align*}
where $X(\tau) = x-v(t-\tau)$. By Lemma \ref{Rfestimatesmall}, we derive the sub-exponential decay property as
\begin{align} \label{small,G,v} 
	G_v (t,s)  \leq e^{-\lambda(1+t)^\rho} e^{\lambda (1+s)^\rho}, 
\end{align}
where we have used $\frac{\nu(v)}{2} + \frac{\vartheta q |v|^2}{8(1+\tau)^{\vartheta+1}} \geq C(1+\tau)^{\frac{(1+\vartheta)\gamma}{2-\gamma}}$ for some generic constant $C>0$. Here, $\rho= 1+ \frac{(1+\vartheta)\gamma}{2-\gamma}>0$ and $\lambda =\frac{C}{\rho}>0$. 

\bigskip
 
\begin{Lem} \label{small,Lp}
	Let $h(t,x,v)$ satisfy the equation \eqref{Rf nonlinear equation} and  $\rho-1 = \frac{(1+\vartheta)\gamma}{2-\gamma}$. Let $0<t\le T\le \infty$. Under the assumption in Lemma \ref{Rfestimatesmall} and $\mathcal{E}(F_0) \leq \varepsilon_2$, the following estimates holds for each case:
	\begin{enumerate}[label=(\arabic*)]
		\item If $-1 \leq \gamma<0$, then 
		\begin{align*}
			&|h(t,x,v)|\\ &\leq C_{\lambda}e^{-\lambda(1+t)^{\rho}} h_0(x-tv,v)  + \frac{C}{(1+|v|)^{-\frac{\gamma}{p}+\frac{5p-5}{4p}}}e^{-\lambda (1+t)^\rho} \Vert h_0 \Vert_{L^p_v L^\infty_x} \int_0^t \Vert h(s) \Vert_{L^p_vL^\infty_x}ds \\
		&\quad +C_{p,\rho,\lambda,\gamma}e^{-\frac{\lambda}{2}(1+t)^{\rho}} \langle v \rangle^{\gamma-1-\frac{1}{p'}} \Vert h_0 \Vert_{L^p_v L^\infty_x}\\
		&\quad +\frac{C_{p,q,\gamma}}{(1+|v|)^{\frac{p-1}{4p}}} \varepsilon^{\gamma+ \frac{3}{p'}}   \left[\sup_{0 \leq s \leq t} \Vert h(s) \Vert_{L^p_v L^\infty_x}+\sup_{0 \leq s \leq t} \Vert h(s) \Vert_{L^p_v L^\infty_x}^2\right]\\
		&\quad + \frac{C_{p,q,\varepsilon,\gamma,\beta}}{(1+|v|)^{\frac{p-1}{4p}}} \left(\frac{1}{N}+\frac{1}{N^{-\gamma}}+\frac{1}{N^{\frac{\gamma+3}{2}}}+\delta  \right) \\
		& \qquad \times \left[\sup_{0\leq s \leq t} \Vert h(s) \Vert_{L^p_vL^\infty_x}+\sup_{0\leq s \leq t} \Vert h(s) \Vert_{L^p_vL^\infty_x}^2+\sup_{0\leq s \leq t} \Vert h(s) \Vert_{L^p_vL^\infty_x}^3\right]\\
		&\quad +\frac{C_{p,q,\varepsilon,\gamma,N,\delta}}{(1+|v|)^{\frac{p-1}{4p}}}\left(1+\sup_{0\leq s \leq t} \Vert h(s) \Vert_{L^p_vL^\infty_x}+\sup_{0\leq s \leq t} \Vert h(s) \Vert_{L^p_vL^\infty_x}^2\right)\\
		& \qquad \times  \left[ \mathcal{E}(F_0)^{1/2} + \sup_{0\le s \le t}\|h(s)\|_{L^p_vL^\infty_x}^{\frac{p}{2p-2}}\mathcal{E}(F_0)^{\frac{p-2}{2p-2}}\right],
		\end{align*}
		where $0<\delta\ll1, 0<\varepsilon\ll1$, and $N\gg1$ can be chosen arbitrarily small and large, respectively. 
		\item If $-3<\gamma<-1$, then
		\begin{align*}
			&|h(t,x,v)|\\ &\leq C_{\lambda}e^{-\lambda(1+t)^{\rho}} h_0(x-tv,v)  +\frac{C}{(1+|v|)^{-\gamma+ \frac{p-1}{p}\varpi}}  e^{-\lambda(1+t)^\rho} \Vert h_0 \Vert_{L^p_v L^\infty_x} \int_0^t \Vert h(s) \Vert_{L^p_v L^\infty_x} ds \nonumber \\
		&\quad +C_{p,\rho,\lambda,\gamma}e^{-\frac{\lambda}{2}(1+t)^{\rho}} \langle v \rangle^{\gamma-1-\frac{1}{p'}} \Vert h_0 \Vert_{L^p_v L^\infty_x}\\
		&\quad +\frac{C_{p,q,\gamma}}{(1+|v|)^{\frac{p-1}{p}\varpi}} \varepsilon^{\gamma+ \frac{3}{p'}}   \left[\sup_{0 \leq s \leq t} \Vert h(s) \Vert_{L^p_v L^\infty_x}+\sup_{0 \leq s \leq t} \Vert h(s) \Vert_{L^p_v L^\infty_x}^2\right]\\
		&\quad + \frac{C_{p,q,\varepsilon,\gamma,\beta}}{(1+|v|)^{\frac{p-1}{p}\varpi}} \left(\frac{1}{N}+\frac{1}{N^{-\gamma}}+\frac{1}{N^{\frac{\gamma+3}{2}}}+\delta  \right) \\
		& \qquad \times \left[\sup_{0\leq s \leq t} \Vert h(s) \Vert_{L^p_vL^\infty_x}+\sup_{0\leq s \leq t} \Vert h(s) \Vert_{L^p_vL^\infty_x}^2+\sup_{0\leq s \leq t} \Vert h(s) \Vert_{L^p_vL^\infty_x}^3\right]\\
		&\quad +\frac{C_{p,q,\varepsilon,\gamma,N,\delta}}{(1+|v|)^{\frac{p-1}{p}\varpi}}\left(1+\sup_{0\leq s \leq t} \Vert h(s) \Vert_{L^p_vL^\infty_x}\right)\left[ \mathcal{E}(F_0)^{1/2} + \sup_{0\le s \le t}\|h(s)\|_{L^p_vL^\infty_x}^{\frac{p}{2p-2}}\mathcal{E}(F_0)^{\frac{p-2}{2p-2}}\right]\\
		& \quad +\frac{C_{p,q,\varepsilon,\gamma,N,\delta}}{(1+|v|)^{\frac{p-1}{p}\varpi}}\left(\sup_{0\leq s \leq t} \Vert h(s) \Vert_{L^p_vL^\infty_x}+\sup_{0\leq s \leq t} \Vert h(s) \Vert_{L^p_vL^\infty_x}^2\right)\\
		& \qquad \times  \Bigg[\sup_{0 \leq s \leq t}\|h(s)\|_{L^p_vL^\infty_x}^{1-\frac{1/p'm'-1/p}{1/2-1/p}}\mathcal{E}(F_0)^{\frac{1}{2}\cdot\frac{1/p'm'-1/p}{1/2-1/p}}+ \sup_{0 \leq s \leq t}\|h(s)\|_{L^p_vL^\infty_x}^{1-\frac{1/p'm'-1/p}{1-1/p}}\mathcal{E}(F_0)^{\frac{1/p'm'-1/p}{1-1/p}}\Bigg],
		\end{align*}
		where $0<\delta\ll1, 0<\varepsilon\ll1$, and $N\gg1$ can be chosen arbitrarily small and large, respectively. 
	\end{enumerate}	
\end{Lem}
\begin{proof}
	For fixed $(t,x,v) \in \R^+ \times \T^3 \times \R^3$, using \eqref{Rf Duhamel}, one obtains that 
	\begin{align*}
		|h(t,x,v)| &\leq G_v(t,0) |h_0 (x-vt,v)| \\
		&\quad +  \int_0^t G_v(t,s)[|K_w h| + |w\Gamma^+(f,f)|](s,x-v(t-s),v) ds\\
		&:= I + II + III. 
	\end{align*}
	For $I$, it follows from \eqref{G,v} that 
	\begin{align} \label{I}
		I \leq C_{\lambda}e^{-\lambda (1+t)^\rho} h_0(x-vt, v). 
	\end{align}
	To handle $II$, we split the term $II$ as follows:
	\begin{align*}
		II = \int_0^t G_v(t,s) [|K_w^{ns} h| + |K_w^{s}h |](s,x-v(t-s),v)ds. 
	\end{align*}
	To estimate the second term above, we apply the Lemma \ref{Ksingular} and obtain the following
	\begin{align} \label{K,est1}
			&\int_0^t G_v(t,s) |K^{s}_w h|(s,x-v(t-s),v) ds \nonumber \\
			&\leq C_{p,\gamma}\int_0^t G_v(t,s) \nu(v) \nu^{-1}(v) \mu(v) ^{\frac{1-q}{8}} \varepsilon^{\gamma+ \frac{3}{p'}} ds \sup_{0\leq s \leq t} \Vert h(s) \Vert_{L^p_v L^\infty_x} \nonumber \\ 
			&\leq C_{p,\gamma} \varepsilon^{\gamma+ \frac{3}{p'}} \mu(v)^{\frac{1-q}{16}}  \sup_{0 \leq s \leq t} \Vert h(s) \Vert_{L^p_v L^\infty_x},
	\end{align}	
	where the last inequality comes from $G_v(t,s)\leq e^{-\frac{\nu(v)}{2}(t-s)}$  and 
	\begin{align*}
		\int_0^t G_v(t,s) \nu(v) ds \leq  \int_0^t e^{-\frac{\nu(v)}{2}(t-s)} \nu(v) ds \leq 2. 
	\end{align*}
	For the remaining part of $II$, applying Duhamel’s principle once more yields the following estimate
	\begin{align} \label{II,esti1}
			&\int_0^t G_v(t,s) \int_{\R^3} k_w^{ns}(v,u) |h(s,x-v(t-s),u)|du ds \nonumber \\
			&\leq \int_0^t G_v(t,s) \int_{\R^3} k_w^{ns}(v,u) G_u(s,0) |h_0(x-v(t-s)-us,u)| duds \nonumber \\
			&\quad + \int_0^t G_v(t,s) \int_{\R^3} k_w^{ns}(v,u) \int_0^sG_u(s,s')[|K_wh| + |w\Gamma^+(f,f)|](s') ds' duds \nonumber \\
			&:= II_1+II_2+II_3.
	\end{align}
	For $II_1$, using \eqref{G,v}, H\"{o}lder's inequality, and Corollary \ref{pKestimate}, we have 
	\begin{align} \label{II1,final}
			II_1 &\leq C\int_0^t e^{-\lambda (1+t)^{\rho}} \int_{\R^3} k_w(v,u)  |h_0(x-v(t-s)-us,u)| du ds \nonumber \\
			&\leq C\int_0^t e^{-\lambda (1+t)^{\rho}} \left(\int_{\R^3} |k_w(v,u)|^{p'} du \right)^{1/p'} \left(\int_{\R^3} |h_0(x-v(t-s)-us,u)|^p du\right)^{1/p} ds \nonumber\\
			&\leq C_{p,\gamma}\int_0^t e^{-\lambda (1+t)^{\rho}}  \langle v \rangle^{\gamma-1-\frac{1}{p'}} ds  \Vert h_0 \Vert_{L^p_v L^\infty_x} \nonumber \\
			&\leq C_{p,\rho,\lambda,\gamma}e^{-\frac{\lambda}{2}(1+t)^{\rho}} \langle v \rangle^{\gamma-1-\frac{1}{p'}} \Vert h_0 \Vert_{L^p_v L^\infty_x},
	\end{align}
	because $t e^{-\frac{\lambda}{2}(1+t)^\rho} \leq C_{\rho,\lambda}$ for all $t \geq 0$.  \\
	
	\noindent  To deal with $II_2$, we decompose the integration domain into $\{|u| \leq 2N\}$ and $\{|u|\geq 2N\}$: 
	\begin{align} \label{II2,1}
		II_2 &\leq \int_0^t \int_0^s  \int_{|u| \geq 2N} G_v(t,s)G_u(s,s')k_w^{ns}(v,u)\nonumber \\
		&\qquad \times \int_{\R^3} k_w^{ns} (u,u') |h(s',X(s)-u(s-s'), u')|du'duds'ds \nonumber \\
		&\quad +\int_0^t \int_0^s  \int_{|u| \leq 2N} G_v(t,s)G_u(s,s')k_w^{ns}(v,u)\nonumber \\
		&\qquad \times  \int_{\R^3} k_w^{ns} (u,u') |h(s',X(s)-u(s-s'), u')|du'duds'ds\nonumber\\
		&\quad + \int_0^t \int_0^s \int_{\R^3} G_v(t,s)G_u(s,s')k_w^{ns}(v,u) |K_w^{s}(s',X(s)-u(s-s'),u)|du'duds'ds,
	\end{align}
	where we have denoted $X(s)= x-v(t-s)$. 
	For the last term in \eqref{II2,1}, we use Lemma \ref{Ksingular} and Lemma \ref{Knonsing} to obtain
	\begin{align} \label{II2,ss}
		&\int_0^t \int_0^s \int_{\R^3} G_v(t,s)G_u(s,s')k_w^{ns}(v,u) |K_w^{s}(s',X(s)-u(s-s'),u)|du'duds'ds\nonumber\\
		& \le C_{p,\gamma}\ve^{\gamma+\frac{3}{p'}}\sup_{0\leq s \leq t } \Vert h(s)\Vert_{L^p_vL^\infty_x}\int_0^t G_v(t,s) \int_{\R^3} k_w^{ns}(v,u) \nu(u)^{-1}\mu(u)^{\frac{1-q}{8}}\int_0^sG_u(s,s')\nu(u) ds'duds\nonumber \\
		& \le C_{p,\gamma}\ve^{\gamma+\frac{3}{p'}}\sup_{0\leq s \leq t } \Vert h(s)\Vert_{L^p_vL^\infty_x}\int_0^t G_v(t,s) \langle v \rangle^{\gamma-2} ds \nonumber\\
		& \le C_{p,\gamma}\ve^{\gamma+\frac{3}{p'}} \langle v\rangle^{-2}\sup_{0\leq s \leq t } \Vert h(s)\Vert_{L^p_vL^\infty_x}.
	\end{align}
	For the first term in \eqref{II2,1}, it holds from Lemma \ref{Knonsing} that 
	\begin{align} \label{II2,2}
		&\int_0^t \int_0^s  \int_{|u| \geq 2N} G_v(t,s)G_u(s,s')k_w^{ns}(v,u) \int_{\R^3} k_w^{ns} (u,u') |h(s',X(s)-u(s-s'), u')|du'duds'ds \nonumber \\
		&\leq \int_0^t \int_0^s \int_{|u| \geq 2N} e^{-\frac{\nu(v)}{2}(t-s)}e^{-\frac{\nu(u)}{2}(s-s')} k_w^{ns}(v,u) \nonumber \\
		&\qquad \times \int_{\R^3} k_w^{ns} (u,u') |h(s',X(s)-u(s-s'), u')|du'duds'ds \nonumber \\
		&\leq C_{p,\varepsilon}\sup_{0\leq s \leq t } \Vert h(s)\Vert_{L^p_vL^\infty_x} \int_0^t \int_0^s \int_{|u| \geq 2N} e^{-\frac{\nu(v)}{2}(t-s)}e^{-\frac{\nu(u)}{2}(s-s')} k_w^{ns}(v,u) \langle u \rangle^{\gamma-1-\frac{1}{p'}}duds'ds \nonumber \\
		&\leq \frac{C_{p,\varepsilon}}{N}\sup_{0\leq s \leq t } \Vert h(s)\Vert_{L^p_vL^\infty_x} \int_0^t e^{-\frac{\nu(v)}{2}(t-s)}\int_{|u|\geq 2N} \int_0^s  e^{-\frac{\nu(u)}{2}(s-s')} \nu(u) ds' k_w^{ns}(v,u) du ds  \nonumber \\
		&\leq \frac{C_{p,\varepsilon}}{N}\sup_{0\leq s \leq t } \Vert h(s)\Vert_{L^p_vL^\infty_x} \int_0^t e^{-\frac{\nu(v)}{2}(t-s)} \int_{|u|\geq 2N} k_w^{ns}(v,u)duds \nonumber \\
		&\leq \frac{C_{p,\varepsilon}}{N}\sup_{0\leq s \leq t } \Vert h(s)\Vert_{L^p_vL^\infty_x} \int_0^t e^{-\frac{\nu(v)}{2}(t-s)} \langle v \rangle^{\gamma-2}ds \nonumber \\
		&\leq \frac{C_{p,\varepsilon}}{N}\frac{1}{(1+|v|)^2}\sup_{0\leq s \leq t } \Vert h(s)\Vert_{L^p_vL^\infty_x}.
	\end{align}
	Then, let us consider the remaining region $\{|u| \leq 2N\}$ in \eqref{II2,1}: 
	\begin{align*}
		&\int_0^t \int_0^s  \int_{|u| \leq 2N} G_v(t,s)G_u(s,s')k_w^{ns}(v,u) \int_{\R^3} k_w^{ns} (u,u') |h(s',X(s)-u(s-s'), u')|du'duds'ds\\
		&=\int_0^t \int_0^s  \int_{|u| \leq 2N} G_v(t,s)G_u(s,s')k_w^{ns}(v,u) \\
		&\qquad \times \int_{|u'| \geq 3N} k_w^{ns} (u,u') |h(s',X(s)-u(s-s'), u')|du'duds'ds\\
		&\quad +\int_0^t \int_0^s  \int_{|u| \leq 2N} G_v(t,s)G_u(s,s')k_w^{ns}(v,u) \\
		&\qquad \times \int_{|u'| \leq 3N} k_w^{ns} (u,u') |h(s',X(s)-u(s-s'), u')|du'duds'ds
	\end{align*}
	Note that $|u-u'| \geq N$ on the integration region $\{|u| \leq 2N\} $ and $\{ |u'| \geq 3N \}$. Using Lemma \ref{Kker} and Lemma \ref{Knonsing}, one obtains that 
	\begin{align} \label{II2,3}
		&\int_0^t \int_0^s  \int_{|u| \leq 2N} G_v(t,s)G_u(s,s')k_w^{ns}(v,u) \int_{|u'| \geq 3N} k_w^{ns} (u,u') |h(s',X(s)-u(s-s'), u')|du'duds'ds \nonumber \\
		&\leq \int_0^t \int_0^s  \int_{|u| \leq 2N} G_v(t,s)G_u(s,s')k_w^{ns}(v,u) \nonumber \\
		&\quad \times \int_{|u'| \geq 3N} [k_{w,1} (u,u') + k_{w,2}(u,u')] |h(s',X(s)-u(s-s'), u')|du'duds'ds \nonumber \\
		&\leq C_{\gamma} \int_0^t \int_0^s  \int_{|u| \leq 2N} G_v(t,s)G_u(s,s')k_w^{ns}(v,u) \nonumber \\
		&\quad \times \int_{|u'| \geq 3N} \left[|u-u'|^\gamma e^{-\frac{|u|^2}{4}}e^{-\frac{|u'|^2}{4}} + \frac{1}{|u-u'|^{\frac{3-\gamma}{2}}} e^{-\frac{|u-u'|^2}{8}}e^{-\frac{||u|^2-|u'|^2|^2}{8|u-u'|^2}} \right] \frac{w(u)}{w(u')} 
	\nonumber \\
		&\qquad \times |h(s',X(s)-u(s-s'), u')|du'duds'ds\nonumber \\
		&\leq C_{\gamma }\int_0^t \int_0^s  \int_{|u| \leq 2N} G_v(t,s)G_u(s,s')k_w^{ns}(v,u)\nonumber \\
		&\quad \times \int_{|u'|\geq 3N} \left[(1+|u|^2)^{\beta}|u-u'|^{\gamma} e^{-\frac{1-q}{4}|u|^2}e^{-\frac{|u'|^2}{4}} \right. \nonumber \\
		&\left. \qquad + \frac{1}{|u-u'|^{\frac{3-\gamma}{2}}} e^{-\frac{|u-u'|^2}{8}}e^{-\frac{||u|^2-|u'|^2|^2}{8|u-u'|^2}}\frac{w(u)}{w(u')}\right] |h(s',X(s)-u(s-s'),u')| du'duds'ds\nonumber\\
		&\leq \frac{C_{\gamma}}{N^{-\gamma}} \sup_{0\leq s \leq t} \Vert h(s)\Vert_{L^p_v L^\infty_x} \int_0^t \int_0^s \int_{|u| \leq 2N} G_v(t,s) G_u(s,s') k_w^{ns}(v,u) e^{-\frac{1-q}{8}|u|^2} du ds'ds \nonumber\\
		&\quad + \frac{C_{\gamma}}{N^{\frac{3-\gamma}{2}}}\sup_{0\leq s \leq t} \Vert h(s) \Vert_{L^p_vL^\infty_x} \int_0^t \int_0^s \int_{|u| \leq 2N} G_v(t,s) G_u(s,s') k_w^{ns}(v,u)\frac{1}{1+|u|}duds'ds\nonumber\\
		&\leq \frac{C_{\gamma}}{N^{-\gamma}} \sup_{0\leq s \leq t} \Vert h(s)\Vert_{L^p_v L^\infty_x} \nonumber \\
		&\qquad \times \int_0^t \int_0^s \int_{|u| \leq 2N} e^{-\frac{\nu(v)}{2}(t-s)} e^{-\frac{\nu(u)}{2}(s-s')} \nu(u) k_w^{ns}(v,u)\nu^{-1}(u) e^{-\frac{1-q}{8}|u|^2} du ds'ds \nonumber\\
		&\quad + \frac{C_{\gamma}}{N^{\frac{3-\gamma}{2}}}\sup_{0\leq s \leq t} \Vert h(s) \Vert_{L^p_vL^\infty_x} \nonumber \\
		&\qquad \times \int_0^t \int_0^s \int_{|u| \leq 2N} e^{-\frac{\nu(v)}{2}(t-s)}  e^{-\frac{\nu(u)}{2}(s-s')} \nu(u) k_w^{ns}(v,u)\nu^{-1}(u) \frac{1}{1+|u|}duds'ds\nonumber\\
		&\leq C_{p,q,\varepsilon,\gamma}\left(\frac{1}{N^{-\gamma}}+\frac{1}{N^{\frac{\gamma+3}{2}}} \right) \sup_{0\leq s \leq t} \Vert h(s) \Vert_{L^p_vL^\infty_x}\int_0^t e^{-\frac{\nu(v)}{2}(t-s)} \langle v \rangle^{\gamma-1-\frac{1}{p'}} dv \nonumber\\
		&\leq \frac{C_{p,q,\varepsilon,\gamma}}{(1+|v|)^{1+\frac{1}{p'}}} \left(\frac{1}{N^{-\gamma}}+\frac{1}{N^{\frac{\gamma+3}{2}}} \right) \sup_{0\leq s \leq t} \Vert h(s) \Vert_{L^p_vL^\infty_x}.
	\end{align}
	Note that from Lemma \ref{k1ker} and Lemma \ref{k2ker}, we have
	\begin{align*}
		\left(\int_{\R^3} |k^{ns}_w(v,u)|^2du\right)^{1/2} \le C_{\ve} \langle v\rangle^{\gamma-1}.
	\end{align*}
	On the region $\{|u| \leq 2N\}$ and $\{|u'| \leq 3N\}$, we derive the following estimate: 
	\begin{align} \label{II2,4}
		&\int_0^t \int_0^s  \int_{|u| \leq 2N} G_v(t,s)G_u(s,s')k_w^{ns}(v,u) \nonumber \\
		&\qquad \times \int_{|u'| \leq 3N} k_w^{ns} (u,u') |h(s',X(s)-u(s-s'), u')|du'duds'ds \nonumber \\
		&=\int_0^t \int_{s-\delta}^s  \int_{|u| \leq 2N} G_v(t,s)G_u(s,s')k_w^{ns}(v,u) \nonumber \\
		&\qquad \times  \int_{|u'| \leq 3N} k_w^{ns} (u,u') |h(s',X(s)-u(s-s'), u')|du'duds'ds \nonumber \\
		&\quad + \int_0^t \int_{0}^{s-\delta}  \int_{|u| \leq 2N} G_v(t,s)G_u(s,s')k_w^{ns}(v,u)\nonumber \\
		&\qquad \times  \int_{|u'| \leq 3N} k_w^{ns} (u,u') |h(s',X(s)-u(s-s'), u')|du'duds'ds \nonumber \\
		&\leq \frac{C_{p,q,\varepsilon,\gamma}}{(1+|v|)^{1+\frac{1}{p'}}} \delta \sup_{0\leq s \leq t} \Vert h(s) \Vert_{L^p_v L^\infty_x} \nonumber \\
		&\quad + C\int_0^t e^{-\frac{\nu(v)}{2}(t-s)} \int_0^{s-\delta} e^{-CN^{\gamma}(s-s')} \int_{|u|\leq 2N}   k_w^{ns}(v,u) \nonumber \\
		&\qquad \times \int_{|u'| \leq 3N} k_w^{ns} (u,u') |h(s',X(s)-u(s-s'), u')|du'duds'ds \nonumber \\
		&\leq \frac{C_{p,q,\varepsilon,\gamma}}{(1+|v|)^{1+\frac{1}{p'}}} \delta \sup_{0\leq s \leq t} \Vert h(s) \Vert_{L^p_v L^\infty_x}  \nonumber \\
		&\quad +C_{\gamma,N} \int_0^t e^{-\frac{\nu(v)}{2}(t-s)} \nonumber \\
		&\qquad \times \sup_{0\leq s' \leq s-\delta} \int_{|u| \leq 2N} k_w^{ns}(v,u) \int_{|u'|\leq 3N} k_w^{ns}(u,u') |h(s',X(s)-u(s-s'),u') | du' du ds \nonumber \\
		&\leq \frac{C_{p,q,\varepsilon,\gamma}}{(1+|v|)^{1+\frac{1}{p'}}} \delta \sup_{0\leq s \leq t} \Vert h(s) \Vert_{L^p_v L^\infty_x}  \nonumber \\
		&\quad + C_{p,q,\varepsilon,\gamma,N} \int_0^t e^{-\frac{\nu(v)}{2}(t-s)} \langle v \rangle^{\gamma-1}\nonumber \\
		&\qquad \times \sup_{0\leq s'  \leq s-\delta} \left(\int_{|u| \leq 2N} \int_{|u'| \leq 3N} |h(s',X(s)-u(s-s'),u')|^2 du'du\right)^{1/2} \nonumber \\
		&\leq  \frac{C_{p,q,\varepsilon,\gamma}}{(1+|v|)^{1+\frac{1}{p'}}} \delta \sup_{0\leq s \leq t} \Vert h(s) \Vert_{L^p_v L^\infty_x} +\frac{C_{p,q,\varepsilon,\gamma,N,\delta}}{1+|v|}\left[ \mathcal{E}(F_0)^{1/2} + \sup_{0\le s \le t}\|h(s)\|_{L^p_vL^\infty_x}^{\frac{p}{2p-2}}\mathcal{E}(F_0)^{\frac{p-2}{2p-2}}\right]. 
	\end{align}
	where  we have made a change of variables $u \mapsto y:= X(s)-u(s-s')$ with $\left|\frac{dy}{du} \right| =(s-s')^3$ and used \eqref{RTIRE} in the last inequality. Combining \eqref{II2,1}, \eqref{II2,ss}, \eqref{II2,2}, \eqref{II2,3}, and \eqref{II2,4}, $II_2$ in \eqref{II,esti1} can be further bounded by 
	\begin{align} \label{II2,final}
		II_2 &\leq  C_{p,\gamma}\ve^{\gamma+\frac{3}{p'}} \langle v\rangle^{-2}\sup_{0\leq s \leq t } \Vert h(s)\Vert_{L^p_vL^\infty_x}+\frac{C_{p,q,\varepsilon,\gamma}}{1+|v|} \left(\frac{1}{N}+\frac{1}{N^{-\gamma}}+\frac{1}{N^{\frac{\gamma+3}{2}}}+\delta  \right) \sup_{0\leq s \leq t} \Vert h(s) \Vert_{L^p_vL^\infty_x}\nonumber\\
		&\quad + \frac{C_{p,q,\varepsilon,\gamma,N,\delta}}{1+|v|} \left[ \mathcal{E}(F_0)^{1/2} + \sup_{0\le s \le t}\|h(s)\|_{L^p_vL^\infty_x}^{\frac{p}{2p-2}}\mathcal{E}(F_0)^{\frac{p-2}{2p-2}}\right].
	\end{align}
	To estimate $II_3$ in \eqref{II,esti1}, we first consider the case $-1 \leq \gamma <0$. \\
	
	\textbf{(Case 1. $-1\leq \gamma<0$)} We firstly use the estimate \eqref{pointwiseGamma+estimate1} in Lemma \ref{pointwiseGamma+estimate}, and then divide the integration domain into $\{|\eta| \geq 3N\}$ and $\{|\eta|\leq 3N\}$: 
	\begin{align} \label{II3}
		II_3 &\leq C_{p,\gamma,\beta} \sup_{0\leq s \leq t} \Vert h(s) \Vert_{L^p_vL^\infty_x}\int_0^t G_v(t,s) \int_{\R^3}k_w^{ns} (v,u) \int_0^s G_u(s,s') \frac{1}{(1+|u|)^{-\frac{\gamma}{p}+\frac{5p-5}{4p}}} \nonumber \\
		&\quad \times  \left(\int_{\R^3}(1+|\eta|)^{\frac{4p-8}{p}-2\beta} |h(s',X(s)-u(s-s'),\eta)|^2d\eta \right)^{1/2}ds'duds \nonumber \\
		&\leq  C_{p,\gamma,\beta} \sup_{0\leq s \leq t} \Vert h(s) \Vert_{L^p_vL^\infty_x} \int_0^t G_v(t,s) \int_{\R^3}k_w^{ns} (v,u) \int_0^s G_u(s,s') \frac{1}{(1+|u|)^{-\frac{\gamma}{p}+\frac{5p-5}{4p}}}\nonumber \\
		&\quad \times \Bigg[ \left(\int_{|\eta| \geq 3N} (1+|\eta|)^{\frac{4p-8}{p}-2\beta} |h(s',X(s)-u(s-s'),\eta)|^2d\eta \right)^{1/2}\nonumber\\
		& \qquad \quad   +\left(\int_{|\eta| \leq 3N} (1+|\eta|)^{\frac{4p-8}{p}-2\beta} |h(s',X(s)-u(s-s'),\eta)|^2d\eta \right)^{1/2}  \Bigg]ds'duds \nonumber \\
		&:=II_{3,1} + II_{3,2}.
	\end{align}
	For $II_{3,1}$, it follows from H\"{o}lder's inequality and Lemma \ref{Knonsing} that 
	\begin{align} \label{II3,1}
		II_{3,1} &\leq C_{p,\gamma,\beta} \sup_{0\leq s \leq t} \Vert h(s) \Vert_{L^p_vL^\infty_x}^2 \int_0^t G_v(t,s) \int_{\R^3}k_w^{ns} (v,u) \int_0^s G_u(s,s') \frac{1}{(1+|u|)^{-\frac{\gamma}{p}+\frac{5p-5}{4p}}} \nonumber \\
		&\quad \times \left(\int_{|\eta| \geq 3N} (1+|\eta|)^{4-\frac{2p\beta}{p-2}} d\eta \right)^{\frac{p-2}{2p}}  ds'duds \nonumber \\
		&\leq  \frac{C_{p,\gamma,\beta}}{N} \sup_{0\leq s \leq t} \Vert h(s) \Vert_{L^p_vL^\infty_x}^{2} \int_0^t G_v(t,s) \int_{\R^3}k_w^{ns} (v,u) \int_0^s G_u(s,s') \frac{1}{(1+|u|)^{-\frac{\gamma}{p}+\frac{5p-5}{4p}}}ds' du ds \nonumber \\
		&\leq \frac{C_{p,\gamma,\beta}}{N} \sup_{0\leq s \leq t} \Vert h(s) \Vert_{L^p_vL^\infty_x}^{2} \int_0^t G_v(t,s) \int_{\R^3}k_w^{ns} (v,u) \nonumber \\
		&\qquad \times  \int_0^s G_u(s,s') \nu(u) \frac{1}{(1+|u|)^{\gamma(1-\frac{1}{p})+\frac{5p-5}{4p}}}ds'duds\nonumber \\
		&\leq \frac{C_{p,\gamma,\beta}}{N} \sup_{0\leq s \leq t} \Vert h(s) \Vert_{L^p_vL^\infty_x}^{2} \int_0^t G_v(t,s) \int_{\R^3}k_w^{ns} (v,u) du ds\nonumber \\ 
		&\leq \frac{C_{p,\gamma,\beta,\varepsilon}}{N}\frac{\sup_{0\leq s \leq t} \Vert h(s) \Vert_{L^p_vL^\infty_x}^{2}}{(1+|v|)^2} \int_0^t G_v(t,s) \nu(v) ds \nonumber \\
		&\leq \frac{C_{p,\gamma,\beta,\varepsilon}}{N}\frac{\sup_{0\leq s \leq t} \Vert h(s) \Vert_{L^p_vL^\infty_x}^{2}}{(1+|v|)^2},
	\end{align}
	due to $\gamma\left(1-\frac{1}{p}\right)+\frac{5p-5}{4p}>0$ for $-1 \leq \gamma <0$ and $4-\frac{2p(\beta-1)}{p-2}<-3$ for $\beta>\frac{7(p-2)}{2p}+1$. 
	To estimate $II_{3,2}$, we split the integration domain in $u$ into the regions $\{|u| \geq 2N\}$ and $\{|u| \leq 2N\}$. 
	\begin{align} \label{II3,2,1}
		II_{3,2} &= C_{p,\gamma,\beta} \sup_{0\leq s \leq t} \Vert h(s) \Vert_{L^p_v L^\infty_x}  \int_0^t G_v(t,s) \int_{|u|\geq 2N}k_w^{ns} (v,u) \int_0^s G_u(s,s') \frac{1}{(1+|u|)^{-\frac{\gamma}{p}+\frac{5p-5}{4p}}} \nonumber \\
		&\qquad \times \left(\int_{|\eta| \leq 3N} (1+|\eta|)^{\frac{4p-8}{p}-2\beta} |h(s',X(s)-u(s-s'),\eta)|^2d\eta \right)^{1/2}ds'duds \nonumber \\
		&\quad +C_{p,\gamma,\beta} \sup_{0\leq s \leq t} \Vert h(s) \Vert_{L^p_v L^\infty_x}  \int_0^t G_v(t,s) \int_{|u|\leq 2N}k_w^{ns} (v,u) \int_0^s G_u(s,s') \frac{1}{(1+|u|)^{-\frac{\gamma}{p}+\frac{5p-5}{4p}}}\nonumber \\
		&\qquad \times \left(\int_{|\eta| \leq 3N} (1+|\eta|)^{\frac{4p-8}{p}-2\beta} |h(s',X(s)-u(s-s'),\eta)|^2d\eta \right)^{1/2}ds'duds\nonumber\\
		&\leq \frac{C_{p,\gamma,\beta}}{N} \sup_{0\leq s \leq t} \Vert h(s)\Vert_{L^p_v L^\infty_x}^2 \int_0^t G_v(t,s) \int_{|u| \geq 2N} k^{ns}_{w}(v,u) \int_0^s G_u(s,s') \nu(u) \nonumber\\
		&\qquad \times \frac{1}{(1+|u|)^{\gamma\left(1-\frac{1}{p}\right)+\frac{5p-5}{4p}-1}}ds'duds \nonumber \\
		&\quad +C_{p,\gamma,\beta} \sup_{0\leq s \leq t} \Vert h(s) \Vert_{L^p_v L^\infty_x}  \int_0^t G_v(t,s) \int_{|u|\leq 2N}k_w^{ns} (v,u) \int_0^s G_u(s,s') \frac{1}{(1+|u|)^{-\frac{\gamma}{p}+\frac{5p-5}{4p}}}\nonumber \\
		&\qquad \times \left(\int_{|\eta| \leq 3N} (1+|\eta|)^{\frac{4p-8}{p}-2\beta} |h(s',X(s)-u(s-s'),\eta)|^2d\eta \right)^{1/2}ds'duds \nonumber \\
		&\leq \frac{C_{p,\gamma,\beta,\varepsilon}}{N} \frac{\sup_{0\leq s \leq t} \Vert h(s)\Vert_{L^p_v L^\infty_x}^2}{(1+|v|)^2} \nonumber \\ 
		&\quad +C_{p,\gamma,\beta} \sup_{0\leq s \leq t} \Vert h(s) \Vert_{L^p_v L^\infty_x}  \int_0^t G_v(t,s) \int_{|u|\leq 2N}k_w^{ns} (v,u) \int_0^s G_u(s,s') \frac{1}{(1+|u|)^{-\frac{\gamma}{p}+\frac{5p-5}{4p}}}\nonumber \\
		&\qquad \times \left(\int_{|\eta| \leq 3N} (1+|\eta|)^{\frac{4p-8}{p}-2\beta} |h(s',X(s)-u(s-s'),\eta)|^2d\eta \right)^{1/2}ds'duds,
	\end{align}
	where the last inequality comes from $\gamma\left(1-\frac{1}{p}\right)+\frac{5p-5}{4p}-1>0$ for  $p>7-4\gamma$, and $4-\frac{2p\beta}{p-2}<-3$  for $\beta >9/2$. In a similar way to \eqref{II2,4}, the remaining term can be bounded as follows: 
	\begin{align}\label{II3,2,2}
		&\int_0^t G_v(t,s) \int_{|u|\leq 2N}k_w^{ns} (v,u) \int_{0}^s G_u(s,s') \frac{1}{(1+|u|)^{-\frac{\gamma}{p}+\frac{5p-5}{4p}}} \nonumber \\
		&\quad \times \left(\int_{|\eta| \leq 3N} (1+|\eta|)^{\frac{4p-8}{p}-2\beta} |h(s',X(s)-u(s-s'),\eta)|^2d\eta \right)^{1/2}ds'duds \nonumber \\
		&=\int_0^t G_v(t,s) \int_{|u|\leq 2N}k_w^{ns} (v,u) \int_{s-\delta}^s G_u(s,s') \frac{1}{(1+|u|)^{-\frac{\gamma}{p}+\frac{5p-5}{4p}}} \nonumber \\
		&\qquad \times \left(\int_{|\eta| \leq 3N} (1+|\eta|)^{\frac{4p-8}{p}-2\beta} |h(s',X(s)-u(s-s'),\eta)|^2d\eta \right)^{1/2}ds'duds \nonumber \\
		&\quad + \int_0^t G_v(t,s) \int_{|u|\leq 2N}k_w^{ns} (v,u) \int_0^{s-\delta} G_u(s,s') \frac{1}{(1+|u|)^{-\frac{\gamma}{p}+\frac{5p-5}{4p}}} \nonumber \\
		&\qquad \times \left(\int_{|\eta| \leq 3N} (1+|\eta|)^{\frac{4p-8}{p}-2\beta} |h(s',X(s)-u(s-s'),\eta)|^2d\eta \right)^{1/2}ds'duds \nonumber \\
		&\leq \frac{C_{p,q,\varepsilon,\gamma}}{1+|v|} \delta \sup_{0 \leq s \leq t}\Vert h(s) \Vert_{L^p_v L^\infty_x} + C\int_0^t e^{-\frac{\nu(v)}{2}(t-s)} \int_0^{s-\delta} e^{-CN^{\gamma}(s-s')} \int_{|u|\leq 2N}  k_w^{ns}(v,u) \nonumber \\
		&\qquad \times \left(\int_{|\eta| \leq 3N} |h(s',X(s)-u(s-s'),\eta))|^2 d\eta \right)^{1/2} duds'ds \nonumber \\
		&\leq  \frac{C_{p,q,\varepsilon,\gamma}}{1+|v|} \delta \sup_{0 \leq s \leq t}\Vert h(s) \Vert_{L^p_v L^\infty_x} +C_{p,q,\varepsilon,\gamma,N} \int_0^t e^{-\frac{\nu(v)}{2}(t-s)}\langle v \rangle^{\gamma-1} \nonumber \\
		&\qquad \times \sup_{0\leq s' \leq s-\delta} \left(\int_{|u|\leq2N} \int_{|\eta| \leq 3N} |h(s',X(s)-u(s-s'),\eta)|^2 d\eta du\right)^{1/2} \nonumber \\
		&\leq  \frac{C_{p,q,\varepsilon,\gamma}}{1+|v|} \delta \sup_{0 \leq s \leq t}\Vert h(s) \Vert_{L^p_v L^\infty_x} + \frac{C_{p,q,\varepsilon,\gamma,N,\delta}}{1+|v|}  \left[\mathcal{E}(F_0)^{1/2} + \|h(s)\|_{L^p_vL^\infty_x}^{\frac{p}{2p-2}}\mathcal{E}(F_0)^{\frac{p-2}{2p-2}}\right],
	\end{align}
	where we have made a change of variables $u \mapsto y:=X(s)-u(s-s')$ with $\left|\frac{dy}{du} \right| =(s-s')^3$ and used \eqref{RTIRE} in the last inequality. Combining \eqref{II3}, \eqref{II3,1}, \eqref{II3,2,1}, and \eqref{II3,2,2}, we can bound the $II_3$ term in \eqref{II3} for the case $-1 \leq \gamma<0$ as 
	\begin{align} \label{II3,final}
		II_3 &\leq \frac{C_{p,q,\gamma,\beta,\varepsilon}}{1+|v|}\left(\frac{1}{N}+\delta \right) \sup_{0\leq s \leq t} \Vert h(s) \Vert_{L^p_vL^\infty_x}^{2} \nonumber\\
		& \quad + \frac{C_{p,q,\beta,\varepsilon,\gamma,N,\delta}}{1+|v|} \left[ \sup_{0 \leq s \leq t} \Vert h(s) \Vert_{L^p_vL^\infty_x} \mathcal{E}(F_0)^{1/2} +\|h(s)\|_{L^p_vL^\infty_x}^{\frac{3p-2}{2p-2}}\mathcal{E}(F_0)^{\frac{p-2}{2p-2}}\right].
	\end{align}
	\bigskip
	
	Next, we consider the term $II_3$ for the case $-3< \gamma<-1$.\\
	
	\textbf{(Case 2. $-3< \gamma<-1$)} Similar to \textbf{Case 1}, we apply \eqref{pointwiseGamma+estimate2} in Lemma \ref{pointwiseGamma+estimate} to the term $II_3$, and then divide the integration domain into $\{|\eta| \geq 3N\}$ and $\{|\eta| \leq 3N\}$: 
	\begin{align} \label{2,II3}
		II_3 &\leq C_{p,\gamma,\beta} \sup_{0\leq s \leq t} \Vert h(s) \Vert_{L^p_v L^\infty_x}\int_0^t G_v(t,s) \int_{\R^3} k_w^{ns}(v,u) \int_0^s G_u(s,s') \frac{1}{(1+|u|)^{-\gamma+\frac{p-1}{p}\varpi}} \nonumber\\
		&\qquad \times \left(\int_{\R^3} (1+|\eta|)^{\frac{4}{m-1}-p'm'\beta} |h(s',X(s)-u(s-s'),\eta)|^{p'm'} d\eta\right)^{\frac{1}{p'm'}} 	\nonumber \\
		&\leq C_{p,\gamma,\beta} \sup_{0\leq s \leq t} \Vert h(s) \Vert_{L^p_v L^\infty_x}\int_0^t G_v(t,s) \int_{\R^3} k_w^{ns}(v,u) \int_0^s G_u(s,s') \frac{1}{(1+|u|)^{-\gamma+\frac{p-1}{p}\varpi}} \nonumber \\
		&\quad \times \left[\left(\int_{|\eta| \geq 3N} (1+|\eta|)^{\frac{4}{m-1}-p'm'\beta} |h(s'X(s)-u(s-s'),\eta)|^{p'm'} d\eta\right)^{\frac{1}{p'm'}} \right. \nonumber \\
		&\qquad \left. + \left(\int_{|\eta| \leq 3N} (1+|\eta|)^{\frac{4}{m-1}-p'm'\beta} |h(s',X(s)-u(s-s'),\eta)|^{p'm'} d\eta\right)^{\frac{1}{p'm'}} \right]ds'duds \nonumber \\
		&:= II_{3,1} + II_{3,2}  
	\end{align}
	Similar to \eqref{II3,1} in \textbf{Case 1}, the term $II_{3,1}$ can be bounded by  
	\begin{align} \label{2,II3,1}
		II_{3,1}\leq \frac{C_{p,\gamma,\beta,\varepsilon}}{N} \frac{1}{(1+|v|)^2}\sup_{0\leq s \leq t}\Vert h(s) \Vert_{L^p_vL^\infty_x}^2,
	\end{align}
	due to $\left(\frac{4}{m-1}-p'm'(\beta-1)\right)\left(\frac{p}{p-p'm'}\right) < -3,$ which is equivalent to
	\begin{align*}
		\beta > \frac{4pm+p-6m-1}{pm}
		\quad \text{for } \beta > 5.
	\end{align*}
	We divide the $\eta$-integration region into $\{|u| \leq 2N\}$ and $\{|u| \geq 2N\}$ to estimate $II_{3,2}$
	\begin{align*}
		II_{3,2} &= C_{p,\gamma,\beta} \sup_{0\leq s \leq t} \Vert h(s) \Vert_{L^p_v L^\infty_x}\int_0^t G_v(t,s) \int_{|u| \geq 2N} k_w^{ns}(v,u) \int_0^s G_u(s,s') \frac{1}{(1+|u|)^{-\gamma+\frac{p-1}{p}\varpi}}\\
		&\qquad \times  \left(\int_{|\eta| \leq 3N} (1+|\eta|)^{\frac{4}{m-1}-p'm'\beta} |h(s',X(s)-u(s-s'),\eta)|^{p'm'} d\eta\right)^{\frac{1}{p'm'}} ds'duds\\
		&\quad + C_{p,\gamma,\beta} \sup_{0\leq s \leq t} \Vert h(s) \Vert_{L^p_v L^\infty_x}\int_0^t G_v(t,s) \int_{|u| \leq 2N} k_w^{ns}(v,u) \int_0^s G_u(s,s') \frac{1}{(1+|u|)^{-\gamma+\frac{p-1}{p}\varpi}}\\
		&\qquad \times  \left(\int_{|\eta| \leq 3N} (1+|\eta|)^{\frac{4}{m-1}-p'm'\beta} |h(s',X(s)-u(s-s'),\eta)|^{p'm'} d\eta\right)^{\frac{1}{p'm'}} ds'duds\\
		&\leq \frac{C_{p,\gamma,\beta,\varepsilon}}{N} \frac{1}{1+|v|} \sup_{0\leq s \leq t} \Vert h(s) \Vert_{L^p_v L^\infty_x}^2\\
		& \quad + C_{p,\gamma,\beta} \sup_{0\leq s \leq t} \Vert h(s) \Vert_{L^p_v L^\infty_x}\int_0^t G_v(t,s) \int_{|u| \leq 2N} k_w^{ns}(v,u) \int_0^s G_u(s,s') \frac{1}{(1+|u|)^{-\gamma+\frac{p-1}{p}\varpi}}\\
		&\qquad \times  \left(\int_{|\eta| \leq 3N} (1+|\eta|)^{\frac{4}{m-1}-p'm'\beta} |h(s',X(s)-u(s-s'),\eta)|^{p'm'} d\eta\right)^{\frac{1}{p'm'}} ds'duds,
	\end{align*}
	where the last inequality comes from $\frac{p-1}{p}\varpi-\gamma>0$. To complete the estimate for $II_{3,2}$, we further bound the following integral:
	\begin{align*}
		&\int_0^t G_v(t,s) \int_{|u| \leq 2N} k_w^{ns}(v,u) \int_0^s G_u(s,s') \frac{1}{(1+|u|)^{-\gamma+\frac{p-1}{p}\varpi}}
		\\
		&\quad \times \left(\int_{|\eta| \leq 3N} (1+|\eta|)^{\frac{4}{m-1}-p'm'\beta} |h(s',X(s)-u(s-s'),\eta)|^{p'm'} d\eta\right)^{\frac{1}{p'm'}} ds'duds\\
		&= \int_0^t G_v(t,s) \int_{|u| \leq 2N} k_w^{ns}(v,u) \int_0^{s-\delta} G_u(s,s') \frac{1}{(1+|u|)^{-\gamma+\frac{p-1}{p}\varpi}}\\
		&\qquad \times \left(\int_{|\eta| \leq 3N} (1+|\eta|)^{\frac{4}{m-1}-p'm'\beta} |h(s',X(s)-u(s-s'),\eta)|^{p'm'} d\eta\right)^{\frac{1}{p'm'}} ds'duds\\
		&\quad + \int_0^t G_v(t,s) \int_{|u| \leq 2N} k_w^{ns}(v,u) \int_{s-\delta}^s G_u(s,s') \frac{1}{(1+|u|)^{-\gamma+\frac{p-1}{p}\varpi}}\\
		&\qquad \times \left(\int_{|\eta| \leq 3N} (1+|\eta|)^{\frac{4}{m-1}-p'm'\beta} |h(s',X(s)-u(s-s'),\eta)|^{p'm'} d\eta\right)^{\frac{1}{p'm'}} ds'duds\\
		&\leq \frac{C_{p,q,\varepsilon,\gamma}}{1+|v|} \delta \sup_{0\leq s \leq t} \Vert h(s) \Vert_{L^p_v L^\infty_x} + \frac{C_{p,\gamma,\beta,\varepsilon,N,\delta}}{1+|v|}  \Bigg[\sup_{0 \leq s \leq t}\|h(s)\|_{L^p_vL^\infty_x}^{1-\frac{1/p'm'-1/p}{1/2-1/p}}\mathcal{E}(F_0)^{\frac{1}{2}\cdot\frac{1/p'm'-1/p}{1/2-1/p}}\\
		& \quad  + \sup_{0 \leq s \leq t}\|h(s)\|_{L^p_vL^\infty_x}^{1-\frac{1/p'm'-1/p}{1-1/p}}\mathcal{E}(F_0)^{\frac{1/p'm'-1/p}{1-1/p}}\Bigg], 
	\end{align*}
	where the last inequality comes from $\frac{p-1}{p}\varpi-\gamma>0$ and the interpolation used in \eqref{case2.I5}. Thus, in Case 2, we can further bound the term $II_{3,2}$ by 
	\begin{align} \label{2,II3,2}
		II_{3,2} &\leq \frac{C_{p,q,\varepsilon,\gamma,\beta}}{1+|v|} \left(\frac{1}{N} + \delta \right) \sup_{0\leq s \leq t} \Vert h(s) \Vert_{L^p_vL^\infty_x}^2 \nonumber  \\
		&\quad + \frac{C_{p,\gamma,\beta,\varepsilon,N,\delta}}{1+|v|}  \Bigg[\sup_{0 \leq s \leq t}\|h(s)\|_{L^p_vL^\infty_x}^{2-\frac{1/p'm'-1/p}{1/2-1/p}}\mathcal{E}(F_0)^{\frac{1}{2}\cdot\frac{1/p'm'-1/p}{1/2-1/p}}\nonumber\\
		& \qquad  + \sup_{0 \leq s \leq t}\|h(s)\|_{L^p_vL^\infty_x}^{2-\frac{1/p'm'-1/p}{1-1/p}}\mathcal{E}(F_0)^{\frac{1/p'm'-1/p}{1-1/p}}\Bigg]. 
	\end{align}
	For $-3<\gamma<-1$, gathering \eqref{2,II3}, \eqref{2,II3,1}, and \eqref{2,II3,2} gives 
	\begin{align}\label{2,II3,final} 
		II_3 &\leq \frac{C_{p,q,\varepsilon,\gamma,\beta}}{1+|v|} \left(\frac{1}{N} + \delta \right) \sup_{0\leq s \leq t} \Vert h(s) \Vert_{L^p_vL^\infty_x}^2 \nonumber \\
		&\quad + \frac{C_{p,\gamma,\beta,\varepsilon,N,\delta}}{1+|v|}  \Bigg[\sup_{0 \leq s \leq t}\|h(s)\|_{L^p_vL^\infty_x}^{2-\frac{1/p'm'-1/p}{1/2-1/p}}\mathcal{E}(F_0)^{\frac{1}{2}\cdot\frac{1/p'm'-1/p}{1/2-1/p}}\nonumber\\
		& \qquad  + \sup_{0 \leq s \leq t}\|h(s)\|_{L^p_vL^\infty_x}^{2-\frac{1/p'm'-1/p}{1-1/p}}\mathcal{E}(F_0)^{\frac{1/p'm'-1/p}{1-1/p}}\Bigg]. 
	\end{align}
	Consequently, for $-1 \leq \gamma <0$, it follows from \eqref{K,est1}, \eqref{II,esti1}, \eqref{II1,final}, \eqref{II2,final}, and \eqref{II3,final} that 
	\begin{align} \label{II,final}
		II &\leq  C_{p,\rho,\lambda,\gamma}e^{-\frac{\lambda}{2}(1+t)^{\rho}} \langle v \rangle^{\gamma-1-\frac{1}{p'}} \Vert h_0 \Vert_{L^p_v L^\infty_x} +C_{p,\gamma} \varepsilon^{\gamma+ \frac{3}{p'}} \mu(v)^{\frac{1-q}{16}}  \sup_{0 \leq s \leq t} \Vert h(s) \Vert_{L^p_v L^\infty_x}\nonumber \\
		& \quad  +\frac{C_{p,q,\varepsilon,\gamma}}{1+|v|} \left(\frac{1}{N}+\frac{1}{N^{-\gamma}}+\frac{1}{N^{\frac{\gamma+3}{2}}}+\delta  \right) \sup_{0\leq s \leq t} \Vert h(s) \Vert_{L^p_vL^\infty_x}
		\nonumber \\
		&\quad + \frac{C_{p,q,\gamma,\beta,\varepsilon}}{1+|v|}\left(\frac{1}{N}+\delta \right)\sup_{0\leq s \leq t}\Vert h(s)\Vert_{L^p_v L^\infty_x}^2 \nonumber \\
		&\quad + \frac{C_{p,q,\varepsilon,\gamma,N,\delta}}{1+|v|} \left[ \mathcal{E}(F_0)^{1/2} + \sup_{0\le s \le t}\|h(s)\|_{L^p_vL^\infty_x}^{\frac{p}{2p-2}}\mathcal{E}(F_0)^{\frac{p-2}{2p-2}}\right]\nonumber\\
		& \quad + \frac{C_{p,q,\beta,\varepsilon,\gamma,N,\delta}}{1+|v|} \left[ \sup_{0 \leq s \leq t} \Vert h(s) \Vert_{L^p_vL^\infty_x} \mathcal{E}(F_0)^{1/2} +\|h(s)\|_{L^p_vL^\infty_x}^{\frac{3p-2}{2p-2}}\mathcal{E}(F_0)^{\frac{p-2}{2p-2}}\right].
	\end{align}
	For $-3<\gamma<-1$, we complete the estimate for $II$ from \eqref{K,est1}, \eqref{II,esti1}, \eqref{II1,final}, \eqref{II2,final}, and \eqref{2,II3,final} 
	\begin{align} \label{2,II,final} 
		II & \leq C_{p,\rho,\lambda,\gamma}e^{-\frac{\lambda}{2}(1+t)^{\rho}} \langle v \rangle^{\gamma} \Vert h_0 \Vert_{L^p_v L^\infty_x}+C_{p,\gamma} \varepsilon^{\gamma+ \frac{3}{p'}} \mu(v)^{\frac{1-q}{16}}  \sup_{0 \leq s \leq t} \Vert h(s) \Vert_{L^p_v L^\infty_x} \nonumber \\
		& \quad  +\frac{C_{p,q,\varepsilon,\gamma}}{1+|v|} \left(\frac{1}{N}+\frac{1}{N^{-\gamma}}+\frac{1}{N^{\frac{\gamma+3}{2}}}+\delta  \right) \sup_{0\leq s \leq t} \Vert h(s) \Vert_{L^p_vL^\infty_x}
		\nonumber \\
		&\quad + \frac{C_{p,q,\gamma,\beta,\varepsilon}}{1+|v|}\left(\frac{1}{N}+\delta \right)\sup_{0\leq s \leq t}\Vert h(s)\Vert_{L^p_v L^\infty_x}^2 \nonumber \\
		&\quad + \frac{C_{p,q,\varepsilon,\gamma,N,\delta}}{1+|v|} \left[ \mathcal{E}(F_0)^{1/2} + \sup_{0\le s \le t}\|h(s)\|_{L^p_vL^\infty_x}^{\frac{p}{2p-2}}\mathcal{E}(F_0)^{\frac{p-2}{2p-2}}\right]\nonumber\\
		& \quad +\frac{C_{p,\gamma,\beta,\varepsilon,N,\delta}}{1+|v|}  \Bigg[\sup_{0 \leq s \leq t}\|h(s)\|_{L^p_vL^\infty_x}^{2-\frac{1/p'm'-1/p}{1/2-1/p}}\mathcal{E}(F_0)^{\frac{1}{2}\cdot\frac{1/p'm'-1/p}{1/2-1/p}} \nonumber \\
		&\qquad + \sup_{0 \leq s \leq t}\|h(s)\|_{L^p_vL^\infty_x}^{2-\frac{1/p'm'-1/p}{1-1/p}}\mathcal{E}(F_0)^{\frac{1/p'm'-1/p}{1-1/p}}\Bigg] .
	\end{align}
	For $III$, first of all, let us consider the case $-1\leq \gamma <0$. \\
	
	\textbf{(Case 1. $-1\leq \gamma <0$)} We use \eqref{pointwiseGamma+estimate1} in Lemma \ref{pointwiseGamma+estimate} to obtain the following estimate 
	\begin{align} \label{III}
		III &\leq C_{p,\gamma,\beta}  \int_0^t G_v(t,s) \frac{\Vert h(s) \Vert_{L^p_v L^\infty_x}}{(1+|v|)^{-\frac{\gamma}{p}+\frac{5p-5}{4p}}} \left(\int_{\R^3} (1+|\eta|)^{\frac{4p-8}{p}-2\beta} |h(s,X(s),\eta)|^2 d \eta \right)^{1/2}ds   \nonumber \\
		&\le C_{p,\gamma,\beta} \int_0^t G_v(t,s) \frac{\Vert h(s) \Vert_{L^p_v L^\infty_x}}{(1+|v|)^{-\frac{\gamma}{p}+\frac{5p-5}{4p}}} \left(\int_{|\eta| \leq N} (1+|\eta|)^{\frac{4p-8}{p}-2\beta} |h(s,X(s),\eta)|^2 d \eta \right)^{1/2}ds \nonumber\\ 
		&\quad + C_{p,\gamma,\beta}  \int_0^t G_v(t,s) \frac{\Vert h(s) \Vert_{L^p_v L^\infty_x}}{(1+|v|)^{-\frac{\gamma}{p}+\frac{5p-5}{4p}}} \left(\int_{|\eta| \geq N} (1+|\eta|)^{\frac{4p-8}{p}-2\beta} |h(s,X(s),\eta)|^2 d \eta \right)^{1/2}ds \nonumber \\
		&\leq C_{p,\gamma,\beta}  \int_0^t G_v(t,s) \frac{\Vert h(s) \Vert_{L^p_v L^\infty_x}}{(1+|v|)^{-\frac{\gamma}{p}+\frac{5p-5}{4p}}} \left(\int_{|\eta| \leq N} (1+|\eta|)^{\frac{4p-8}{p}-2\beta} |h(s,X(s),\eta)|^2 d \eta \right)^{1/2}ds \nonumber \\
		&\quad + \frac{C_{p,\gamma,\beta}}{N} \sup_{0\leq s \leq t} \Vert h(s) \Vert_{L^p_v L^\infty_x}^2 \int_0^t G_v(t,s) \nu(v) \frac{1}{(1+|v|)^{\frac{p-1}{p}(\frac{5}{4}+\gamma)}} \nonumber \\
		&\qquad \times \left(\int_{|\eta| \geq N } {(1+|\eta|)^{4- \frac{2p}{p-2}(\beta-1)}} d\eta \right)^{\frac{p-2}{2p}}ds \nonumber \\
		&\leq C_{p,\gamma,\beta}  \int_0^t G_v(t,s) \frac{\Vert h(s) \Vert_{L^p_v L^\infty_x}}{(1+|v|)^{-\frac{\gamma}{p}+\frac{5p-5}{4p}}} \left(\int_{|\eta| \leq N} (1+|\eta|)^{\frac{4p-8}{p}-2\beta} |h(s,X(s),\eta)|^2 d \eta \right)^{1/2}ds \nonumber \\
		&\quad + \frac{C_{p,\gamma,\beta}}{N} \frac{\sup_{0\leq s \leq t}\Vert h(s) \Vert_{L^p_v L^\infty_x}^2}{(1+|v|)^{\frac{p-1}{4p}}},
	\end{align} 
	where the last inequality comes from $4-\frac{2p}{p-2}(\beta-1)<-3$ for $\beta> \frac{7}{2}$. To complete the estimate for $III$, we apply Duhamel's principle to the first term in \eqref{III} once more: 
	\begin{align*}
		&\int_0^t G_v(t,s) \frac{\Vert h(s) \Vert_{L^p_v L^\infty_x}}{(1+|v|)^{-\frac{\gamma}{p}+\frac{5p-5}{4p}}} \left(\int_{|\eta| \leq N} (1+|\eta|)^{\frac{4p-8}{p}-2\beta} |h(s,X(s),\eta)|^2 d \eta \right)^{1/2}ds \\
		&\leq \int_0^t G_v(t,s)  \frac{\Vert h(s) \Vert_{L^p_v L^\infty_x}}{(1+|v|)^{-\frac{\gamma}{p}+\frac{5p-5}{4p}}} \left( \int_{|\eta|\leq N}  (1+|\eta|)^{\frac{4p-8}{p}-2\beta} G_\eta^2 (s,0)| h_0 (X(s)-\eta s,\eta)|^2 d\eta \right) ^{1/2}ds \\
		&\quad +\sup_{0\leq s \leq t}\Vert h(s) \Vert_{L^p_vL^\infty_x} \int_0^t G_v(t,s) \frac{1}{(1+|v|)^{-\frac{\gamma}{p}+\frac{5p-5}{4p}}} \\
		&\qquad \times \left(\int_{|\eta| \leq N} (1+|\eta|)^{\frac{4p-8}{p}-2\beta} \left( \int_0^s G_\eta (s,\tau) |K_wh|(\tau) d\tau \right)^2 d\eta \right)^{1/2}ds\\
		&\quad + \sup_{0\leq s \leq t}\Vert h(s) \Vert_{L^p_vL^\infty_x} \int_0^t G_v(t,s) \frac{1}{(1+|v|)^{-\frac{\gamma}{p}+\frac{5p-5}{4p}}} \\
		&\qquad \times \left(\int_{|\eta| \leq N} (1+|\eta|)^{\frac{4p-8}{p}-2\beta} \left( \int_0^s G_\eta (s,\tau) |w\Gamma^+(f,f)|(\tau) d\tau \right)^2 d\eta \right)^{1/2}ds \\
		&:= III_1 + \sup_{0\leq s \leq t}\Vert h(s) \Vert_{L^p_vL^\infty_x}\times [III_2 + III_3]. 
	\end{align*}
	For $III_1$, we easily get the following bound as 
	\begin{align}  \label{III,1}
		III_1 \leq \frac{C}{(1+|v|)^{-\frac{\gamma}{p}+\frac{5p-5}{4p}}}e^{-\lambda (1+t)^\rho} \Vert h_0 \Vert_{L^p_v L^\infty_x} \int_0^t \Vert h(s) \Vert_{L^p_vL^\infty_x}ds . 
	\end{align}
	For $III_2$, we split the $s$-integration interval into $[0,s]$ and $[s-\delta,s]$ as follows:
	\begin{align*}
		III_2 &\le  \int_0^t G_v(t,s) \frac{1}{(1+|v|)^{-\frac{\gamma}{p}+\frac{5p-5}{4p}}}\\
		&\qquad \times  \left(\int_{|\eta| \leq N} (1+|\eta|)^{\frac{4p-8}{p}-2\beta} \left( \int_0^{s-\delta} G_\eta (s,\tau) |K_wh|(\tau) d\tau \right)^2 d\eta \right)^{1/2}ds\\
		&\quad +  \int_0^t G_v(t,s) \frac{1}{(1+|v|)^{-\frac{\gamma}{p}+\frac{5p-5}{4p}}} \\
		&\qquad \times \left(\int_{|\eta| \leq N} (1+|\eta|)^{\frac{4p-8}{p}-2\beta} \left( \int_{s-\delta}^s G_\eta (s,\tau) |K_wh|(\tau) d\tau \right)^2 d\eta \right)^{1/2}ds\\
		&\leq  \int_0^t G_v(t,s) \frac{1}{(1+|v|)^{-\frac{\gamma}{p}+\frac{5p-5}{4p}}}\\
		&\qquad \times  \left(\int_{|\eta| \leq N} (1+|\eta|)^{\frac{4p-8}{p}-2\beta} \left( \int_0^{s-\delta} G_\eta (s,\tau) |K_wh|(\tau) d\tau \right)^2 d\eta \right)^{1/2}ds\\
		&\quad + \frac{C}{(1+|v|)^{\frac{p-1}{4p}}} \delta \sup_{0\leq s \leq t} \Vert h(s) \Vert_{L^p_v L^\infty_x}. 
	\end{align*}
	To close the estimate for $III_2$, by using Lemma \ref{Kker}, Lemma \ref{Ksingular}, and a similar argument in \eqref{K,est1}, let us consider the following integral
	\begin{align*}
		&\int_0^t G_v(t,s) \frac{1}{(1+|v|)^{-\frac{\gamma}{p}+\frac{5p-5}{4p}}} \left(\int_{|\eta| \leq N} (1+|\eta|)^{\frac{4p-8}{p}-2\beta} \left( \int_0^{s-\delta} G_\eta (s,\tau) |K_wh|(\tau) d\tau \right)^2 d\eta \right)^{1/2}ds \\
		&\leq \int_0^t G_v(t,s) \frac{1}{(1+|v|)^{-\frac{\gamma}{p}+\frac{5p-5}{4p}}} 
		 \left(\int_{|\eta| \leq N} (1+|\eta|)^{\frac{4p-8}{p}-2\beta} \right. \\ 
		 &\left. \qquad \times \left( \int_0^{s-\delta} G_\eta (s,\tau) \int_{\R^3} |k_w^{ns}(\eta,\eta')||h(\tau,X(s)-\eta(s-\tau),\eta')| d\eta' d\tau \right)^2 d\eta \right)^{1/2}ds\\
		&\quad + \int_0^t G_v(t,s) \frac{1}{(1+|v|)^{-\frac{\gamma}{p}+\frac{5p-5}{4p}}}\\
		&\qquad \times  \left(\int_{|\eta| \leq N} (1+|\eta|)^{\frac{4p-8}{p}-2\beta} \left( \int_0^{s-\delta} G_\eta (s,\tau) |K_w^{s}h|(\tau) d\tau \right)^2 d\eta \right)^{1/2}ds\\
		&\leq \int_0^t G_v(t,s) \frac{1}{(1+|v|)^{-\frac{\gamma}{p}+\frac{5p-5}{4p}}}  \left(\int_{|\eta| \leq N} (1+|\eta|)^{\frac{4p-8}{p}-2\beta} \right. \\ 
		&\left. \qquad \times  \left( \int_0^{s-\delta} G_\eta (s,\tau) \int_{|\eta'| \leq 2N} |k_w^{ns}(\eta,\eta')||h(\tau,X(s)-\eta(s-\tau),\eta')| d\eta' d\tau \right)^2 d\eta \right)^{1/2}ds\\
		&\quad + \int_0^t G_v(t,s) \frac{1}{(1+|v|)^{-\frac{\gamma}{p}+\frac{5p-5}{4p}}} \left(\int_{|\eta| \leq N} (1+|\eta|)^{\frac{4p-8}{p}-2\beta} \right. \\ 
		&\left. \qquad \times \left( \int_0^{s-\delta} G_\eta (s,\tau) \int_{|\eta'| \geq 2N} |k_w^{ns}(\eta,\eta')||h(\tau,X(s)-\eta(s-\tau),\eta')| d\eta' d\tau \right)^2 d\eta \right)^{1/2}ds\\ 
		&\quad + C_{p,\gamma} \varepsilon^{\gamma+\frac{3}{p'}} \sup_{0\leq s \leq t} \Vert h(s) \Vert_{L^p_v L^\infty_x} \int_0^t G_v(t,s)   \frac{1}{(1+|v|)^{-\frac{\gamma}{p}+\frac{5p-5}{4p}}}\\ 
		&\qquad \times \left(\int_{|\eta| \leq N}(1+|\eta|)^{\frac{4p-8}{p}-2\beta} \left(\int_0^{s-\delta} G_{\eta}(s,\tau) \nu(\eta) \nu^{-1}(\eta) \mu(\eta)^{\frac{1-q}{8}} d\tau \right) d\eta \right)^2 ds\\
		&\leq  \int_0^t G_v(t,s) \frac{1}{(1+|v|)^{-\frac{\gamma}{p}+\frac{5p-5}{4p}}} 
		 \left(\int_{|\eta| \leq N} (1+|\eta|)^{\frac{4p-8}{p}-2\beta} \right. \\ 
		 &\left. \qquad \times \left( \int_0^{s-\delta} G_\eta (s,\tau) \int_{|\eta'|\leq 2N} |k_w^{ns}(\eta,\eta')||h(\tau,X(s)-\eta(s-\tau),\eta')| d\eta' d\tau \right)^2 d\eta \right)^{1/2}ds\\
		&\quad + \int_0^t G_v(t,s) \frac{1}{(1+|v|)^{-\frac{\gamma}{p}+\frac{5p-5}{4p}}} \\
		&\qquad \times  \left(\int_{|\eta| \leq N} (1+|\eta|)^{\frac{4p-8}{p}-2\beta} \left( \int_0^{s-\delta} G_\eta (s,\tau) \int_{|\eta'| \geq 2N} \left[|\eta-\eta'|^\gamma e^{-\frac{|\eta|^2}{4}}e^{-\frac{|\eta'|^2}{4}}  \right. \right. \right. \\
		&\left. \left. \left. \qquad \quad + \frac{1}{|\eta-\eta'|^{\frac{3-\gamma}{2}}} e^{-\frac{|\eta-\eta'|^2}{8}}e^{-\frac{||\eta|^2-|\eta'|^2|^2}{8|\eta-\eta'|^2}} \right] \frac{w(\eta)}{w(\eta')}|h(\tau,X(s)-\eta(s-\tau), \eta')| d\eta' d\tau \right)^2 d\eta \right)^{1/2}ds\\
		&\quad + \frac{C_{p,\gamma}}{(1+|v|)^{\frac{p-1}{4p}}}  \varepsilon^{\gamma+\frac{3}{p'}} \sup_{0\leq s \leq t}  \Vert h(s) \Vert_{L^p_v L^\infty _x}\\
		&\leq  \int_0^t G_v(t,s) \frac{1}{(1+|v|)^{-\frac{\gamma}{p}+\frac{5p-5}{4p}}} 
		\left(\int_{|\eta| \leq N} (1+|\eta|)^{\frac{4p-8}{p}-2\beta} \right. \\ 
		&\left. \qquad \times \left( \int_0^{s-\delta} G_\eta (s,\tau) \int_{|\eta'|\leq 2N} |k_w^{ns}(\eta,\eta')||h(\tau,X(s)-\eta(s-\tau),\eta')| d\eta' d\tau \right)^2 d\eta \right)^{1/2}ds\\
		&\quad + C_{p,\gamma}\left(\frac{1}{N^{-\gamma}}+ \frac{1}{N^{\frac{3+\gamma}{2}}}+\varepsilon^{\gamma+\frac{3}{p'}}\right)\frac{\sup_{0\leq s \leq t} \Vert h(s)\Vert_{L^p_vL^\infty_x} }{(1+|v|)^{\frac{p-1}{4p}}},
	\end{align*}
	where we have used $|\eta - \eta'| \geq N$ whenever $|\eta| \leq N$ and $|\eta'| \geq 2N$. Using H\"{o}lder's inequality, the remaining part above can be further bounded as
	\begin{align*}
		&\int_0^t G_v(t,s) \frac{1}{(1+|v|)^{-\frac{\gamma}{p}+\frac{5p-5}{4p}}} 
		  \left(\int_{|\eta| \leq N} (1+|\eta|)^{\frac{4p-8}{p}-2\beta}\right. \\
		  &\left. \qquad \times  \left( \int_0^{s-\delta} G_\eta (s,\tau) \int_{|\eta'|\leq 2N} |k_w^{ns}(\eta,\eta')||h(\tau,X(s)-\eta(s-\tau),\eta')| d\eta' d\tau \right)^2 d\eta \right)^{1/2}ds\\
		&\leq \frac{C_{N,\varepsilon}}{(1+|v|)^{\frac{p-1}{4p}}} \left(\int_0^t e^{-\frac{\nu(v)}{2}(t-s)}\nu(v) 
		\int_0^{s-\delta} e^{-CN^{\gamma}(s-\tau)} \right. \\
		&\left. \qquad \times \int_{|\eta| \leq N}  \int_{|\eta'| \leq 2N} |h(\tau,X(s)-\eta(s-\tau),\eta')|^2 d\eta' d\eta d\tau   ds\right)^{1/2} \\
		&\leq \frac{C_{\gamma,N,\varepsilon}}{(1+|v|)^{\frac{p-1}{4p}}}\sup_{0\leq \tau \leq s-\delta} \left(\int_{|\eta| \leq N} \int_{|\eta'|\leq 2N} |h(\tau,X(s)-\eta(s-\tau),\eta')|^2d\eta' d\eta \right)^{1/2}\\
		&\leq  \frac{C_{\gamma,N,\varepsilon}}{(1+|v|)^{\frac{p-1}{4p}}} \left[ \mathcal{E}(F_0)^{1/2} + \sup_{0\le s \le t}\|h(s)\|_{L^p_vL^\infty_x}^{\frac{p}{2p-2}}\mathcal{E}(F_0)^{\frac{p-2}{2p-2}}\right],
	\end{align*}
	where we have used a similar argument in \eqref{RTIRE}, and it follows that 
	\begin{align} \label{III,2} 
		III_2 &\leq  \frac{C_{\gamma,N,\varepsilon}}{(1+|v|)^{\frac{p-1}{4p}}} \left[ \mathcal{E}(F_0)^{1/2} + \sup_{0\le s \le t}\|h(s)\|_{L^p_vL^\infty_x}^{\frac{p}{2p-2}}\mathcal{E}(F_0)^{\frac{p-2}{2p-2}}\right]\nonumber\\
		& \quad + \frac{C_{p,\gamma}}{(1+|v|)^{\frac{p-1}{4p}}}\left(\delta+ \frac{1}{N^{-\gamma}}+ \frac{1}{N^{\frac{3+\gamma}{2}}}+\varepsilon^{\gamma+\frac{3}{p'}}\right)\sup_{0\leq s \leq t} \Vert h(s)\Vert_{L^p_vL^\infty_x} .
	\end{align}
	Similar to how we dealt with the term $III_2$, we treat the term $III_3$. Using \eqref{pointwiseGamma+estimate1} in Lemma \ref{pointwiseGamma+estimate}, it holds that 
	\begin{align}\label{III,3} 
		III_3 &\leq C_{p,\gamma,\beta}\sup_{0\leq s \leq t} \Vert h(s) \Vert_{L^p_vL^\infty_x} \int_0^t G_v(t,s) \frac{1}{(1+|v|)^{-\frac{\gamma}{p}+\frac{5p-5}{4p}}} \left(\int_{|\eta|\leq N} (1+|\eta|)^{\frac{4p-8}{p}-2\beta} \right. \nonumber\\
		&\left. \qquad \times  \left(\int_0^s G_{\eta}(s,\tau) \frac{1}{(1+|\eta|)^{-\frac{\gamma}{p}+\frac{5p-5}{4p}} } \right. \right. \nonumber \\
		&\left. \left. \quad \qquad \times \left(\int_{\R^3} (1+|\eta'|)^{\frac{4p-8}{p}-2\beta} |h(\tau,X(s)-\eta(s-\tau),\eta')|^2 d\eta' \right)^{1/2}d\tau \right)^2 d\eta \right)^{1/2}  ds \nonumber \\
		&\le C_{p,\gamma,\beta}\sup_{0\leq s \leq t} \Vert h(s) \Vert_{L^p_vL^\infty_x} \int_0^t G_v(t,s) \frac{1}{(1+|v|)^{-\frac{\gamma}{p}+\frac{5p-5}{4p}}} \left(\int_{|\eta|\leq N} (1+|\eta|)^{\frac{4p-8}{p}-2\beta} \right.  \nonumber \\
		&\left.  \qquad \times  \left(\int_0^s G_{\eta}(s,\tau) \frac{1}{(1+|\eta|)^{-\frac{\gamma}{p}+\frac{5p-5}{4p}} }\right. \right. \nonumber \\
		&\left. \left. \quad \qquad \times \left(\int_{|\eta'|\leq 2N} (1+|\eta'|)^{\frac{4p-8}{p}-2\beta} |h(\tau,X(s)-\eta(s-\tau),\eta')|^2 d\eta' \right)^{1/2}d\tau \right)^2 d\eta \right)^{1/2}  ds \nonumber \\
		& \quad +C_{p,\gamma,\beta}\sup_{0\leq s \leq t} \Vert h(s) \Vert_{L^p_vL^\infty_x} \int_0^t G_v(t,s) \frac{1}{(1+|v|)^{-\frac{\gamma}{p}+\frac{5p-5}{4p}}} \left(\int_{|\eta|\leq N} (1+|\eta|)^{\frac{4p-8}{p}-2\beta} \right. \nonumber \\
		&\left. \qquad \times  \left(\int_0^s G_{\eta}(s,\tau) \frac{1}{(1+|\eta|)^{-\frac{\gamma}{p}+\frac{5p-5}{4p}} } \right. \right. \nonumber \\ 
		&\left. \left. \quad \qquad \times \left(\int_{|\eta'|\geq 2N} (1+|\eta'|)^{\frac{4p-8}{p}-2\beta} |h(\tau,X(s)-\eta(s-\tau),\eta')|^2 d\eta' \right)^{1/2}d\tau \right)^2 d\eta \right)^{1/2}  ds \nonumber \\
		&\leq C_{p,\gamma,\beta}\sup_{0\leq s \leq t} \Vert h(s) \Vert_{L^p_vL^\infty_x} \int_0^t G_v(t,s) \frac{1}{(1+|v|)^{-\frac{\gamma}{p}+\frac{5p-5}{4p}}} \left(\int_{|\eta|\leq N} (1+|\eta|)^{\frac{4p-8}{p}-2\beta} \right.\nonumber \\
		&\left. \qquad \times  \left(\int_0^{s-\delta} G_{\eta}(s,\tau) \frac{1}{(1+|\eta|)^{-\frac{\gamma}{p}+\frac{5p-5}{4p}} }\right. \right. \nonumber \\
		&\left. \left. \qquad \quad \times \left(\int_{|\eta'|\leq 2N} (1+|\eta'|)^{\frac{4p-8}{p}-2\beta} |h(\tau,X(s)-\eta(s-\tau),\eta')|^2 d\eta' \right)^{1/2}d\tau \right)^2 d\eta \right)^{1/2}  ds\nonumber \\
		&\quad + C_{p,\gamma,\beta}\sup_{0\leq s \leq t} \Vert h(s) \Vert_{L^p_vL^\infty_x} \int_0^t G_v(t,s) \frac{1}{(1+|v|)^{-\frac{\gamma}{p}+\frac{5p-5}{4p}}} \left(\int_{|\eta|\leq N} (1+|\eta|)^{\frac{4p-8}{p}-2\beta} \right. \nonumber\\
		&\left. \qquad \times  \left(\int_{s-\delta}^{s} G_{\eta}(s,\tau) \frac{1}{(1+|\eta|)^{-\frac{\gamma}{p}+\frac{5p-5}{4p}} }\right. \right.  \nonumber \\
		&\left. \left. \qquad \quad \times \left(\int_{|\eta'|\leq 2N} (1+|\eta'|)^{\frac{4p-8}{p}-2\beta} |h(\tau,X(s)-\eta(s-\tau),\eta')|^2 d\eta' \right)^{1/2}d\tau \right)^2 d\eta \right)^{1/2}  ds \nonumber\\
		&\quad + \frac{C_{p,\gamma,\beta}}{N} \frac{\sup_{0\leq s \leq t} \Vert h(s) \Vert_{L^p_vL^\infty_x}^2 }{(1+|v|)^{\frac{p-1}{4p}}} \nonumber \\
		&\leq \frac{C_{p,\gamma,\beta}}{(1+|v|)^{\frac{p-1}{4p}}} \sup_{0\leq s \leq t} \Vert h(s) \Vert_{L^p_vL^\infty_x} \left[ \mathcal{E}(F_0)^{1/2} + \sup_{0\le s \le t}\|h(s)\|_{L^p_vL^\infty_x}^{\frac{p}{2p-2}}\mathcal{E}(F_0)^{\frac{p-2}{2p-2}}\right] \nonumber\\
		& \quad + \frac{C_{p,\gamma,\beta}}{(1+|v|)^{\frac{p-1}{4p}}} \delta \sup_{0\leq s \leq t} \Vert h(s) \Vert_{L^p_vL^\infty_s}^ 2 +  \frac{C_{p,\gamma,\beta}}{N} \frac{\sup_{0\leq s \leq t} \Vert h(s) \Vert_{L^p_vL^\infty_x}^2 }{(1+|v|)^{\frac{p-1}{4p}}}.
	\end{align}
	Combining \eqref{III}, \eqref{III,1}, \eqref{III,2}, and \eqref{III,3}, we have 
	\begin{align} \label{III,final} 
		III &\leq\frac{C}{(1+|v|)^{-\frac{\gamma}{p}+\frac{5p-5}{4p}}}e^{-\lambda (1+t)^\rho} \Vert h_0 \Vert_{L^p_v L^\infty_x} \int_0^t \Vert h(s) \Vert_{L^p_vL^\infty_x}ds \nonumber\\
		&\quad + \frac{C_{p,\gamma,\beta}}{(1+|v|)^{\frac{p-1}{4p}}} \delta \left[\sup_{0\leq s \leq t} \Vert h(s) \Vert_{L^p_v L^\infty_x}^2 +\sup_{0\leq s \leq t}\Vert h(s) \Vert_{L^p_vL^\infty_x}^3 \right] + \frac{C_{p,\gamma,\beta}}{N}  \frac{\sup_{0\leq s \leq t} \Vert h(s) \Vert_{L^p_vL^\infty_x}^3 }{(1+|v|)^{\frac{p-1}{4p}}}  \nonumber \\
		&\quad + \frac{C_{p,\gamma}}{(1+|v|)^{\frac{p-1}{4p}}}\left(\frac{1}{N} + \frac{1}{N^{-\gamma}}+ \frac{1}{N^{\frac{3+\gamma}{2}}}+\varepsilon^{\gamma+\frac{3}{p'}}\right)\sup_{0\leq s \leq t} \Vert h(s)\Vert_{L^p_vL^\infty_x}^2 \nonumber \\
		&\quad +\frac{C_{p,\gamma,\beta,N, \varepsilon}}{(1+|v|)^{\frac{p-1}{4p}}}\left(\sup_{0\leq s \leq t} \Vert h(s) \Vert_{L^p_vL^\infty_x}+ \sup_{0\leq s \leq t} \Vert h(s) \Vert_{L^p_vL^\infty_x}^2 \right)\nonumber\\
		& \qquad \times \left[ \mathcal{E}(F_0)^{1/2} + \sup_{0\le s \le t}\|h(s)\|_{L^p_vL^\infty_x}^{\frac{p}{2p-2}}\mathcal{E}(F_0)^{\frac{p-2}{2p-2}}\right].
		\end{align}
		\newline
	Next, we estimate the term $III$ for the case $-3<\gamma<-1$. \\
	
	\textbf{(Case 2. $-3< \gamma <-1$)} It follows from \eqref{pointwiseGamma+estimate2} in Lemma \ref{pointwiseGamma+estimate} that 
	\begin{align*}
		III &\leq C_{p,\gamma,\beta} \int_0^t G_v(t,s) \frac{ \|h(s)\|_{L^p_vL^\infty_x}}{(1+|v|)^{-\gamma+\frac{p-1}{p}\varpi}}\left(\int_{\R^3}(1+|\eta|)^{\frac{4}{m-1}-p'm'\beta}|h(s,X(s),\eta)|^{p'm'} d\eta \right)^{\frac{1}{p'm'}}ds\\
		&\leq  C_{p,\gamma,\beta} \int_0^t G_v(t,s) \frac{ \|h(s)\|_{L^p_vL^\infty_x}}{(1+|v|)^{-\gamma+\frac{p-1}{p}\varpi}}\\
		&\qquad \times \left(\int_{|\eta|\geq N}(1+|\eta|)^{\frac{4}{m-1}-p'm'\beta}|h(s,X(s),\eta)|^{p'm'} d\eta \right)^{\frac{1}{p'm'}}ds\\
		&\quad + C_{p,\gamma,\beta} \int_0^t G_v(t,s) \frac{ \|h(s)\|_{L^p_vL^\infty_x}}{(1+|v|)^{-\gamma+\frac{p-1}{p}\varpi}}\\
		&\qquad \times \left(\int_{|\eta|\leq N}(1+|\eta|)^{\frac{4}{m-1}-p'm'\beta}|h(s,X(s),\eta)|^{p'm'} d\eta \right)^{\frac{1}{p'm'}}ds\\
		&\leq \frac{C_{p,\gamma,\beta}}{(1+|v|)^{\frac{p-1}{p}\varpi}} \frac{\sup_{0\leq s \leq t}\Vert h(s) \Vert_{L^p_v L^\infty_x}^2}{N} \\
		&\quad + C_{p,\gamma,\beta} \int_0^t G_v(t,s) \frac{ \|h(s)\|_{L^p_vL^\infty_x}}{(1+|v|)^{-\gamma+\frac{p-1}{p}\varpi}}\\
		&\qquad \times \left(\int_{|\eta|\leq N}(1+|\eta|)^{\frac{4}{m-1}-p'm'\beta}|h(s,X(s),\eta))|^{p'm'} d\eta \right)^{\frac{1}{p'm'}}ds.
	\end{align*}
	We apply Duhamel's principle to the integral above. 
	\begin{align*}
		&\int_0^t G_v(t,s) \frac{ \|h(s)\|_{L^p_vL^\infty_x}}{(1+|v|)^{-\gamma+\frac{p-1}{p}\varpi}}\left(\int_{|\eta|\leq N}(1+|\eta|)^{\frac{4}{m-1}-p'm'\beta}|h(s,X(s),\eta)|^{p'm'} d\eta \right)^{\frac{1}{p'm'}}ds\\
		&\leq \int_0^t G_v(t,s) \frac{ \|h(s)\|_{L^p_vL^\infty_x}}{(1+|v|)^{-\gamma+\frac{p-1}{p}\varpi}}\\
		&\qquad \times \left(\int_{|\eta|\leq N}(1+|\eta|)^{\frac{4}{m-1}-p'm'\beta}|G_{\eta}(s,0)h_0(X(0),\eta)|^{p'm'} d\eta \right)^{\frac{1}{p'm'}}ds\\
		&\quad +\sup_{0\leq s \leq t} \Vert h(s) \Vert_{L^p_v L^\infty_x} \int_0^t G_v(t,s) \frac{1}{(1+|v|)^{-\gamma+\frac{p-1}{p}\varpi}} \\
		&\qquad \times \left(\int_{|\eta| \leq N} (1+|\eta|)^{\frac{4}{m-1}-p'm'\beta} \left( \int_0^s G_\eta (s,\tau) |K_wh|(\tau) d\tau \right)^{p'm'} d\eta \right)^{\frac{1}{p'm'}}ds\\
		&\quad + \sup_{0\leq s \leq t}\Vert h(s) \Vert_{L^p_vL^\infty_x} \int_0^t G_v(t,s) \frac{1}{(1+|v|)^{-\gamma +\frac{p-1}{p}\varpi}} \\
		&\qquad \times \left(\int_{|\eta| \leq N} (1+|\eta|)^{\frac{4}{m-1}-p'm'\beta} \left( \int_0^s G_\eta (s,\tau) |w\Gamma^+(f,f)|(\tau) d\tau \right)^{p'm'} d\eta \right)^{\frac{1}{p'm'}}ds \\
		&:= III_1 + \sup_{0\leq s \leq t}\Vert h(s) \Vert_{L^p_vL^\infty_x}\times [III_2 + III_3]. 
	\end{align*}
	Similar to Case 1, we can further bound the terms $III_1$, $III_2$, and $III_3$
	\begin{align*} 
		III_1 &\leq \frac{C}{(1+|v|)^{-\gamma + \frac{p-1}{p}\varpi}}  e^{-\lambda(1+t)^\rho} \Vert h_0 \Vert_{L^p_v L^\infty_x} \int_0^t \Vert h(s) \Vert_{L^p_v L^\infty_x} ds, \nonumber \\
		III_2 &\leq \frac{C_{N,\varepsilon}}{(1+|v|)^{\frac{p-1}{p}\varpi}} \left[ \mathcal{E}(F_0)^{1/2} + \sup_{0\le s \le t}\|h(s)\|_{L^p_vL^\infty_x}^{\frac{p}{2p-2}}\mathcal{E}(F_0)^{\frac{p-2}{2p-2}}\right]\nonumber\\
		& \quad  + C_{p,\gamma} \left(\delta + \frac{1}{N^{-\gamma}}+\frac{1}{N^{\frac{3+\gamma}{2}}}+ \varepsilon^{\gamma+\frac{3}{p'}} \right) \frac{\sup_{0\leq s \leq t} \Vert h(s) \Vert_{L^p_v L^\infty_x}}{(1+|v|)^{\frac{p-1}{p}\varpi}}, \nonumber \\ 
		III_3 &\leq \frac{C_{p,\gamma,\beta}}{(1+|v|)^{\frac{p-1}{p}\varpi}}  \sup_{0 \leq s \leq t}\|h(s)\|_{L^p_vL^\infty_x}\Bigg[\sup_{0 \leq s \leq t}\|h(s)\|_{L^p_vL^\infty_x}^{1-\frac{1/p'm'-1/p}{1/2-1/p}}\mathcal{E}(F_0)^{\frac{1}{2}\cdot\frac{1/p'm'-1/p}{1/2-1/p}}\nonumber\\
		& \quad  + \sup_{0 \leq s \leq t}\|h(s)\|_{L^p_vL^\infty_x}^{1-\frac{1/p'm'-1/p}{1-1/p}}\mathcal{E}(F_0)^{\frac{1/p'm'-1/p}{1-1/p}}\Bigg]+ \frac{C_{p,\gamma,\beta}}{(1+|v|)^{\frac{p-1}{p}\varpi}} \delta \sup_{0\leq s \leq t} \Vert h(s) \Vert_{L^p_v L^\infty_x}^2 \nonumber \\
		&\quad + \frac{C_{p,\gamma,\beta}}{N} \frac{\sup_{0\leq s \leq t} \Vert h(s) \Vert_{L^p_v L^\infty_x}^2}{(1+|v|)^{\frac{p-1}{p}\varpi}},
	\end{align*}
	where we have applied the inequality \eqref{RTIRE1212} to estimate $III_3$,  which implies that 
	\begin{align} \label{2,III,final}
		III &\leq \frac{C}{(1+|v|)^{-\gamma+ \frac{p-1}{p}\varpi}}  e^{-\lambda(1+t)^\rho} \Vert h_0 \Vert_{L^p_v L^\infty_x} \int_0^t \Vert h(s) \Vert_{L^p_v L^\infty_x} ds \nonumber \\
		&\quad + \frac{C_{p,\gamma,\beta}}{(1+|v|)^{\frac{p-1}{p}\varpi}} \delta \left[\sup_{0\leq s \leq t} \Vert h(s) \Vert_{L^p_v L^\infty_x}^2 +\sup_{0\leq s \leq t}\Vert h(s) \Vert_{L^p_vL^\infty_x}^3 \right] \nonumber \\
		&\quad + \frac{C_{p,\gamma,\beta}}{N}  \frac{\sup_{0\leq s \leq t} \Vert h(s) \Vert_{L^p_vL^\infty_x}^3 }{(1+|v|)^{\frac{p-1}{p}\varpi}} \nonumber \\
		& \quad  + \frac{C_{p,\gamma}}{(1+|v|)^{\frac{p-1}{p}\varpi}}\left(\frac{1}{N} + \frac{1}{N^{-\gamma}}+ \frac{1}{N^{\frac{3+\gamma}{2}}}+\varepsilon^{\gamma+\frac{3}{p'}}\right)\sup_{0\leq s \leq t} \Vert h(s)\Vert_{L^p_vL^\infty_x}^2  \nonumber \\
		&\quad +\frac{C_{N, \varepsilon}}{(1+|v|)^{\frac{p-1}{p}\varpi}} \sup_{0\leq s \leq t} \Vert h(s) \Vert_{L^p_vL^\infty_x} \left[ \mathcal{E}(F_0)^{1/2} + \sup_{0\le s \le t}\|h(s)\|_{L^p_vL^\infty_x}^{\frac{p}{2p-2}}\mathcal{E}(F_0)^{\frac{p-2}{2p-2}}\right]\nonumber\\
		& \quad + \frac{C_{p,\gamma,\beta}}{(1+|v|)^{\frac{p-1}{p}\varpi}} \sup_{0 \leq s \leq t}\|h(s)\|_{L^p_vL^\infty_x}^2\Bigg[\sup_{0 \leq s \leq t}\|h(s)\|_{L^p_vL^\infty_x}^{1-\frac{1/p'm'-1/p}{1/2-1/p}}\mathcal{E}(F_0)^{\frac{1}{2}\cdot\frac{1/p'm'-1/p}{1/2-1/p}}\nonumber\\
		& \qquad  + \sup_{0 \leq s \leq t}\|h(s)\|_{L^p_vL^\infty_x}^{1-\frac{1/p'm'-1/p}{1-1/p}}\mathcal{E}(F_0)^{\frac{1/p'm'-1/p}{1-1/p}}\Bigg].
	\end{align}
	Finally, for $-1\leq \gamma<0$, we derive the following estimate from \eqref{I}, \eqref{II,final}, and \eqref{III,final}
	\begin{align*}
		&|h(t,x,v)| \\
		&\leq C_{\lambda}e^{-\lambda(1+t)^{\rho}} h_0(x-tv,v)  + \frac{C}{(1+|v|)^{-\frac{\gamma}{p}+\frac{5p-5}{4p}}}e^{-\lambda (1+t)^\rho} \Vert h_0 \Vert_{L^p_v L^\infty_x} \int_0^t \Vert h(s) \Vert_{L^p_vL^\infty_x}ds \\
		&\quad +C_{p,\rho,\lambda,\gamma}e^{-\frac{\lambda}{2}(1+t)^{\rho}} \langle v \rangle^{\gamma-1-\frac{1}{p'}} \Vert h_0 \Vert_{L^p_v L^\infty_x}\\
		&\quad +\frac{C_{p,q,\gamma}}{(1+|v|)^{\frac{p-1}{4p}}} \varepsilon^{\gamma+ \frac{3}{p'}}   \left[\sup_{0 \leq s \leq t} \Vert h(s) \Vert_{L^p_v L^\infty_x}+\sup_{0 \leq s \leq t} \Vert h(s) \Vert_{L^p_v L^\infty_x}^2\right]\\
		&\quad + \frac{C_{p,q,\varepsilon,\gamma,\beta}}{(1+|v|)^{\frac{p-1}{4p}}} \left(\frac{1}{N}+\frac{1}{N^{-\gamma}}+\frac{1}{N^{\frac{\gamma+3}{2}}}+\delta  \right) \\
		& \qquad \times \left[\sup_{0\leq s \leq t} \Vert h(s) \Vert_{L^p_vL^\infty_x}+\sup_{0\leq s \leq t} \Vert h(s) \Vert_{L^p_vL^\infty_x}^2+\sup_{0\leq s \leq t} \Vert h(s) \Vert_{L^p_vL^\infty_x}^3\right]\\
		&\quad +\frac{C_{p,q,\varepsilon,\gamma,N,\delta}}{(1+|v|)^{\frac{p-1}{4p}}}\left(1+\sup_{0\leq s \leq t} \Vert h(s) \Vert_{L^p_vL^\infty_x}+\sup_{0\leq s \leq t} \Vert h(s) \Vert_{L^p_vL^\infty_x}^2\right)\\
		& \qquad \times  \left[ \mathcal{E}(F_0)^{1/2} + \sup_{0\le s \le t}\|h(s)\|_{L^p_vL^\infty_x}^{\frac{p}{2p-2}}\mathcal{E}(F_0)^{\frac{p-2}{2p-2}}\right].
	\end{align*}
	For $-3<\gamma<-1$, it follows from \eqref{I}, \eqref{2,II,final}, and \eqref{2,III,final} that 
	\begin{align*}
		&|h(t,x,v)|\\
		 &\leq C_{\lambda}e^{-\lambda(1+t)^{\rho}} h_0(x-tv,v)  +\frac{C}{(1+|v|)^{-\gamma+ \frac{p-1}{p}\varpi}}  e^{-\lambda(1+t)^\rho} \Vert h_0 \Vert_{L^p_v L^\infty_x} \int_0^t \Vert h(s) \Vert_{L^p_v L^\infty_x} ds \nonumber \\
		&\quad +C_{p,\rho,\lambda,\gamma}e^{-\frac{\lambda}{2}(1+t)^{\rho}} \langle v \rangle^{\gamma-1-\frac{1}{p'}} \Vert h_0 \Vert_{L^p_v L^\infty_x}\\
		&\quad +\frac{C_{p,q,\gamma}}{(1+|v|)^{\frac{p-1}{p}\varpi}} \varepsilon^{\gamma+ \frac{3}{p'}}   \left[\sup_{0 \leq s \leq t} \Vert h(s) \Vert_{L^p_v L^\infty_x}+\sup_{0 \leq s \leq t} \Vert h(s) \Vert_{L^p_v L^\infty_x}^2\right]\\
		&\quad + \frac{C_{p,q,\varepsilon,\gamma,\beta}}{(1+|v|)^{\frac{p-1}{p}\varpi}} \left(\frac{1}{N}+\frac{1}{N^{-\gamma}}+\frac{1}{N^{\frac{\gamma+3}{2}}}+\delta  \right) \\
		& \qquad \times \left[\sup_{0\leq s \leq t} \Vert h(s) \Vert_{L^p_vL^\infty_x}+\sup_{0\leq s \leq t} \Vert h(s) \Vert_{L^p_vL^\infty_x}^2+\sup_{0\leq s \leq t} \Vert h(s) \Vert_{L^p_vL^\infty_x}^3\right]\\
		&\quad +\frac{C_{p,q,\varepsilon,\gamma,N,\delta}}{(1+|v|)^{\frac{p-1}{p}\varpi}}\left(1+\sup_{0\leq s \leq t} \Vert h(s) \Vert_{L^p_vL^\infty_x}\right)\left[ \mathcal{E}(F_0)^{1/2} + \sup_{0\le s \le t}\|h(s)\|_{L^p_vL^\infty_x}^{\frac{p}{2p-2}}\mathcal{E}(F_0)^{\frac{p-2}{2p-2}}\right]\\
		& \quad +\frac{C_{p,q,\varepsilon,\gamma,N,\delta}}{(1+|v|)^{\frac{p-1}{p}\varpi}}\left(\sup_{0\leq s \leq t} \Vert h(s) \Vert_{L^p_vL^\infty_x}+\sup_{0\leq s \leq t} \Vert h(s) \Vert_{L^p_vL^\infty_x}^2\right)\\
		& \qquad \times  \Bigg[\sup_{0 \leq s \leq t}\|h(s)\|_{L^p_vL^\infty_x}^{2-\frac{1/p'm'-1/p}{1/2-1/p}}\mathcal{E}(F_0)^{\frac{1}{2}\cdot\frac{1/p'm'-1/p}{1/2-1/p}}+ \sup_{0 \leq s \leq t}\|h(s)\|_{L^p_vL^\infty_x}^{2-\frac{1/p'm'-1/p}{1-1/p}}\mathcal{E}(F_0)^{\frac{1/p'm'-1/p}{1-1/p}}\Bigg].
	\end{align*}
	
\end{proof}
\bigskip

\begin{Coro} \label{smalllp}
Let $h(t,x,v)$ satisfy the equation \eqref{Rf nonlinear equation} and  $\rho-1 = \frac{(1+\vartheta)\gamma}{2-\gamma}$. Under the a priori assumption \eqref{apriorismall} and $\mathcal{E}(F_0) \leq \varepsilon_2$, we get the following estimate for each case:
\begin{enumerate}
	\item If $-1 \leq \gamma <0$, then there exists a constant $C_{3,1}$, depending on $\lambda,p,q$, and $\gamma$, so that
	\begin{align*} 
		\Vert h(t) \Vert_{L^p_v L^\infty_x} 
		&\leq C_{3,1} \Vert h_0 \Vert_{L^p_vL^\infty_x} \left(1+\int_0^t \Vert h(s) \Vert_{L^p_v L^\infty_x}ds \right) e^{-\frac{\lambda}{2} (1+t)^\rho} \nonumber\\
		&\quad +C_{p,q,\gamma} \varepsilon^{\gamma+ \frac{3}{p'}}   \left[\sup_{0 \leq s \leq t} \Vert h(s) \Vert_{L^p_v L^\infty_x}+\sup_{0 \leq s \leq t} \Vert h(s) \Vert_{L^p_v L^\infty_x}^2\right]\nonumber\\
		&\quad + C_{p,q,\varepsilon,\gamma,\beta}\left(\frac{1}{N}+\frac{1}{N^{-\gamma}}+\frac{1}{N^{\frac{\gamma+3}{2}}}+\delta  \right) \nonumber\\
		& \qquad \times \left[\sup_{0\leq s \leq t} \Vert h(s) \Vert_{L^p_vL^\infty_x}+\sup_{0\leq s \leq t} \Vert h(s) \Vert_{L^p_vL^\infty_x}^2+\sup_{0\leq s \leq t} \Vert h(s) \Vert_{L^p_vL^\infty_x}^3\right]\nonumber\\
		&\quad +C_{p,q,\varepsilon,\gamma,N,\delta}\left(1+\sup_{0\leq s \leq t} \Vert h(s) \Vert_{L^p_vL^\infty_x}+\sup_{0\leq s \leq t} \Vert h(s) \Vert_{L^p_vL^\infty_x}^2\right)\nonumber\\
		& \qquad \times  \left[ \mathcal{E}(F_0)^{1/2} + \sup_{0\le s \le t}\|h(s)\|_{L^p_vL^\infty_x}^{\frac{p}{2p-2}}\mathcal{E}(F_0)^{\frac{p-2}{2p-2}}\right],
	\end{align*}
	for all $0\leq t \leq T$, where $\varepsilon>0$ and $\delta>0$ can be arbitrarily small and $N>0$ can be arbitrarily large.
	\item If $-3 < \gamma<-1$, then then there exists a constant $C_{3,2}$, depending on $\lambda,p,q$, and $\gamma$, so that
	\begin{align*}
			\Vert h(t) \Vert_{L^p_v L^\infty_x} 
			&\leq C_{3,2}  \Vert h_0 \Vert_{L^p_v L^\infty_x} \left(1+\int_0^t \Vert h(s) \Vert_{L^p_v L^\infty_x} ds\right)e^{-\frac{\lambda}{2}(1+t)^\rho} \\ 
		&\quad +C_{p,q,\gamma}\varepsilon^{\gamma+ \frac{3}{p'}}   \left[\sup_{0 \leq s \leq t} \Vert h(s) \Vert_{L^p_v L^\infty_x}+\sup_{0 \leq s \leq t} \Vert h(s) \Vert_{L^p_v L^\infty_x}^2\right]\\
		&\quad + C_{p,q,\varepsilon,\gamma,\beta} \left(\frac{1}{N}+\frac{1}{N^{-\gamma}}+\frac{1}{N^{\frac{\gamma+3}{2}}}+\delta  \right) \\
		& \qquad \times \left[\sup_{0\leq s \leq t} \Vert h(s) \Vert_{L^p_vL^\infty_x}+\sup_{0\leq s \leq t} \Vert h(s) \Vert_{L^p_vL^\infty_x}^2+\sup_{0\leq s \leq t} \Vert h(s) \Vert_{L^p_vL^\infty_x}^3\right]\\
		&\quad +C_{p,q,\varepsilon,\gamma,N,\delta}\left(1+\sup_{0\leq s \leq t} \Vert h(s) \Vert_{L^p_vL^\infty_x}\right)\left[ \mathcal{E}(F_0)^{1/2} + \sup_{0\le s \le t}\|h(s)\|_{L^p_vL^\infty_x}^{\frac{p}{2p-2}}\mathcal{E}(F_0)^{\frac{p-2}{2p-2}}\right]\\
		& \quad +C_{p,q,\varepsilon,\gamma,N,\delta}\left(\sup_{0\leq s \leq t} \Vert h(s) \Vert_{L^p_vL^\infty_x}+\sup_{0\leq s \leq t} \Vert h(s) \Vert_{L^p_vL^\infty_x}^2\right)\\
		& \qquad \times  \Bigg[\sup_{0 \leq s \leq t}\|h(s)\|_{L^p_vL^\infty_x}^{1-\frac{1/p'm'-1/p}{1/2-1/p}}\mathcal{E}(F_0)^{\frac{1}{2}\cdot\frac{1/p'm'-1/p}{1/2-1/p}}\\
		&\qquad \quad + \sup_{0 \leq s \leq t}\|h(s)\|_{L^p_vL^\infty_x}^{1-\frac{1/p'm'-1/p}{1-1/p}}\mathcal{E}(F_0)^{\frac{1/p'm'-1/p}{1-1/p}}\Bigg],
	\end{align*}
	for all $0\leq t \leq T$, where $\varepsilon>0$ and $\delta>0$ can be arbitrarily small and $N>0$ can be arbitrarily large.
\end{enumerate}
\end{Coro}
\begin{proof}
	Within Lemma \ref{small,Lp}, for the pointwise estimate of $h$ to be in $L^p_v$, it is necessary that the following condition holds for each case: 
	\newline 
		\noindent \textbf{Case 1.}  ($-1 \leq \gamma <0$)
		\begin{align} \label{pcondition1}
				-\gamma +\frac{5p-5}{4} >3, \; p\left(-\gamma +1 + \frac{p}{p'}\right)>3, \quad \textrm{and} \quad \frac{p-1}{4} >3 
		\end{align}
		for $p >13$. \\
		
		\noindent \textbf{Case 2.} ($-3<\gamma<-1$)
		\begin{align} \label{pcondition2}
			-p\gamma +(p-1) \varpi>3, \; p\left(-\gamma +1 +\frac{1}{p'}\right)>3, \quad \textrm{and} \quad (p-1)\varpi >3 
		\end{align}
		for $p > \frac{3(\sqrt{8\gamma^2 +9} +3-2\gamma)}{(-\gamma)(3+\gamma)}$. 
\end{proof}

\subsection{Global Existence and Sub-exponential decay} 
\label{smallglobalproof}
\begin{proof} [Proof of Theorem \ref{smallmain}]
Set $C_4:=\max\{C_{3,1},C_{3,2}\}$.
Note that $(1+t)^{1-\rho} e^{-\frac{\lambda}{4}(1+t)^\rho} \le \tilde C_{\lambda,\rho}$ and $(1+t)e^{-\frac{\lambda}{4}(1+t)^\rho} \le \hat C_{\lambda,\rho}$ for some constants $\tilde C_{\lambda,\rho}$ and $\hat C_{\lambda,\rho}\ge 1$, depending on $\lambda$ and $\rho$.
Assume that $T>0$ is an arbitrary constant. Let us fix $\eta_0>0$ to satisfy the conditions \eqref{cond1apriori}, \eqref{cond2apriori}, \eqref{cond3apriori}, \eqref{cond4apriori}, and
\begin{align} \label{cond5apriori}
	\exp\left\{\frac{4C_4\tilde C_{\lambda,\rho}\eta_0}{\lambda \rho} \right\} <2.
\end{align}
Take
		\begin{align*} 
			\bar{\eta} :=8 \hat C_{\lambda,\rho}C_4 \eta_0.
		\end{align*}
		Using the a priori assumption \eqref{apriorismall} and Corollary \ref{smalllp}, we have
		\begin{align} \label{LA66}
			\|h(t)\|_{_{L^p_vL^\infty_x}} \le C_4 \eta_0 \left(1+\int_0^t \|h(s)\|_{L^p_vL^\infty_x}ds\right)e^{-\frac{\lambda}{2}(1+t)^\rho}+E \quad \text{for all } t\ge 0,
		\end{align}
		where
		\begin{align*}
			E&=C_{p,q,\gamma} \varepsilon^{\gamma+ \frac{3}{p'}}   \left[\bar{\eta}+\bar{\eta}^2\right]+ C_{p,q,\varepsilon,\gamma,\beta}\left(\frac{1}{N}+\frac{1}{N^{-\gamma}}+\frac{1}{N^{\frac{\gamma+3}{2}}}+\delta  \right) \left[\bar{\eta}+\bar{\eta}^2+\bar{\eta}^3\right]\nonumber\\
		&\quad +C_{p,q,\varepsilon,\gamma,N,\delta}\left(1+\bar{\eta}+\bar{\eta}^2\right)\left[ \mathcal{E}(F_0)^{1/2} + \bar{\eta}^{\frac{p}{2p-2}}\mathcal{E}(F_0)^{\frac{p-2}{2p-2}}\right]\\
		& \quad +C_{p,q,\varepsilon,\gamma,N,\delta}\left(\bar{\eta}+\bar{\eta}^2\right)\Bigg[\bar{\eta}^{1-\frac{1/p'm'-1/p}{1/2-1/p}}\mathcal{E}(F_0)^{\frac{1}{2}\cdot\frac{1/p'm'-1/p}{1/2-1/p}}+ \bar{\eta}^{1-\frac{1/p'm'-1/p}{1-1/p}}\mathcal{E}(F_0)^{\frac{1/p'm'-1/p}{1-1/p}}\Bigg].
		\end{align*}
		We define
		\begin{align*}
			G(t) := 1+\int_0^t \|h(s)\|_{L^p_vL^\infty_x}ds.
		\end{align*}
		Then it holds that for all $0<t \le T$,
		\begin{align} \label{LA67}
			G'(t) -C_4 \eta_0 e^{-\frac{\lambda}{2}(1+t)^\rho} G(t)
		 &\ge G'(t) -C_4\tilde C_{\lambda,\rho}\eta_0 (1+t)^{\rho-1}e^{-\frac{\lambda}{4}(1+t)^\rho} G(t)\nonumber\\
			& = \frac{d}{dt}\left\{ G(t)\exp\left\{ -\frac{4C_4\tilde C_{\lambda,\rho}\eta_0}{\lambda \rho}\left(1-e^{-\frac{\lambda}{4}(1+t)^\rho}\right)\right\}\right\}.
		\end{align}
		From \eqref{LA66} and \eqref{LA67}, we obtain
		\begin{align*}
			G(t)\exp\left\{ -\frac{4C_4\tilde C_{\lambda,\rho}\eta_0}{\lambda \rho}\left(1-e^{-\frac{\lambda}{4}(1+t)^\rho}\right)\right\} &\le G(0)\exp\left\{ -\frac{4C_4\tilde C_{\lambda,\rho}\eta_0}{\lambda \rho}\left(1-e^{-\frac{\lambda}{4}}\right)\right\}+ Et  \\
			&\le 1+ Et
		\end{align*}
		for $0<t \le T$, and it follows that
		\begin{equation} \label{LA68}
		\begin{aligned}
			G(t) &\le (1+Et)\exp\left\{\frac{4C_4\tilde C_{\lambda,\rho}\eta_0}{\lambda \rho}\left(1-e^{-\frac{\lambda}{4}(1+t)^\rho}\right)\right\} \le (1+Et) \exp\left\{\frac{4C_4\tilde C_{\lambda,\rho}\eta_0}{\lambda \rho} \right\}
		\end{aligned}
		\end{equation}
		for all $0< t \le T$.\\
		Inserting \eqref{LA68} into \eqref{LA66} and using the condition \eqref{cond5apriori}, we deduce for $0<t \le T$
		\begin{align*}
			\|h(t)\|_{L^p_vL^\infty_x} &\le C_4 \eta_0  \exp\left\{\frac{4C_4\tilde C_{\lambda,\rho}\eta_0}{\lambda \rho} \right\}(1+Et)e^{-\frac{\lambda}{2}(1+t)^\rho}+E.
		\end{align*}
		We first choose $\ve >0$ small enough, then choose $N>0$ sufficiently large, $\delta>0$ sufficiently small, and assume $\mathcal{E}(F_0) \le \ve_1(\eta_0)$ with small enough $\ve_1(\eta_0)<\ve_2$, which is determined in Lemma \ref{Rfestimatesmall},  so that
		\begin{align*}
			E \le \min\left\{1,\frac{\bar{\eta}}{4}\right\}.
		\end{align*}
		Using the condition \eqref{cond5apriori}, this yields for $0<t \le T$
		\begin{align} \label{LA69}
			\|h(t)\|_{L^p_vL^\infty_x}& \le 2\hat C_{\lambda,\rho} C_4 \eta_0  e^{-\frac{\lambda}{4}(1+t)^\rho}+E\nonumber\\
			& \le \frac{1}{4}\bar{\eta}e^{-\frac{\lambda}{4}(1+t)^\rho} +E\nonumber\\
			& \le \frac{1}{2} \bar{\eta},
		\end{align}
		where we have used $(1+t)e^{-\frac{\lambda}{4}(1+t)^\rho} \le \hat C_{\lambda,\rho}$.
		Hence we have closed the a priori assumption over $t \in [0,T)$ if $\mathcal{E}(F_0) \le \ve_1$.\\
		We claim that a solution to the Boltzmann equation \eqref{Boltzmanneq} extends into time interval $[0,T)$. From Lemma \ref{local estimate}, there exists the Boltzmann solution $F(t) \ge 0$ to \eqref{Boltzmanneq} on $[0,\hat{t}_0]$ such that
		\begin{align} \label{LA70}
			\sup_{0\le t \le  \hat{t}_0} \|h(t)\|_{L^p_vL^\infty_x} \le 2 \|h_0\|_{L^p_vL^\infty_x} \le \frac{1}{2} \bar{\eta}.
		\end{align}
		Taking $t=\hat{t}_0$ as the initial time, it follows from \eqref{LA70} and Theorem \ref{local estimate} that we can extend the Boltzmann equation solution $F(t) \ge 0$ into time interval $[0, \hat{t}_0+t_*]$ satisfying
		\begin{align*}
			\sup_{\hat{t}_0 \le t \le \hat{t}_0+t_*} \|h(t)\|_{L^p_vL^\infty_x} \le 2 \|h(\hat{t}_0)\|_{L^p_vL^\infty_x} \le \bar{\eta}.
		\end{align*}
		Thus we have
		\begin{align} \label{LA71}
			\sup_{0 \le t \le \hat{t}_0+t_*} \|h(t)\|_{L^p_vL^\infty_x} \le \bar{\eta}.
		\end{align}
		Note that \eqref{LA71} means $h(t)$ satisfies the a priori assumption \eqref{apriorismall} over $[0,\hat{t}_0+t_*]$.
		From \eqref{LA69}, we can obtain
		\begin{align*}
			\sup_{0 \le t \le \hat{t}_0+t_*} \|h(t)\|_{L^p_vL^\infty_x} \le\frac{1}{2}\bar{\eta}.
		\end{align*}
		Repeating the same process infinitely many times, we can derive that there exists the Boltzmann equation solution $F(t)\ge 0$ on the time interval $[0,T)$ such that
		\begin{align*}
			\sup_{0\le t < T} \|h(t)\|_{L^p_vL^\infty_x} \le \frac{1}{2} \bar{\eta}.
		\end{align*}
		Since $T$ is arbtrary, we can conclude global-in-time existence and uniqueness.\\
		\newline
	\indent We now turn to the sub-exponential decay of the solution to the Boltzmann equation. We recall the semigroup
\begin{align*}
	G_v(t,s)= \exp\left\{-\int_s^t R(f)(\tau,x-v(t-\tau),v))d\tau\right\}.
\end{align*}
By Lemma \ref{Rfestimatesmall}, note that
	\begin{align*}
		&G_v(t,s) \le e^{-\frac{\nu(v)}{2}(t-s)},\\
		&G_v(t,s) \le e^{-\frac{\lambda}{2}(1+t)^\rho} e^{\frac{\lambda}{2}(1+s)^\rho}.
	\end{align*}
By Duhamel's principle,
\begin{align*}
|h(t,x,v)|
&\le \big|G_v(t,0)h_0(x-vt,v)\big| \\
&\quad + \int_0^t \big|G_v(t,s)w\Gamma^+(f,f)\big(s,x-v(t-s),v\big)\big|ds \\
&\quad + \int_0^t \big|G_v(t,s)K_w h\big(s,x-v(t-s),v\big)\big|ds \\
&=: I_1+I_2+I_3.
\end{align*}

\vspace{6pt}
\noindent
First,
\begin{align} \label{smallI1}
	\|I_1\|_{L^p_vL^\infty_x}\le \big\|G_v(t,0)h_0\big\|_{L^p_vL^\infty_x}
\le C_{p,\gamma}  e^{-\frac{\lambda}{2}(1+t)^\rho}\|h_0\|_{L^p_vL^\infty_x}.
\end{align}
\newline
Next, for $I_2$, we use Lemma \ref{pointwiseGamma+estimate} to get
\begin{align*}
I_2
&\le C_{p,\gamma,\beta}\int_0^t e^{-\frac{\nu(v)}{4}(t-s)}e^{-\frac{\lambda}{4}(1+t)^\rho}e^{\frac{\lambda}{4}(1+s)^\rho}\nu(v)
\frac{\|h(s)\|_{L^p_vL^\infty_x}^2}{(1+|v|)^{\mathcal{M}(p,\gamma)}}ds \\
&\le \frac{C_{p,\gamma,\beta}}{(1+|v|)^{\mathcal{M}(p,\gamma)}}e^{-\frac{\lambda}{4}(1+t)^\rho}\sup_{0\le s\le t}\Big[e^{\frac{\lambda}{4}(1+s)^\rho}\|h(s)\|_{L^p_vL^\infty_x}^2\Big],
\end{align*}
where
\begin{align*}
	\mathcal{M}(p,\gamma) =\begin{cases}
\frac{\gamma(p-1)}{p}+\frac{5p-5}{4p}, & \text{if }-1\le \gamma <0 ,\\
\frac{p-1}{p}\varpi, & \text{if }-3 < \gamma <-1,
\end{cases}
\end{align*}
which implies that
\begin{align} \label{smallI2}
	\|I_2\|_{L^p_vL^\infty_x} \le C_{p,\gamma,\beta}e^{-\frac{\lambda}{4}(1+t)^\rho}\sup_{0\le s\le t}\Big[ e^{\frac{\lambda}{4}(1+s)^\rho}\|h(s)\|_{L^p_vL^\infty_x}^2\Big].	
\end{align}
From now on, let us consider \(I_3\).
\begin{align*}
I_3 
&\le \int_0^t G_v(t,s)\int_{\mathbb R^3} |k_w(v,u)||h(s,x-v(t-s),u)|duds\\
&\le \int_0^t G_v(t,s)\int_{\mathbb R^3} |k_w(v,u)|
G_u(s,0)|h_0(x-v(t-s)-u s,u)|duds\\
&\quad +\int_0^t G_v(t,s)\int_{\mathbb R^3} |k_w(v,u)|
\int_0^s G_u(s,0)|w\Gamma^+(f,f)(s',x-v(t-s)-u(s-s'),u)|ds'duds\\
&\quad +\int_0^t G_v(t,s) \int_{\mathbb R^3} |k_w(v,u)|
\int_0^s G_u(s,0)\\
&\qquad \times \int_{\mathbb R^3} |k_w(u,u^*)||h(s',x-v(t-s)-u(s-s'),u^*)|du^*ds'duds\\
&=: I_{31}+I_{32}+I_{33}.
\end{align*}
For $I_{31}$, we use Corollary \ref{pKestimate} to derive
\begin{align*}
I_{31}
&\le \int_0^t e^{-\frac{\lambda}{2}(1+t)^\rho}\|h_0\|_{L^p_vL^\infty_x}\Big(\int_{\mathbb R^3}|k_w(v,u)|^{p'}du\Big)^{1/p'}ds\\
&\le C_{p,\gamma}te^{-\frac{\lambda}{2}(1+t)^\rho}\langle v\rangle^{\gamma-1-\frac{1}{p'}}\|h_0\|_{L^p_vL^\infty_x}\\
&\le C_{p,\gamma,\lambda,\rho}e^{-\frac{\lambda}{4}(1+t)^\rho}\langle v\rangle^{\gamma-1-\frac{1}{p'}}\|h_0\|_{L^p_vL^\infty_x}.
\end{align*}
Thus we can deduce 
\begin{align} \label{smallI31}
	\|I_{31}\|_{L^p_v L^\infty_x} \le C_{p,\gamma,\lambda,\rho}e^{-\frac{\lambda}{4}(1+t)^\rho}\|h_0\|_{L^p_vL^\infty_x}.
\end{align}
For $I_{32}$, we use Lemma \ref{pointwiseGamma+estimate} and Corollary \ref{pKestimate} to obtain
\begin{align*}
I_{32} &\le C_{p,\gamma,\beta}\int_0^t G_v(t,s)\int_{\mathbb{R}^3} |k_w(v,u)| \int_0^s G_u(s,s')
\frac{\|h(s')\|_{L^p_vL^\infty_x}^2}{(1+|u|)^{\mathcal{M}(p,r)}}ds'duds\\
&\le C_{p,\gamma,\beta}e^{-\frac{\lambda}{4} (1+t)^\rho}\sup_{0\le s\le t}\Big[ e^{\frac{\lambda}{4} (1+s)^\rho}\|h(s)\|_{L^p_vL^\infty_x}^2\Big] \int_0^t e^{-\frac{\nu(v)}{2}(t-s)}\langle v \rangle^{\gamma-2} ds\\
&\le C_{p,\gamma,\beta}\langle v\rangle^{-2}e^{-\frac{\lambda}{4} (1+t)^\rho}\sup_{0\le s\le t}\Big[ e^{\frac{\lambda}{4} (1+s)^\rho}\|h(s)\|_{L^p_vL^\infty_x}^2\Big].
\end{align*}
Thus we can derive
\begin{align} \label{smallI32}
	\|I_{32}\|_{L^p_vL^\infty_x}\le C_{p,\gamma,\beta}e^{-\frac{\lambda}{4} (1+t)^\rho} \sup_{0\le s\le t}\Big[ e^{\frac{\lambda}{4} (1+s)^\rho}\|h(s)\|_{L^p_vL^\infty_x}^2\Big].
\end{align}
Let us consider $I_{33}$. We split it into 4 terms.
\begin{align*}
&\int_0^t G_v(t,s)\int_{\mathbb{R}^3} |k_w(v,u)|\int_0^s G_u(s,s')
\\
&\quad \times \int_{\mathbb{R}^3} |k_w(u,u^*)||h(s',x-v(t-s)-u(s-s'),u^*)|du^*ds'duds.
\end{align*}
\newline
$\mathbf{Case\ 1:}$ $|v|\ge N$.\\
From Corollary \ref{pKestimate}, we obtain
\begin{align*}
I_{33}^1 &\le C_{p,\gamma}\int_0^t G_v(t,s)\mathbf{1}_{\{|v|\ge N\}}
\int_{\mathbb{R}^3} |k_w(v,u)|\int_0^s G_u(s,s')\|h(s')\|_{L^p_vL^\infty_x}ds'duds\\
&\le C_{p,\gamma} \int_0^t G_v(t,s)\mathbf{1}_{\{|v|\ge N\}}
\int_{\mathbb{R}^3} |k_w(v,u)|e^{-\frac{\lambda}{4}(1+s)^\rho}\sup_{0\le s'\le s}\Big[e^{-\frac{\lambda}{4} (1+s')^\rho}\|h(s')\|_{L^p_vL^\infty_x}\Big]duds\\
&\le C_{p,\gamma}\int_0^t G_v(t,s) \mathbf{1}_{\{|v|\ge N\}} \langle v\rangle^{\gamma-2} e^{-\frac{\lambda}{4}(1+s)^{\rho}}\sup_{0\le s'\le s}\Big[e^{-\frac{\lambda}{4} (1+s')^\rho}\|h(s')\|_{L^p_vL^\infty_x}\Big]ds \\
&\le \frac{C_{p,\gamma}}{N}\langle v\rangle^{-2} e^{-\frac{\lambda}{4}(1+t)^{\rho}} 
\sup_{0\le s\le t}\Big[ e^{\frac{\lambda}{4}(1+s)^{\rho}}\|h(s)\|_{L^p_v L^\infty_x}\Big],
\end{align*}
and it follows that
\begin{align} \label{smallI331}
	\|I_{33}^1\|_{L^p_v L^\infty_x}
\le \frac{C_{p,\gamma}}{N} e^{-\frac{\lambda}{4}(1+t)^{\rho}}\sup_{0\le s\le t}\Big[ e^{\frac{\lambda}{4}(1+s)^{\rho}}\|h(s)\|_{L^p_v L^\infty_x}\Big].
\end{align}
\newline
$\mathbf{Case\ 2:}$ $|v|\le N, |u| \ge 2N$.\\
From Lemma \ref{Ksingular} and Lemma \ref{Knonsing}, we obtain
\begin{align*}
	I_{33}^2 
&\le C_{p,\gamma}\int_0^t G_v(t,s)\mathbf{1}_{\{|v|\le N\}}
\int_{|u|\ge 2N} |k_w(v,u)| \int_0^s G_u(s,s')\|h(s')\|_{L^p_v L^\infty_x}\langle u\rangle^{\gamma-1-\frac{1}{p'}} ds'duds \\
&\le \frac{C_{p,\gamma}}{N}\int_0^t G_v(t,s) \mathbf{1}_{\{|v|\le N\}}
\int_{|u|\ge 2N} |k_w(v,u)|  e^{-\frac{\lambda}{4}(1+s)^{\rho}} 
\sup_{0\le s'\le s}\Big[ e^{\frac{\lambda}{4}(1+s')^{\rho}}\|h(s')\|_{L^p_v L^\infty_x}\Big] duds \\
&\le \frac{C_{p,\gamma}}{N}\langle v\rangle^{-2} e^{-\frac{\lambda}{4}(1+t)^{\rho}}\sup_{0\le s\le t}\Big[ e^{\frac{\lambda}{4}(1+s)^{\rho}}\|h(s)\|_{L^p_v L^\infty_x}\Big].
\end{align*}
It follows that
\begin{align} \label{smallI332}
	\|I_{33}^2\|_{L^p_v L^\infty_x}
\le \frac{C_{p,\gamma}}{N} e^{-\frac{\lambda}{4}(1+t)^{\rho}}\sup_{0\le s\le t}\Big[ e^{\frac{\lambda}{4}(1+s)^{\rho}}\|h(s)\|_{L^p_v L^\infty_x}\Big].
\end{align}
\newline
$\mathbf{Case\ 3:}$ $|v|\le N, |u| \le 2N, |u^*|\ge 3N$.\\
We apply Lemma \ref{Kker} to get
\begin{align*}
I_{33}^3 
&\le \int_0^t G_v(t,s)\mathbf{1}_{\{|v|\le N\}}
\int_{|u|\le 2N} |k_w(v,u)| \int_0^s G_u(s,s')\int_{|u^*|\ge 3N}\big(|k_{w,1}(u,u^*)|+|k_{w,2}(u,u^*)|\big)\\
&\quad \times \big|h(s',x-v(t-s)-u(s-s'),u^*)\big|du^*ds'duds\\
&\le C_\gamma \int_0^t G_v(t,s)\mathbf{1}_{\{|v|\le N\}}
\int_{|u|\le 2N} |k_w(v,u)| \int_0^s G_u(s,s')\\
&\quad\times \int_{|u^*|\ge 3N}
\Big[|u-u^*|^\gamma e^{-\frac{|u|^2}{4}}e^{-\frac{|u^*|^2}{4}} 
+\frac{1}{|u-u^*|^{\frac{3-\gamma}{2}}} e^{-\frac{||u|^2-|u^*|^2|^2}{8|u-u^*|^2}}
\Big]\frac{w(u)}{w(u^*)}\\
& \quad \times |h(s',x-v(t-s)-u(s-s'),u^*)|du^*ds'duds\\
& \le C_\gamma \int_0^t G_v(t,s) \mathbf{1}_{\{|v|\le N\}}
\int_{|u|\le 2N} |k_w(v,u)| \int_0^s G_u(s,s')\Big(\frac{1}{N^{-\gamma}}e^{-\frac{1-q}{8}|u|^2} + \frac{1}{N^{\frac{3+\gamma}{2}}}\frac{\nu(u)}{1+|u|}\Big)\\
&\quad \times \|h(s')\|_{L^p_v L^\infty_x}ds'duds\\
&\le C_\gamma\left(\frac{1}{N^{-\gamma}}+\frac{1}{N^{\frac{3+\gamma}{2}}}\right)
\int_0^t G_v(t,s)\mathbf{1}_{\{|v|\le N\}}
\int_{|u|\le 2N} |k_w(v,u)| \int_0^s e^{-\frac{\nu(u)}{4}(s-s')}\nu(u)\\
&\quad \times e^{-\frac{\lambda}{4}(1+s)^\rho} e^{\frac{\lambda}{4}(1+s')^\rho}
\big\|h(s')\big\|_{L^p_v L^\infty_x}ds'duds\\
&\le C_\gamma\left(\frac{1}{N^{-\gamma}}+\frac{1}{N^{\frac{3+\gamma}{2}}}\right)
\int_0^t G_v(t,s)\mathbf{1}_{\{|v|\le N\}} \langle v \rangle^{\gamma-2} e^{-\frac{\lambda}{4}(1+t)^\rho}
\sup_{0\le s\le t}\Big[e^{\frac{\lambda}{4}(1+s)^\rho}\|h(s)\|_{L^p_v L^\infty_x}\Big]ds\\
&\le C_\gamma\left(\frac{1}{N^{-\gamma}}+\frac{1}{N^{\frac{3+\gamma}{2}}}\right)
\frac{1}{1+|v|}e^{-\frac{\lambda}{4}(1+t)^\rho}\sup_{0\le s\le t}\Big[e^{\frac{\lambda}{4}(1+s)^\rho}\|h(s)\|_{L^p_v L^\infty_x}\Big].
\end{align*}
Thus,
\begin{align} \label{smallI333}
	\|I^3_{33}\|_{L^p_vL^\infty_x}\le C_\gamma\left(\frac{1}{N^{-\gamma}}+\frac{1}{N^{\frac{3+\gamma}{2}}}\right)
e^{-\frac{\lambda}{4}(1+t)^\rho}\sup_{0\le s\le t}\Big[e^{\frac{\lambda}{4}(1+s)^\rho}\|h(s)\|_{L^p_v L^\infty_x}\Big].
\end{align}
\newline
$\mathbf{Case\ 4:}$ $|v|\le N, |u| \le 2N, |u^*|\le 3N$.\\
We decompose it into 4 terms:
\begin{align*}
I^4_{33}
&\le \int_0^t G_v(t,s)\mathbf{1}_{\{|v|\le N\}}
\int_{|u|\le 2N} |k_w(v,u)| \int_0^{s-\delta} G_u(s,s') \int_{|u^*|\le 3N} |k_w(u,u^*)| \\
&\qquad\times |h(s',x-v(t-s)-u(s-s'),u^*)|du^*ds'duds \\
&\le \int_0^t G_v(t,s)\mathbf{1}_{\{|v|\le N\}}
\int_{|u|\le 2N} |k_w^{s}(v,u)| \int_0^{s-\delta} G_u(s,s') \int_{|u^*|\le 3N} |k_w(u,u^*)|\\
&\qquad\times |h(s',x-v(t-s)-u(s-s'),u^*)|du^*ds'duds \\
&\quad + \int_0^t G_v(t,s)\mathbf{1}_{\{|v|\le N\}}
\int_{|u|\le 2N} |k_w^{ns}(v,u)| \int_0^{s-\delta} G_u(s,s') \int_{|u^*|\le 3N} |k_w^{s}(u,u^*)|\\
&\qquad\times |h(s',x-v(t-s)-u(s-s'),u^*)|du^*ds'duds \\
&\quad + \int_0^t G_v(t,s)\mathbf{1}_{\{|v|\le N\}}
\int_{|u|\le 2N} |k_w^{ns}(v,u)| \int_0^{s-\delta} G_u(s,s')\int_{|u^*|\le 3N} |k_w^{ns}(u,u^*)|\\
&\qquad\times |h(s',x-v(t-s)-u(s-s'),u^*)|du^*ds'duds \\
&\quad + \int_0^t G_v(t,s)\mathbf{1}_{\{|v|\le N\}}
\int_{|u|\le 2N} |k_w(v,u)| \int_{s-\delta}^{s} G_u(s,s')\int_{|u^*|\le 3N} |k_w(u,u^*)|\\
&\qquad\times |h(s',x-v(t-s)-u(s-s'),u^*)|du^*ds'duds \\
&=: J_1+J_2+J_3+J_4,
\end{align*}
where
\[
X'(s') := x- v(t-s) - u(s-s').
\]
For $J_1$, we use Corollary \ref{pKestimate} and Lemma \ref{Ksingular} to derive
\begin{align*}
J_1
&\le C_{p,\gamma}\int_0^t G_v(t,s)\mathbf{1}_{\{|v|\le N\}}
\int_{|u|\le 2N} \big|k_w^s(v,u)\big|\int_0^s  G_u(s,s') \langle v\rangle^{\gamma-1-\frac{1}{p'}}
\|h(s')\|_{L^p_vL^\infty_{x}}ds'duds\\
&\le C_{p,\gamma}\int_0^t G_v(t,s) \mathbf{1}_{\{|v|\le N\}}
\int_{|u|\le 2N} \big|k_w^s(v,u)\big|
e^{-\frac{\lambda}{4}(1+s)^\rho}\sup_{0\le s'\le s}\Big[ e^{\frac{\lambda}{4}(1+s')^\rho}\|h(s')\|_{L^p_vL^\infty_{x}}\Big]duds\\
&\le C_{p,\gamma}\ve^{\gamma+3}\int_0^t G_v(t,s)\mathbf{1}_{\{|v|\le N\}}\mu(v)^{\frac{1-q}{8}}
 e^{-\frac{\lambda}{4}(1+s)^\rho}\sup_{0\le s'\le s}\Big[ e^{\frac{\lambda}{4}(1+s')^\rho}\|h(s')\|_{L^p_vL^\infty_{x}}\Big]ds\\
&\le C_{p,q,\gamma}\ve^{\gamma+3}\frac{1}{1+|v|} e^{-\frac{\lambda}{4}(1+t)^\rho}\sup_{0\le s\le t}\Big[ e^{\frac{\lambda}{4}(1+s)^\rho}\|h(s)\|_{L^p_vL^\infty_{x}}\Big].
\end{align*}
Thus,
\begin{align} \label{smallJ1}
	\|J_1\|_{L^p_vL^\infty_x}
\le C_{p,q,\gamma}\ve^{\gamma+3} e^{-\frac{\lambda}{4}(1+t)^\rho}\sup_{0\le s\le t}\Big[ e^{\frac{\lambda}{4}(1+s)^\rho}\|h(s)\|_{L^p_vL^\infty_{x}}\Big].
\end{align}
For $J_2$, we apply Lemma \ref{Ksingular} and Corollary \ref{pKestimate} to derive
\begin{align*}
	J_2
&\le \int_0^t G_v(t,s)\mathbf{1}_{\{|v|\le N\}}
\int_{|u|\le 2N} \big|k_w^{ns}(v,u)\big|
\int_0^{s-\delta}  G_u(s,s') 
\int_{|u^*|\le 3N} \big|k_w^{s}(u,u^*)\big|\\
&\quad\times |h(s',X'(s'),u^*)|du^*ds'duds\\
&\le C_{p,\gamma}\ve^{\gamma+\frac{3}{p'}}\int_0^t G_v(t,s) \mathbf{1}_{\{|v|\le N\}}
\int_{|u|\le 2N} \big|k_w^{ns}(v,u)\big| \\
&\quad \times \int_0^{s-\delta}   G_u(s,s') \mu(u)^{\frac{1-q}{8}}\|h(s')\|_{L^p_vL^\infty_{x}}ds'duds\\
&\le C_{p,\gamma}\ve^{\gamma+\frac{3}{p'}}\int_0^t G_v(t,s) \mathbf{1}_{\{|v|\le N\}}\\
&\quad \times 
\int_{|u|\le 2N} \big|k_w^{ns}(v,u)\big|e^{-\frac{\lambda}{4}(1+s)^\rho}\sup_{0\le s'\le s}\Big[ e^{\frac{\lambda}{4}(1+s')^\rho}\|h(s')\|_{L^p_vL^\infty_{x}}\Big]duds\\
&\le C_{p,\gamma}\varepsilon^{\gamma+\frac{3}{p'}}
\int_{0}^{t} G_v(t,s)\mathbf{1}_{\{|v|\le N\}}\\
&\quad \times 
\int_{|u|\le 2N}|k_w(v,u)|e^{-\frac{\lambda}{4}(1+s)^\rho}\sup_{0\le s'\le s}\Big[ e^{\frac{\lambda}{4}(1+s')^\rho}\|h(s')\|_{L^p_vL^\infty_{x}}\Big]duds\\
&\le C_{p,\gamma}\varepsilon^{\gamma+\frac{3}{p'}} e^{-\frac{\lambda}{4}(1+t)^\rho}\sup_{0\le s\le t}\Big[ e^{\frac{\lambda}{4}(1+s)^\rho}\|h(s)\|_{L^p_vL^\infty_{x}}\Big]
\int_{0}^{t} G_v(t,s)\mathbf{1}_{\{|v|\le N\}}
\langle v\rangle^{\gamma-2}ds\\
&\le C_{p,\gamma}\varepsilon^{\gamma+\frac{3}{p'}}
\langle v\rangle^{-2} e^{-\frac{\lambda}{4}(1+t)^\rho}\sup_{0\le s\le t}\Big[ e^{\frac{\lambda}{4}(1+s)^\rho}\|h(s)\|_{L^p_vL^\infty_{x}}\Big].
\end{align*}
Thus,
\begin{align} \label{smallJ2}
	\|J_2\|_{L^p_vL^\infty_x} \le C_{p,\gamma}\varepsilon^{\gamma+\frac{3}{p'}} e^{-\frac{\lambda}{4}(1+t)^\rho}\sup_{0\le s\le t}\Big[ e^{\frac{\lambda}{4}(1+s)^\rho}\|h(s)\|_{L^p_vL^\infty_{x}}\Big],
\end{align}
where $p>\frac{6}{3+\gamma}$.\\
For $J_3$, we deduce
\begin{align*}
	J_3 &\le \int_{0}^{t} G_v(t,s)\mathbf{1}_{\{|v|\le N\}}
\int_{|u|\le 2N}|k_w^{ns}(v,u)|
\int_{0}^{s-\delta} G_u(s,s')\int_{|u^{*}|\le 3N}|k_w^{ns}(u,u^{*})|\\
&\quad\times |h(s',x-v(t-s)-u(s-s'),u^{*})|du^{*}ds'duds\\
&\le C(1+2N)^{-\gamma}\int_{0}^{t} G_v(t,s)\mathbf{1}_{\{|v|\le N\}}
\int_{|u|\le 2N}\int_{0}^{s-\delta} G_u(s,s')\nu(u)\\
&\quad\times\int_{|u^{*}|\le 3N}|k_w^{ns}(v,u)||k_w^{ns}(u,u^{*})|
|h(s',x-v(t-s)-u(s-s'),u^{*})|du^{*}ds'duds\\
&\le C_{N}\int_{0}^{t} G_v(t,s)\mathbf{1}_{\{|v|\le N\}}
e^{-\frac{\lambda}{2}s^{\rho}}
\sup_{0\le s'\le s-\delta}\Bigg[ e^{\frac{\lambda}{2}{s'}^{\rho}}\int_{|u|\le 2N}\int_{|u^{*}|\le 3N}|k_w^{ns}(v,u)||k_w^{ns}(u,u^{*})|\\
&\quad \times|h(s',x-v(t-s)-u(s-s'),u^{*})|du^{*}du\Bigg]ds\\
&\le C_N \int_0^t  G_v(t,s)\mathbf{1}_{\{|v|\le N\}}
\big\| k_w^{ns}(v,u)k_w^{ns}(u,u^*)\big\|_{L^2_{u,u^*}}
e^{-\frac{\lambda}{4}(1+s)^\rho} \\
&\quad\times
\sup_{0\le s'\le s-\delta}\Bigg[ e^{\frac{\lambda}{4}(1+s')^\rho}\Bigg(\int_{|u|\le 2N}\int_{|u^*|\le 3N}
\big|h(s',x-v(t-s)-u(s-s'),u^*)\big|^2du^*du\Bigg)^{1/2}\Bigg]ds,
\end{align*}
where $\big\| k_w^{ns}(v,u)k_w^{ns}(u,u^*)\big\|_{L^2_{u,u^*}}$ is bounded.
Making a change of variables \(u\mapsto y=x-v(t-s)-u(s-s')\) with $\left|\frac{dy}{du}\right|=(s-s')^3\ge \delta^3$, we obtain
\begin{align*}
J_3 &\le C_{N,\delta}\int_0^t  G_v(t,s)\mathbf{1}_{\{|v|\le N\}}
\big\| k_w^{ns}(v,u)k_w^{ns}(u,u^*)\big\|_{L^2_{u,u^*}} e^{-\frac{\lambda}{4}(1+s)^\rho} \\
&\quad\times \sup_{0\le s'\le s}\Bigg[ e^{\frac{\lambda}{4}(1+s')^\rho}\Bigg(\int_{\T^3}  \int_{|u^*|\le 3N}
\big|h(s',y,u^*)\big|^2du^*dy\Bigg)^{1/2}\Bigg]ds\\
&\le C_{N,\delta} e^{-\frac{\lambda}{4}(1+t)^\rho} \sup_{0\le s\le t}\Big[ e^{\frac{\lambda}{4}(1+s)^\rho}\|f(s)\|_{L^2_{x,v}}\Big] \int_0^t e^{-\frac{\nu(v)}{2}(t-s)}\mathbf{1}_{\{|v|\le N\}}\nu(v)ds\\
&\le C_{N,\delta} e^{-\frac{\lambda}{4}(1+t)^\rho}\mathbf{1}_{\{|v|\le N\}} \sup_{0\le s\le t}\Big[ e^{\frac{\lambda}{4}(1+s)^\rho}\|f(s)\|_{L^2_{x,v}}\Big].
\end{align*}
Using Lemma \ref{L2decayunderapriori}, the above inequality becomes
\begin{align*}
J_3 &\le C_{p,q,\gamma,\beta,N,\delta} e^{-\frac{\lambda}{4}(1+t)^\rho}\mathbf{1}_{\{|v|\le N\}} \|w_{q,\beta} f_0\|_{L^p_vL^\infty_x}\\
&\le C_{p,q,\gamma,\beta,N,\delta} e^{-\frac{\lambda}{4}(1+t)^\rho}\mathbf{1}_{\{|v|\le N\}} \|h_0\|_{L^p_vL^\infty_x}.
\end{align*}
Thus,
\begin{align} \label{smallJ3}
	\|J_3\|_{L^p_vL^\infty_x}\le C_{p,q,\gamma,\beta,N,\delta} e^{-\frac{\lambda}{4}(1+t)^\rho} \|h_0\|_{L^p_vL^\infty_x}.
\end{align}
For $J_4$, we use Corollary \ref{pKestimate} to deduce
\begin{align*}
J_4
&\le \int_{0}^{t} G_{v}(t,s)\mathbf{1}_{\{|v|\le N\}}\int_{|u|\le 2N}|k_{w}(v,u)| \int_{s-\delta}^{s} G_{u}(s,s')\int_{|u^{*}|\le 3N}|k_{w}(u,u^{*})|\\
 &\quad\times |h(s',x-v(t-s)-u(s-s'),u^*)|du^{*}ds'duds\\
&\le C_{p,\gamma}\int_{0}^{t} G_{v}(t,s)\mathbf{1}_{\{|v|\le N\}}\int_{|u|\le 2N}|k_{w}(v,u)|
\int_{s-\delta}^{s} G_{u}(s,s')\langle u\rangle^{\gamma-1-\frac{1}{p'}}
\|h(s')\|_{L^p_vL^\infty_{x}}ds'duds\\
&\le C_{p,\gamma}\delta \int_{0}^{t} G_{v}(t,s)\mathbf{1}_{\{|v|\le N\}}
\int_{|u|\le 2N}|k_{w}(v,u)|e^{-\frac{\lambda}{4}(1+s)^{\rho}}
\sup_{0\le s'\le s}\Big[ e^{\frac{\lambda}{4}{(1+s')}^{\rho}}\|h(s')\|_{L^p_vL^\infty_{x}}\Big]duds\\
&\le C_{p,\gamma,\ve}\delta\int_{0}^{t} G_{v}(t,s)\mathbf{1}_{\{|v|\le N\}}
\langle v\rangle^{\gamma-2} e^{-\frac{\lambda}{4}(1+s)^{\rho}} \sup_{0\le s'\le s}\Big[ e^{\frac{\lambda}{4}{(1+s')}^{\rho}}\|h(s')\|_{L^p_vL^\infty_{x}}\Big]ds\\
&\le C_{p,\gamma,\ve}\delta\langle v\rangle^{-2} e^{-\frac{\lambda}{4}(1+t)^{\rho}}
\sup_{0\le s\le t}\Big[ e^{\frac{\lambda}{4}(1+s)^{\rho}}\|h(s)\|_{L^p_vL^\infty_{x}}\Big].
\end{align*}
Thus,
\begin{align} \label{smallJ4}
	\|J_4\|_{L^p_vL^\infty_{x}}
\le C_{p,\gamma,\ve}\delta e^{-\frac{\lambda}{4}(1+t)^{\rho}}
\sup_{0\le s\le t}\Big[ e^{\frac{\lambda}{4}(1+s)^{\rho}}\|h(s)\|_{L^p_vL^\infty_{x}}\Big].
\end{align}
Gathering \eqref{smallI1}, \eqref{smallI2}, \eqref{smallI31}, \eqref{smallI32}, \eqref{smallI331}, \eqref{smallI332}, \eqref{smallI333}, \eqref{smallJ1}, \eqref{smallJ2}, \eqref{smallJ3}, and \eqref{smallJ4}, we obtain
\begin{align*}
\|h(t)\|_{L^p_vL^\infty_{x}}
&\le C_{p,q,\gamma,\beta,N,\delta,\lambda,\rho}e^{-\frac{\lambda}{4}(1+t)^{\rho}}\|h_{0}\|_{L^p_vL^\infty_{x}}\\
&\quad +C_{p,\gamma,\beta}e^{-\frac{\lambda}{4}(1+t)^{\rho}}
\sup_{0\le s\le t}\Big[ e^{\frac{\lambda}{4}(1+s)^{\rho}}\|h(s)\|_{L^p_vL^\infty_{x}}^{2}\Big]\\
&\quad +\left(\frac{C_{p,\gamma,\varepsilon}}{N}+C_{p,\gamma}\ve^{\gamma+\frac{3}{p'}}+C_{p,\gamma,\ve}\delta\right)e^{-\frac{\lambda}{4}(1+t)^{\rho}} \sup_{0\le s\le t}\Big[ e^{\frac{\lambda}{4}(1+s)^{\rho}}\|h(s)\|_{L^p_vL^\infty_{x}}\Big]
\end{align*}
Using the a priori assumption \eqref{apriorismall}, we have
\begin{align*}
	\|h(t)\|_{L^p_vL^\infty_{x}}
&\le C_{p,q,\gamma,\beta,N,\delta,\lambda,\rho}e^{-\frac{\lambda}{4}(1+t)^{\rho}}\|h_{0}\|_{L^p_vL^\infty_{x}}\\
&\quad +C_{\Gamma\Gamma}\bar{\eta} e^{-\frac{\lambda}{4}(1+t)^{\rho}}
\sup_{0\le s\le t}\Big[ e^{\frac{\lambda}{4}(1+s)^{\rho}}\|h(s)\|_{L^p_vL^\infty_{x}}\Big]\\
&\quad +\left(\frac{C_{p,\gamma,\varepsilon}}{N}+C_{p,\gamma}\ve^{\gamma+\frac{3}{p'}}+C_{p,\gamma,\ve}\delta\right)e^{-\frac{\lambda}{4}(1+t)^{\rho}} \sup_{0\le s\le t}\Big[ e^{\frac{\lambda}{4}(1+s)^{\rho}}\|h(s)\|_{L^p_vL^\infty_{x}}\Big],
\end{align*}
where $C_{\Gamma \Gamma}$ is a constant depending on $p,\gamma, \beta$. Taking $\eta$ so that 
\begin{align} \label{cond4apriori}
	C_{\Gamma \Gamma} \bar{\eta} < \frac{1}{4}
\end{align}
and choosing $\ve>0$, $\delta>0$ small enough, then taking $N>0$ sufficiently large so that
\begin{align*}
	\frac{C_{p,\gamma,\varepsilon}}{N}+C_{p,\gamma}\ve^{\gamma+\frac{3}{p'}}+C_{p,\gamma,\ve}\delta < \frac{1}{4},
\end{align*}
we conclude that
\begin{align*}
	\|h(t)\|_{L^p_vL^\infty_{x}} \le C_{p,q,\gamma,\beta,\lambda,\rho} e^{-\frac{\lambda}{4}(1+t)^{\rho}}\|h_{0}\|_{L^p_vL^\infty_{x}}
\end{align*}
for all $t \ge 0$.
\end{proof}

\section{Large-amplitude perturbation problem in $L^p_vL^\infty_x$}
\label{section5}
\noindent Recall that we fixed $p, \beta, q, \vartheta$ in Section \ref{section4}.  Let $f$ satisfy the equation \eqref{FPBER} with initial data $f_0(x,v)$ and $ \mu(v)+\mu(v)^{1/2}f(t,x,v) \ge 0$ on $[0,T]$ with $0<T < \infty$.
Throughout this section, we assume the a priori bound
\begin{align} \label{apriorilarge}
	\sup_{0\le t \le T}\|w_{q,\vt,\beta}f(t)\|_{L^p_vL^\infty_x} \le \bar{M} ,
\end{align}
where $\bar{M}$ is a large constant not depending on the solution $f$ and $T$, but depending on the initial amplitude $M_0$. Later, $M_0$ and $T$ will be determined in subsection \ref{subsection44}.
\subsection{$R(f)$ estimate} 
	Recall that we denote $h=w_{q,\vt,\beta}f$. The a priori assumption \eqref{apriorilarge} for $h$ becomes
	\begin{align} \label{apriori}
		\sup_{0\le t \le T}\|h(t)\|_{L^p_vL^\infty_x}\le \bar{M}.
	\end{align}
\begin{Lem}\label{large.Rf}
	Let us fix $T > 0$ and assume a priori bound \eqref{apriori}. 
	For any $T>t_*$ where $t_*$ will be determined later, there exists a sufficiently small positive consant $\varepsilon_2 = \varepsilon_2(\bar{M},T)>0$ such that whenever $\mathcal{E}(F_0) \leq \varepsilon_2$, the following result holds 
	\begin{align*}
		R(f)(t,x,v) \geq \frac{\nu(v)}{2} + \frac{\vartheta q|v|^2}{8(1+t)^{\vartheta+1}}, \qquad \forall (t,x,v) \in [t_*,T) \times \T^3 \times \R^3.
	\end{align*} 
	
\end{Lem}
\begin{proof}
	
	For both cases in Lemma \ref{exp.h.esti}, we first choose $\delta>0$ and $\varepsilon>0$ sufficiently small, and then $N>0$ large enough, and finally the initial relative entropy $\mathcal{E}(F_0)$ is small enough, so that 
	\begin{align*}
		\int_{\R^3} e^{-\frac{|u|^2}{8}} |h(t,x,u)| du \leq  e^{-\lambda (1+t)^\rho} \Vert h_0 \Vert_{L^p_v L^\infty _x} + \frac{1}{4C}. 
	\end{align*}
	Thus, for $t \geq t_*:= \max\left\{0,\left(\frac{1}{\lambda} \ln \left(4C\Vert h_0 \Vert_{L^p_v L^\infty_x}\right)\right)^{1/\rho}-1 \right\}$, one obtains that 
	\begin{align*}
		\int_{\R^3} e^{-\frac{|u|^2}{8}} |h(t,x,u)| du \leq \frac{1}{2C},\quad \forall (t,x) \in [t_*,T) \times \T^3,
	\end{align*}
	which implies our desired result from \eqref{Rf.esti1}. 
\end{proof}
\subsection{$L^p_vL^\infty_x$ estimate}


In contrast to Section \ref{section5}, under the large-amplitude perturbation, $R(f)$ possesses a positive lower bound for $t \geq t_*$. Hence, the solution operator satisfies the following estimate under the assumption of Lemma \ref{large.Rf}
\begin{align} \label{G,v}
		G_v(t,s) &= G_v(t,s) \chi_{\{t \leq t_*\}} +G_v(t,s)\chi_{\{s \leq t_* \leq t\}} + G_v(t,s) \chi_{\{t_* \leq s \leq t \}} \nonumber\\
		&\leq e^{\lambda(1+t_*)^\rho}e^{-\lambda(1+t)^{\rho}} e^{\lambda(1+s)^\rho}\chi_{\{t \leq t_*\}} \nonumber\\
		&\quad +\exp \left\{-\int_{t_*}^t \left(\frac{\nu(v)}{2} + \frac{\vartheta q |v|^2}{8(1+\tau)^{\vartheta+1}}\right)d\tau \right\}\chi_{\{s\leq t_*\leq t \}} \nonumber \\
		&\quad + \exp \left\{-\int_{s}^t \left(\frac{\nu(v)}{2} + \frac{\vartheta q |v|^2}{8(1+\tau)^{\vartheta+1}}\right)d\tau \right\}\chi_{\{t_*\leq s \leq t \}}   \nonumber \\ 
		&\leq e^{\lambda(1+t_*)^\rho}e^{-\lambda(1+t)^{\rho}} e^{\lambda(1+s)^\rho}\chi_{\{t \leq t_*\}} \nonumber\\
		&\quad + \exp \left \{ - \int_{t_*}^t C(1+\tau)^{\frac{(1+\vartheta)\gamma}{2-\gamma}}d\tau \right \} \chi_{\{ s\leq t_* \leq t \}}  \nonumber\\
		&\quad + \exp \left \{ - \int_{s}^t C(1+\tau)^{\frac{(1+\vartheta)\gamma}{2-\gamma}}d\tau \right \} \chi_{\{ s\leq t_* \leq t \}} \nonumber \\
		&\leq e^{\lambda(1+t_*)^\rho}e^{-\lambda(1+t)^{\rho}} e^{\lambda(1+s)^\rho},
\end{align}
where $\rho =1 + \frac{(1+\vartheta)\gamma}{2-\gamma}>0$ and $\lambda = \frac{C}{\rho}>0$.  
\begin{Lem}
	Let $h(t,x,v)$ satisfy the equation \eqref{Rf nonlinear equation} and  $\rho-1 = \frac{(1+\vartheta)\gamma}{2-\gamma}$.  Let $0<t\le T<\infty$. Under the a priori assumption \eqref{apriori} and $\mathcal{E}(F_0) \leq \varepsilon_2$, the following estimates holds for each case:
	\begin{enumerate}[label=(\arabic*)]
		 \item If $-1 \leq \gamma<0$, then 
		\begin{align*}
			&|h(t,x,v)|\\ &\leq C_{\lambda}e^{\lambda(1+t_*)^\rho}e^{-\lambda(1+t)^{\rho}} h_0(x-tv,v)  + \frac{C_{t_*}}{(1+|v|)^{-\frac{\gamma}{p}+\frac{5p-5}{4p}}}e^{-\frac{\lambda}{2} (1+t)^\rho} \Vert h_0 \Vert_{L^p_v L^\infty_x} \int_0^t \Vert h(s) \Vert_{L^p_vL^\infty_x}ds \\
			&\quad +C_{t_*,p,\rho,\lambda,\gamma}e^{-\frac{\lambda}{2}(1+t)^{\rho}} \langle v \rangle^{\gamma-1-\frac{1}{p'}} \Vert h_0 \Vert_{L^p_v L^\infty_x}\\
			&\quad +\frac{C_{t_*,p,q,\gamma}}{(1+|v|)^{\frac{p-1}{4p}}} \varepsilon^{\gamma+ \frac{3}{p'}}   \left[\sup_{0 \leq s \leq t} \Vert h(s) \Vert_{L^p_v L^\infty_x}+\sup_{0 \leq s \leq t} \Vert h(s) \Vert_{L^p_v L^\infty_x}^2\right]\\
			&\quad + \frac{C_{t_*,p,q,\varepsilon,\gamma,\beta}}{(1+|v|)^{\frac{p-1}{4p}}} \left(\frac{1}{N}+\frac{1}{N^{-\gamma}}+\frac{1}{N^{\frac{\gamma+3}{2}}}+\delta  \right) \\
			& \qquad \times \left[\sup_{0\leq s \leq t} \Vert h(s) \Vert_{L^p_vL^\infty_x}+\sup_{0\leq s \leq t} \Vert h(s) \Vert_{L^p_vL^\infty_x}^2+\sup_{0\leq s \leq t} \Vert h(s) \Vert_{L^p_vL^\infty_x}^3\right]\\
			&\quad +\frac{C_{t_*,p,q,\varepsilon,\gamma,N,\delta}}{(1+|v|)^{\frac{p-1}{4p}}}\left(1+\sup_{0\leq s \leq t} \Vert h(s) \Vert_{L^p_vL^\infty_x}+\sup_{0\leq s \leq t} \Vert h(s) \Vert_{L^p_vL^\infty_x}^2\right)\\
			& \qquad \times  \left[ \mathcal{E}(F_0)^{1/2} + \sup_{0\le s \le t}\|h(s)\|_{L^p_vL^\infty_x}^{\frac{p}{2p-2}}\mathcal{E}(F_0)^{\frac{p-2}{2p-2}}\right],
		\end{align*}
		where $0<\delta\ll1, 0<\varepsilon\ll1$, and $N\gg1$ can be chosen arbitrarily small and large, respectively. 
		where $0<\delta\ll1, 0<\varepsilon\ll1$, and $N\gg1$ can be chosen arbitrarily small and large, respectively. 
		\item If $-3<\gamma<-1$, then
		\begin{align*}
			&|h(t,x,v)|\\ &\leq C_{\lambda}e^{\lambda(1+t_*)^\rho}e^{-\lambda(1+t)^{\rho}} h_0(x-tv,v)  +\frac{C_{t_*}}{(1+|v|)^{-\gamma+ \frac{p-1}{p}\varpi}}  e^{-\frac{\lambda}{2}(1+t)^\rho} \Vert h_0 \Vert_{L^p_v L^\infty_x} \int_0^t \Vert h(s) \Vert_{L^p_v L^\infty_x} ds \nonumber \\
			&\quad +C_{t_*,p,\rho,\lambda,\gamma}e^{-\frac{\lambda}{2}(1+t)^{\rho}} \langle v \rangle^{\gamma-1-\frac{1}{p'}} \Vert h_0 \Vert_{L^p_v L^\infty_x}\\
			&\quad +\frac{C_{t_*,p,q,\gamma}}{(1+|v|)^{\frac{p-1}{p}\varpi}} \varepsilon^{\gamma+ \frac{3}{p'}}   \left[\sup_{0 \leq s \leq t} \Vert h(s) \Vert_{L^p_v L^\infty_x}+\sup_{0 \leq s \leq t} \Vert h(s) \Vert_{L^p_v L^\infty_x}^2\right]\\
			&\quad + \frac{C_{t_*,p,q,\varepsilon,\gamma,\beta}}{(1+|v|)^{\frac{p-1}{p}\varpi}} \left(\frac{1}{N}+\frac{1}{N^{-\gamma}}+\frac{1}{N^{\frac{\gamma+3}{2}}}+\delta  \right) \\
			& \qquad \times \left[\sup_{0\leq s \leq t} \Vert h(s) \Vert_{L^p_vL^\infty_x}+\sup_{0\leq s \leq t} \Vert h(s) \Vert_{L^p_vL^\infty_x}^2+\sup_{0\leq s \leq t} \Vert h(s) \Vert_{L^p_vL^\infty_x}^3\right]\\
			&\quad +\frac{C_{t_*,p,q,\varepsilon,\gamma,N,\delta}}{(1+|v|)^{\frac{p-1}{p}\varpi}}\left(1+\sup_{0\leq s \leq t} \Vert h(s) \Vert_{L^p_vL^\infty_x}\right)\left[ \mathcal{E}(F_0)^{1/2} + \sup_{0\le s \le t}\|h(s)\|_{L^p_vL^\infty_x}^{\frac{p}{2p-2}}\mathcal{E}(F_0)^{\frac{p-2}{2p-2}}\right]\\
			& \quad +\frac{C_{t_*,p,q,\varepsilon,\gamma,N,\delta}}{(1+|v|)^{\frac{p-1}{p}\varpi}}\left(\sup_{0\leq s \leq t} \Vert h(s) \Vert_{L^p_vL^\infty_x}+\sup_{0\leq s \leq t} \Vert h(s) \Vert_{L^p_vL^\infty_x}^2\right)\\
			& \qquad \times  \Bigg[\sup_{0 \leq s \leq t}\|h(s)\|_{L^p_vL^\infty_x}^{1-\frac{1/p'm'-1/p}{1/2-1/p}}\mathcal{E}(F_0)^{\frac{1}{2}\cdot\frac{1/p'm'-1/p}{1/2-1/p}}+ \sup_{0 \leq s \leq t}\|h(s)\|_{L^p_vL^\infty_x}^{1-\frac{1/p'm'-1/p}{1-1/p}}\mathcal{E}(F_0)^{\frac{1/p'm'-1/p}{1-1/p}}\Bigg],
		\end{align*}
		where $0<\delta\ll1, 0<\varepsilon\ll1$, and $N\gg1$ can be chosen arbitrarily small and large, respectively. 
	\end{enumerate}
\end{Lem}

\begin{proof}
The proof follows exactly the same method as in Lemma \ref{small,Lp}. Since we only need to replace the estimate \eqref{small,G,v} for the solution operator with \eqref{G,v} , we will omit the detailed proof. 
\end{proof}
\begin{Coro} \label{largelp}
Let $h(t,x,v)$ satisfy the equation \eqref{Rf nonlinear equation} and  $\rho-1 = \frac{(1+\vartheta)\gamma}{2-\gamma}$. Under the a priori assumption \eqref{apriori}, we get the following estimate for each case:
\begin{enumerate}
	\item If $-1 \leq \gamma <0$, then there exists a constant $C_{4,1}$, depending on $\lambda,t_*,p,q$, and $\gamma$, so that
		\begin{align*} 
		\Vert h(t) \Vert_{L^p_v L^\infty_x} &\leq C_{4,1} \Vert h_0 \Vert_{L^p_vL^\infty_x} \left(1+\int_0^t \Vert h(s) \Vert_{L^p_v L^\infty_x}ds \right) e^{-\frac{\lambda}{2} (1+t)^\rho} \nonumber\\
		&\quad +C_{t_*,p,q,\gamma} \varepsilon^{\gamma+ \frac{3}{p'}}   \left[\sup_{0 \leq s \leq t} \Vert h(s) \Vert_{L^p_v L^\infty_x}+\sup_{0 \leq s \leq t} \Vert h(s) \Vert_{L^p_v L^\infty_x}^2\right]\nonumber\\
		&\quad + C_{t_*,p,q,\varepsilon,\gamma,\beta}\left(\frac{1}{N}+\frac{1}{N^{-\gamma}}+\frac{1}{N^{\frac{\gamma+3}{2}}}+\delta  \right) \nonumber\\
		& \qquad \times \left[\sup_{0\leq s \leq t} \Vert h(s) \Vert_{L^p_vL^\infty_x}+\sup_{0\leq s \leq t} \Vert h(s) \Vert_{L^p_vL^\infty_x}^2+\sup_{0\leq s \leq t} \Vert h(s) \Vert_{L^p_vL^\infty_x}^3\right]\nonumber\\
		&\quad +C_{t_*,p,q,\varepsilon,\gamma,N,\delta}\left(1+\sup_{0\leq s \leq t} \Vert h(s) \Vert_{L^p_vL^\infty_x}+\sup_{0\leq s \leq t} \Vert h(s) \Vert_{L^p_vL^\infty_x}^2\right)\nonumber\\
		& \qquad \times  \left[ \mathcal{E}(F_0)^{1/2} + \sup_{0\le s \le t}\|h(s)\|_{L^p_vL^\infty_x}^{\frac{p}{2p-2}}\mathcal{E}(F_0)^{\frac{p-2}{2p-2}}\right],
	\end{align*}
	for all $0\leq t \leq T$, where $\varepsilon>0$ and $\delta>0$ can be arbitrarily small and $N>0$ can be arbitrarily large.
	\item If $-3 < \gamma<-1$, then then there exists a constant $C_{4,2}$, depending on $\lambda,t_*,p,q$, and $\gamma$, so that
	\begin{align*}
		\Vert h(t) \Vert_{L^p_v L^\infty_x} &\leq C_{4,2}  \Vert h_0 \Vert_{L^p_v L^\infty_x} \left(1+\int_0^t \Vert h(s) \Vert_{L^p_v L^\infty_x} ds\right)e^{-\frac{\lambda}{2}(1+t)^\rho} \\ 
		&\quad +C_{t_*,p,q,\gamma}\varepsilon^{\gamma+ \frac{3}{p'}}   \left[\sup_{0 \leq s \leq t} \Vert h(s) \Vert_{L^p_v L^\infty_x}+\sup_{0 \leq s \leq t} \Vert h(s) \Vert_{L^p_v L^\infty_x}^2\right]\\
		&\quad + C_{t_*,p,q,\varepsilon,\gamma,\beta} \left(\frac{1}{N}+\frac{1}{N^{-\gamma}}+\frac{1}{N^{\frac{\gamma+3}{2}}}+\delta  \right) \\
		& \qquad \times \left[\sup_{0\leq s \leq t} \Vert h(s) \Vert_{L^p_vL^\infty_x}+\sup_{0\leq s \leq t} \Vert h(s) \Vert_{L^p_vL^\infty_x}^2+\sup_{0\leq s \leq t} \Vert h(s) \Vert_{L^p_vL^\infty_x}^3\right]\\
		&\quad +C_{t_*,p,q,\varepsilon,\gamma,N,\delta}\left(1+\sup_{0\leq s \leq t} \Vert h(s) \Vert_{L^p_vL^\infty_x}\right)\left[ \mathcal{E}(F_0)^{1/2} + \sup_{0\le s \le t}\|h(s)\|_{L^p_vL^\infty_x}^{\frac{p}{2p-2}}\mathcal{E}(F_0)^{\frac{p-2}{2p-2}}\right]\\
		& \quad +C_{t_*,p,q,\varepsilon,\gamma,N,\delta}\left(\sup_{0\leq s \leq t} \Vert h(s) \Vert_{L^p_vL^\infty_x}+\sup_{0\leq s \leq t} \Vert h(s) \Vert_{L^p_vL^\infty_x}^2\right)\\
		& \qquad \times  \Bigg[\sup_{0 \leq s \leq t}\|h(s)\|_{L^p_vL^\infty_x}^{1-\frac{1/p'm'-1/p}{1/2-1/p}}\mathcal{E}(F_0)^{\frac{1}{2}\cdot\frac{1/p'm'-1/p}{1/2-1/p}}\\
		&\qquad \quad + \sup_{0 \leq s \leq t}\|h(s)\|_{L^p_vL^\infty_x}^{1-\frac{1/p'm'-1/p}{1-1/p}}\mathcal{E}(F_0)^{\frac{1/p'm'-1/p}{1-1/p}}\Bigg],
	\end{align*}
	for all $0\leq t \leq T$, where $\varepsilon>0$ and $\delta>0$ can be arbitrarily small and $N>0$ can be arbitrarily large.
\end{enumerate}
\end{Coro}
\subsection{Global Existence and Sub-exponential decay}
\label{subsection44}
\begin{proof} [Proof of Theorem \ref{largemain}]
	For the convenience of the notation, we denote $w_{q,\vartheta,\beta}f$ by $h$. If we assume a prioir assumption \eqref{apriori}, then it follows from Corollary \ref{largelp} that 
	\begin{align} \label{mainesti1}
		\Vert h(t) \Vert_{L^p_v L^\infty_x} \leq C_5 \Vert h_0 \Vert_{L^p_v L^\infty_x} \left(1+ \int_0^t \Vert h(s) \Vert_{L^p_v L^\infty_x} ds \right)e^{-\frac{\lambda}{2}(1+t)^\rho}+ D,
	\end{align}
	where $C_5: = \max\{C_{4,1},C_{4,2}\}$. For the case $-1\leq \gamma <0$ and $p> 13$, we define 
		\begin{align*}
			D&:=  C_{t_*,p,q,\gamma}\varepsilon^{\gamma+\frac{3}{p'}}[\bar{M} + \bar{M}^2]+C_{t_*,p,q,\varepsilon,\gamma} \left(\frac{1}{N}+\frac{1}{N^{-\gamma}}+\frac{1}{N^{\frac{\gamma+3}{2}}}+\delta  \right) \left[\bar{M}+\bar{M}^2+\bar{M}^3\right]\\
			&\quad  +C_{t_*,p,q,\varepsilon,\gamma,N,\beta}[1+\bar{M}^2] \left(\sqrt{\mathcal{E}(F_0)}+ \bar{M}^{\frac{p}{2p-2}}\mathcal{E}(F_0)^\frac{p-2}{2p-2}\right).
		\end{align*}
		For the other case $-3 < \gamma<-1$ and $p > \frac{3(\sqrt{8\gamma^2 +9} +3-2\gamma)}{(-\gamma)(3+\gamma)}$, we define 
		\begin{align*}
			D&:= C_{t_*,p,q,\gamma}\varepsilon^{\gamma+\frac{3}{p'}}[\bar{M} + \bar{M}^2]+C_{t_*,p,q,\varepsilon,\gamma,\beta} \left(\frac{1}{N}+\frac{1}{N^{-\gamma}}+\frac{1}{N^{\frac{\gamma+3}{2}}}+\delta  \right) \left[\bar{M}+\bar{M}^2+\bar{M}^3\right]\\
			&\quad +C_{t_*,p,q,\varepsilon,\gamma,N,\delta}\left(1+\bar{M} \right)\left[ \mathcal{E}(F_0)^{1/2} + \bar{M}^{\frac{p}{2p-2}}\mathcal{E}(F_0)^{\frac{p-2}{2p-2}}\right]\\
			& \quad +C_{t_*,p,q,\varepsilon,\gamma,N,\delta}\left(\bar{M}+\bar{M}^2\right) \Bigg[\bar{M}^{1-\frac{1/p'm'-1/p}{1/2-1/p}}\mathcal{E}(F_0)^{\frac{1}{2}\cdot\frac{1/p'm'-1/p}{1/2-1/p}}+ \bar{M}^{1-\frac{1/p'm'-1/p}{1-1/p}}\mathcal{E}(F_0)^{\frac{1/p'm'-1/p}{1-1/p}}\Bigg].
		\end{align*}
		If we denote
		\begin{align*}
			G(t) := 1+ \int_0^t \Vert h(t) \Vert_{L^p_v L^\infty_x}ds,
		\end{align*}
		then we can rewrite as 
		\begin{align*}
			G'(t) -C_{t_*,\lambda} M_0 G(t)e^{-\frac{\lambda}{2}(1+t)^\rho}  \leq D. 
		\end{align*}
		Using Gr\"{o}nwall's inequality, we obtain 
		\begin{align*}
		G(t) \exp \left \{ -\frac{4CM_0}{\lambda \rho}(1-e^{-\frac{\lambda}{4}(1+t)^\rho}) \right \}\leq 1+Dt. 
		\end{align*}
		Therefore, $G(t)$ can be further bounded by 
		\begin{align*}
			G(t) \leq (1+Dt)  \exp \left \{ \frac{4CM_0}{\lambda \rho}(1-e^{-\frac{\lambda}{4}(1+t)^\rho}) \right \} \leq (1+Dt)  \exp \left \{ \frac{4CM_0}{\lambda \rho} \right \},
		\end{align*}
		for all $0\leq t \leq T$. Take 
		\begin{align}\label{Mbar}
			\bar{M} := 2C_{\lambda,t_*} M_0 \exp \left \{ \frac{4CM_0}{\lambda \rho} \right \}.
		\end{align}
		We can firstly choose sufficiently small $\varepsilon_0>0$ depending on $\bar{M}$, then choose  sufficiently large $N$ depending on $\bar{M}$ and $\varepsilon_0$, and sufficiently small $\delta>0$ depending on $\bar{M},\varepsilon_0$, and $N$ so that 
		\begin{align*}
			D \leq \min \left \{1, \frac{\bar{M}}{4}\right\}.
		\end{align*}
		From \eqref{mainesti1} and \eqref{Mbar}, we get 
		\begin{align*}
			\Vert h(t) \Vert_{L^p_v L^\infty_x} &\leq C_{\lambda,t_*} M_0  \exp \left \{ \frac{4CM_0}{\lambda \rho} \right \} (1+t) e^{-\frac{\lambda}{2} (1+t)^\rho} +D\\
			&\leq C_{\lambda,t_*} M_0 \exp \left \{ \frac{4CM_0}{\lambda \rho} \right \} e^{-\frac{\lambda}{4}(1+t)^\rho} + \frac{\bar{M}}{4}\\
			&\leq \frac{\bar{M}}{2},
		\end{align*}
		for $0\leq t \leq T$. In other words, whenever the a priori assumption \eqref{apriori} holds, we have 
		\begin{align*}
			\sup_{0\leq t \leq T} \Vert h(t) \Vert_{L^p_v L^\infty_x} \leq \frac{\bar{M}}{2}. 
		\end{align*}
		For an arbitrary $T>0$,  the solution to the Boltzmann equation on the interval
		$[0,T]$ can be constructed by the same argument as in Theorem \ref{smallmain}. 
		From now on, we consider the interval $[T,\infty)$. Assume that $\mathcal{E}(F_0) \le \epsilon_0 <\epsilon_1$, where $\epsilon_1$ is determined in Theorem \ref{smallmain}.
		 If we take time $T$ such that 
		\begin{align*}
			\Vert h(T) \Vert_{L^p_v L^\infty_x} \leq \eta_0,
		\end{align*}
		where the constant $\eta_0>0$ is defined in Theorem \ref{smallmain}, then from Lemma \ref{relativedecrease}, it holds that
		\begin{align*}
			\mathcal{E}(F(T))) \le \mathcal{E}(F_0) \le  \epsilon_0 <\epsilon_1,
		\end{align*}
		and we can construct the global solution to the Boltzmann equation by Theorem \ref{smallmain}.  Moreover, for $0\leq t \leq T$, we have 
		\begin{align*}
			\Vert h(t) \Vert_{L^p_v L^\infty_x} \leq \frac{\bar{M}}{2} \leq \frac{\bar{M}}{2} e^{\lambda_0 (1+T)^\rho} e^{-\lambda_0 (1+t)^ \rho}. 
		\end{align*}
		For all $t \geq T$, it follows from Theorem \ref{smallmain} that 
		\begin{align*}
			\Vert h(t) \Vert_{L^p_v L^\infty_x} \leq C_0 e^{-\lambda_0 (1+t)^\rho}  e^{\lambda_0 (1+T)^\rho}\Vert h(T)\Vert_{L^p_vL^\infty_x} \leq \frac{C_0}{2} \bar{M}e^{\lambda_0 (1+T)^\rho} e^{-\lambda_0 (1+t)^\rho} .
		\end{align*}
		Therefore, there exist constants $C_1>0$ and $\lambda_0>0$ such that   
		\begin{align*}
			\Vert h(t) \Vert_{L^p_v L^\infty_x} \leq C_1 e^{-\lambda_0 (1+t)^\rho} \Vert h_0 \Vert_{L^p_v L^\infty_x}, 
		\end{align*}
		for all $t\geq 0$. 
\end{proof}
\appendix

\section{Local Existence and Uniqueness}
\label{appendix}
\begin{lemma}\label{local estimate}
	Let $0<q<1$ and $0\le\vartheta<-\frac{2}{\gamma}$ be fixed in the weight function \eqref{weight}. Let $p$ and $\beta$. If $F_0(x,v) = \mu(v) +\sqrt{\mu(v)}f_0(x,v)\ge 0$ and $\|w_{q,\vartheta,\beta}f_0\|_{L^p_vL^\infty_x}<\infty$, then there exists a time $\hat{t}_0>0$ such that the inital value problem \eqref{Boltzmanneq} and \eqref{initialdata} has a unique non-negative solution $F(t,x,v)=\mu(v) + \sqrt{\mu(v)}f(t,x,v)$ for $t\in[0, \hat{t}_0]$, satisfying 
\begin{equation}\label{locales}
	\sup_{0\le t \le \hat{t}_0} \|w_{q,\vartheta,\beta}f(t)\|_{L^p_vL^\infty_x} \le 2\|w_{q,\vartheta,\beta}f_0\|_{L^p_vL^\infty_x}.
\end{equation}
\end{lemma}

\begin{proof}

For the local existence of non-negative solution of the Boltzmann equation \eqref{Boltzmanneq}, we consider the following iteration : 
\begin{align}\label{localitF}
\begin{cases}
\p_t F^{n+1} + v\cdot\nabla_xF^{n+1} + F^{n+1}\int_{\mathbb{R}^3}\int_{\mathbb{S}^2}B(v-u,\omega)F^n(u) d\omega du = Q^+(F^n,F^n)\\
F(0,x,v) = F_0(x,v) \ge 0 ,\quad F^0(t,x,v)=\mu(v).
\end{cases}
\end{align}
By induction on n, we can prove that all $F^n$ is non-negative for all $n\ge 0$.\\
Define $I^n(t,x,v) := \int_{\mathbb{R}^3}\int_{\mathbb{S}^2}B(v-u,\omega)F^n(u) d\omega du$. For $n=0$, by our assumption that $F_0$ is nonneagtive,
\begin{equation*}
F^1(t,x,v) = e^{-\nu(v)t}F_0(x-tv,v) + \int_0^t e^{-\nu(v)(t-s)}\nu(v)\mu(v) ds \ge 0,
\end{equation*}
because $\int_{\R^3}\int_{\mathbb{S}^2}B(u-v,w)\mu(u)dwdu = \nu(u)$. Assume that $F^n$ is nonnegative for $n=1,2,\cdots, k$.
By Duhamel's principle, we get
\begin{equation*}
F^{k+1} = e^{-\int_0^t I^k(\tau,x,v)d\tau}F_0 + \int_0^t e^{-\int_s^t I^k(\tau,x,v)d\tau}Q^+(F^k,F^k) ds \ge 0,
\end{equation*}
since $F_0 \ge 0$ and $Q^+(F^n,F^n) \ge 0$. Therefore $F^n$ is nonnegative for all $n>0$.\\
Hence we can rewrite the above iteration \eqref{localitF} for $h=w_{q,\vartheta,\beta}f$ as follow : 
\begin{align}\label{localitFh}
\begin{split}
\begin{cases}
 (\p_t + v\cdot\nabla_x + \tilde{\nu})h^{n+1}(t) = K_wh^n(t) + w\Gamma^+(f^n,f^n) - w\Gamma^-(f^n,f^{n+1})\\
h^{n+1}(0,x,v)=h_0(x,v), \quad  h^0(t,x,v)=0.
\end{cases}
\end{split}
\end{align}
We will show that there exists $\hat{t}_1>0$ such that \eqref{localitFh} has a solution over $[0,\hat{t}_1]$ satisfying
\begin{equation}\label{localitFhes}
\sup_{0\le t \le \hat{t}_1}\|h^n(t)\|_{L^p_vL^\infty_x} \le 2\|h_0\|_{L^p_vL^\infty_x},
\end{equation}
for all $n\ge 0$.
For $n=0$, $h^1(t)=e^{-\int_0^t \tilde \nu (\tau)d\tau}h_0$, which implies \eqref{localitFhes} holds for $n=0$.\\
To use induction, suppose that \eqref{localitFhes} holds for $n=1,2,\cdots,k$. 
By Duhamel's principle, 
\begin{align*}
	|h^{k+1}(t)|
&\le e^{-\int_0^t \tilde \nu (\tau)d\tau}|h_0|\\
&\quad + \int_0^t e^{-\int_s^t \tilde \nu (\tau)d\tau}\left[ |K_wh^k(s)|+|w\Gamma^+(f^k,f^k)(s)|+|w\Gamma^-(f^k,f^{k+1})(s)| \right] ds\\
& =: I_1+I_2.
\end{align*}
First of all, we can easily compute that
\begin{align} \label{locala1}
	\|I_1\|_{L^p_vL^\infty_x} \le \|h_0\|_{L^p_vL^\infty_x}.
\end{align}
Next, let us estimate $I_2$. From Lemma \ref{pointwiseGamma-estimate} and Lemma \ref{LpGamma+est}, we have
\begin{align} \label{locala2}
	\|I_2\|_{L^p_vL^\infty_x} &\le \int_0^t \left[\left\|K_w h^n\right\|_{L^p_vL^\infty_x} + \left\|w\Gamma^+(f^k,f^k)(s)\right\|_{L^p_vL^\infty_x} +\left\|w\Gamma^-(f^k,f^{k+1})(s)\right\|_{L^p_vL^\infty_x}\right]ds\nonumber\\
	& \le 	C_{p,\gamma,\beta} \hat{t}_1 \Bigg(\sup_{0\le s\le \hat{t}_1}\|h^k(s)\|_{L^p_vL^\infty_x}+\sup_{0\le s\le \hat{t}_1}\|h^k(s)\|^2_{L^p_vL^\infty_x}\nonumber\\
	& \quad +\sup_{0\le s\le \hat{t}_1}\|h^k(s)\|_{L^p_vL^\infty_x}\sup_{0\le s\le \hat{t}_1}\|h^{k+1}(s)\|_{L^p_vL^\infty_x} \Bigg) \nonumber\\
	& \le C_{p,\gamma,\beta}^{(1)}\hat{t}_1\|h_0\|_{L^p_vL^\infty_x}(1+\|h_0\|_{L^p_vL^\infty_x})+C_{p,\gamma,\beta}^{(1)}\hat{t}_1\|h_0\|_{L^p_vL^\infty_x}\sup_{0\le s\le \hat{t}_1}\|h^{k+1}(s)\|_{L^p_vL^\infty_x},
\end{align}
where $C_{p,\gamma,\beta}^{(1)}$ is a constant depending on $p,\gamma,\beta$ and we have used the assumption \eqref{localitFhes} for $n=k$ in the last inequality. Gathering \eqref{locala1} and \eqref{locala2}, we obtain
\begin{align*}
	\sup_{0\le t \le \hat t_1}\|h^{k+1}(t)\|_{L^p_vL^\infty_x} &\le \|h_0\|_{L^p_vL^\infty_x}+ C_{p,\gamma,\beta}^{(1)}\hat{t}_1\|h_0\|_{L^p_vL^\infty_x}(1+\|h_0\|_{L^p_vL^\infty_x})\\
	& \quad +C_{p,\gamma,\beta}^{(1)}\hat{t}_1\|h_0\|_{L^p_vL^\infty_x}\sup_{0\le s\le \hat{t}_1}\|h^{k+1}(s)\|_{L^p_vL^\infty_x}.
\end{align*}
Taking $\hat{t}_1\le \min{\left\{\frac{1}{3}\left\{C_{p,\gamma,\beta}^{(1)}\left(\|h_0\|_{L^p_vL^\infty_x}+1\right)\right\}^{-1}, \frac{1}{3}\left(C_{p,\gamma,\beta}^{(1)}\|h_0\|_{L^p_vL^\infty_x}\right)^{-1}\right\}}$, we can derive
\begin{align*}
\sup_{0\le s\le \hat{t}_1}\|h^{k+1}(s)\|_{L^p_vL^\infty_x} \le 2\|h_0\|_{L^p_vL^\infty_x}.
\end{align*}
By induction on $n$, \eqref{localitFhes} holds for all $n\ge 0$.
To show the convergence of $\{h^n\}$, we consider subtractions $\{h^{n+1}-h^{n}\}$. The sequence $h^{n+1}-h^{n}$ is the solution of the following equation : 
\begin{align*}
\begin{cases}
(\p_t + v\cdot\nabla_x + \tilde{\nu})(h^{n+1}-h^n)&= K_w(h^n-h^{n-1}) + w\Gamma^+(f^n,f^n)-w\Gamma^+(f^{n-1},f^{n-1})\\
&\quad - w\Gamma^-(f^n,f^{n+1}) + w\Gamma^+(f^{n-1},f^{n}),\\
(h^{n+1}-h^n)(0)=0.
\end{cases}
\end{align*}
Note that
\begin{align}\label{gammaminus}
w\Gamma^+(f_1,g_1)-w\Gamma^+(f_2,g_2)= w\Gamma^+(f_1-f_2,g_1) + w\Gamma^+(f_2,g_1-g_2).
\end{align}
The equality \eqref{gammaminus} also holds for $w\Gamma^-$. Applying Duhamel's principle and \eqref{gammaminus}, we have
\begin{align*}
&|h^{n+1}(t)-h^n(t)|\\
&\le \int_0^t e^{-\int_s^t \tilde \nu (\tau)d\tau}\Big[|K_w(h^n-h^{n-1})(s)| + |w\Gamma^+(f^n-f^{n-1},f^n)(s)| + |w\Gamma^+(f^{n-1},f^n-f^{n-1})(s)|\\
&\quad +|w\Gamma^-(f^n-f^{n-1},f^{n+1})(s)|+ |w\Gamma^-(f^{n-1},f^{n+1}-f^n)(s)|\Big]ds.
\end{align*}
Taking the norm in $L^p_vL^\infty_x$ to the above inequality and using Lemma \ref{pointwiseGamma-estimate} and Lemma \ref{LpGamma+est}, it follows that
\begin{align*}
	&\|h^{n+1}(t)-h^n(t)\|_{L^p_vL^\infty_x}\\
&\le C_{p,\gamma,\beta}\hat{t}_0\left(1+\sup_{0\le s \le \hat{t}_0}\|h^n(s)\|_{L^p_vL^\infty_x} + \sup_{0\le s \le \hat{t}_0}\|h^{n-1}(s)\|_{L^p_vL^\infty_x}+\sup_{0\le s \le \hat{t}_0}\|h^{n+1}(s)\|_{L^p_vL^\infty_x}\right)\\
& \qquad \times\sup_{0\le s \le \hat{t}_0}\|h^n(s)-h^{n-1}(s)\|_{L^p_vL^\infty_x}\\
&\quad  + C_{p,\gamma,\beta}\hat{t}_0\sup_{0\le s \le \hat{t}_0}\|h^{n-1}(s)\|_{L^p_vL^\infty_x}\sup_{0\le s \le \hat{t}_0}\|h^{n+1}(s)-h^n(s)\|_{L^p_vL^\infty_x}\\
&\le C_{p,\gamma,\beta}^{(2)}\hat{t}_0(1+\|h_0\|_{L^p_vL^\infty_x})\sup_{0\le s \le \hat{t}_0}\|h^n(s)-h^{n-1}(s)\|_{L^p_vL^\infty_x}\\
& \quad  + C_{p,\gamma,\beta}^{(2)}\hat{t}_0\|h_0\|_{L^p_vL^\infty_x}\sup_{0\le s \le \hat{t}_0}\|h^{n+1}(s)-h^n(s)\|_{L^p_vL^\infty_x},
\end{align*}
for $0\le t \le \hat t_0$, where $\hat t_0$ is determined later and $C_{p,\gamma,\beta}^{(2)}$ is a constant depending on $p,\gamma,\beta$.
Take $\hat{t}_0\le \min\left\{\hat{t}_1, \frac{1}{3}\{C_{p,\gamma,\beta}^{(2)}(\|h_0\|_{L^p_vL^\infty_x}+1)\}^{-1}, \frac{1}{3}(C_{p,\gamma,\beta}^{(2)}\|h_0\|_{L^p_vL^\infty_x})^{-1}\right\}$. Then it follows that
\begin{align*}
\sup_{0\le s \le \hat{t}_0}\|h^{n+1}(s)-h^n(s)\|_{L^p_vL^\infty_x} \le \frac{1}{2}\sup_{0\le s \le \hat{t}_0}\|h^{n}(s)-h^{n-1}(s)\|_{L^p_vL^\infty_x}.
\end{align*}
Therefore $\{h^n\}$ is a convergent sequence in $L^p_vL^\infty_x$ and we can denote $h^n \to h$, and $F^n \to F$ as $n\to \infty$. Since all $F^n$ is non-negative, $F$ is non-negative for $t\in[0,\hat{t}_0]$, and \eqref{localitFhes} implies \eqref{locales}.
For the uniqueness of the local solution, suppose that there is another solution $\tilde f$ to the Boltzmann equation with the same initial condition as $f$ satisfying
\begin{equation}\label{localges}
\sup_{0\le t \le \hat{t}_0} \|w_{q,\vartheta,\beta}\tilde f(t)\|_{L^p_vL^\infty_x} \le 2\|w_{q,\vartheta,\beta}f_0\|_{L^p_vL^\infty_x},
\end{equation}
and set $h_1:=w_{q,\vartheta,\beta}\tilde f$. Then by \eqref{localges}, we obtain for $0\le t \le \hat t_0$,
\begin{align*}
\begin{split}
&\|h_1(t)-h(t)\|_{L^p_vL^\infty_x} \\
& \le \left\|\int_0^t e^{-\int_s^t \tilde \nu (\tau)d\tau}\left[|K_{w}(h_1-h)(s)| + |w\Gamma(f-\tilde f,f)(s)| + |w\Gamma(\tilde f,f-\tilde f)(s)|\right]ds\right\|_{L^p_vL^\infty_x}\\
& \le C_{p,\gamma,\beta}\hat{t}_0\sup_{0\le s\le\hat{t}_0}(\|h_1(s)\|_{L^p_vL^\infty_x}+\|h(s)\|_{L^p_vL^\infty_x}+1)\sup_{0\le s\le\hat{t}_0}\|h_1(s)-h(s)\|_{L^p_vL^\infty_x}\\
&\le  C_{p,\gamma,\beta}^{(2)}\hat{t}_0(\|h_0\|_{L^p_vL^\infty_x}+1)\sup_{0\le s\le\hat{t_0}}\|h_1(s)-h(s)\|_{L^p_vL^\infty_x}\\
&\le \frac{1}{2}\sup_{0\le s\le\hat{t_0}}\|h_1(s)-h(s)\|_{L^p_vL^\infty_x},
\end{split}
\end{align*}
which implies that $h_1=h$ on $[0,\hat{t}_0]$.\\
\end{proof}

\noindent{\bf Data availability:} No data was used for the research described in the article.
\newline

\noindent{\bf Conflict of interest:} The authors declare that they have no conflict of interest.\newline

\noindent{\bf Acknowledgement}
G.Ko is supported by the National Natural Science Foundation of China (No. 12288201). J. Kim is supported by the National Research Foundation of Korea(NRF) grant funded by the Korea government(MSIT)(No.RS-2023-00212304 and No.RS-2023-00219980).

\bibliographystyle{abbrv}
\bibliography{Lpsoft}
\end{document}